\documentclass[11pt,a4paper]{article}
\usepackage{latexsym}
\usepackage{amsmath}
\usepackage{amssymb}
\usepackage{amscd}
\usepackage{amsthm}
\usepackage[all]{xy}

\newtheorem{thm}{Theorem}[section]
\newtheorem{cor}[thm]{Corollary}
\newtheorem{lem}[thm]{Lemma}
\newtheorem{prop}[thm]{Proposition}

\theoremstyle{definition}
\newtheorem{defn}[thm]{Definition}
\newtheorem{rem}[thm]{Remark}

\DeclareMathOperator{\Ker}{Ker}
\newcommand{\Q}{\mathbb Q}
\newcommand{\R}{\mathbb R}
\newcommand{\Z}{\mathbb Z}

\newcommand{\C}{\mathbb C}

\renewcommand{\Im}{\operatorname{Im}}

\allowdisplaybreaks[4]

\newif\ifpdf \pdftrue
\ifx\pdfoutput\undefined \pdffalse \fi \ifx\pdfoutput\relax
\pdffalse \fi

\usepackage{makeidx}
\makeindex

\begin{document}

\title{An arithmetic Lefschetz-Riemann-Roch theorem\\With an appendix by Xiaonan Ma}

\author{Shun Tang}

\date{}

\maketitle

\vspace{-10mm}

\hspace{5cm}\hrulefill\hspace{5.5cm} \vspace{5mm}

\textbf{Abstract.} In this article, we consider regular projective arithmetic schemes in the context of Arakelov geometry, any of which is endowed with an action of the diagonalisable group scheme associated to a finite cyclic group and with an equivariant very ample invertible sheaf. For any equivariant morphism between such arithmetic schemes, which is smooth over the generic fibre, we define a direct image map between corresponding higher equivariant arithmetic K-groups and we discuss its transitivity property. Then we use the localization sequence of higher arithmetic K-groups and the higher arithmetic concentration theorem developed in \cite{T3} to prove an arithmetic Lefschetz-Riemann-Roch theorem. This theorem can be viewed as a generalization, to the higher equivariant arithmetic K-theory, of the fixed point formula of Lefschetz type proved by K. K\"{o}hler and D. Roessler in \cite{KR1}.


\tableofcontents

\section{Introduction}
The aim of this article is to prove an arithmetic Riemann-Roch theorem of Lefschetz type for the higher equivariant arithmetic K-theory of regular arithmetic schemes in the context of Arakelov geometry. This theorem is an arithmetic analogue of a special case of K\"{o}ck's Lefschetz theorem in higher equivariant K-theory (cf. \cite{Ko1}), and it also generalizes K\"{o}hler-Roessler's Lefschetz fixed point formula \cite[Theorem 4.4]{KR1} to the case where higher arithmetic K-groups are concerned. To make things more explicit, let us first recall the study of such Lefschetz-Riemann-Roch problems.

Let $X$ be a smooth projective variety over an algebraically closed field $k$, and suppose that $X$ is endowed with an action of a cyclic group $\langle g\rangle$ of finite order $n$ such that $n$ is prime to the characteristic of $k$. A $\langle g\rangle$-equivariant coherent sheaf on $X$ is a coherent $\mathcal{O}_X$-module $F$ on $X$ together with an automorphism $\varphi: g^*F\to F$ such that $\varphi^n$ is equal to the identity map. Then the classical Lefschetz trace formula gives an expression of the alternating sum of the trace of $H^i(\varphi)$ on the cohomology space $H^i(X,F)$, as a sum of the contributions from the components of the fixed point subvariety $X_g$.
For $k=\mathbb{C}$, the field of complex numbers, such a Lefschetz trace formula was presented via index theory and topological K-theory in \cite[III]{ASe}. While for general $k$, a Grothendieck type generalization to the scheme theoretic algebraic geometry is very natural to expect. Precisely, denote by $K_0(X,g)$ the Grothendieck group of the category of equivariant locally free coherent sheaves on $X$, then $K_0({\rm Pt},g)$ is isomorphic to the group ring $\mathbb{Z}[g]\cong \mathbb{Z}[T]/{(1-T^n)}$ and $K_0(X,g)$ has a natural $K_0({\rm Pt},g)$-algebra structure (${\rm Pt}$ stands for the point ${\rm Spec(k)}$). Let $Y$ be another $\langle g\rangle$-equivariant smooth projective variety, let $f: X\to Y$ be a projective morphism compatible with both $\langle g\rangle$-actions on $X$ and on $Y$, then we have a direct image map $f_*: K_0(X,g)\to K_0(Y,g)$ given by
$$E\mapsto \sum_{i\geq 0}(-1)^iR^if_*(E).$$
Unsurprisingly, the direct image map $f_*$ doesn't commute with the restriction map $\tau: K_0(\cdot,g)\to K_0\big((\cdot)_g,g\big)$ from the equivariant $K_0$-group of an equivariant variety to the equivariant $K_0$-group of its fixed point subvariety. Namely, the restriction map $\tau$ is not a natural transformation between the covariant functors $K_0(\cdot,g)$ and $K_0\big((\cdot)_g,g\big)$. Like the other Riemann-Roch problems, the Lefschetz-Riemann-Roch theorem makes a correction of $\tau$ such that it becomes a natural transformation. In fact, for any $\langle g\rangle$-equivariant smooth projective variety $X$, let $N_{X/{X_g}}$ stand for the normal bundle associated to the regular immersion $X_g\hookrightarrow X$ and let $\lambda_{-1}(N^\vee_{X/{X_g}})$ be the alternating sum $\sum(-1)^j\wedge^jN^\vee_{X/{X_g}}$, then $\lambda_{-1}(N^\vee_{X/{X_g}})$ is an invertible element in $K_0(X_g,g)\otimes_{\mathbb{Z}[g]}\mathcal{R}$ where $\mathcal{R}$ is any $\mathbb{Z}[g]$-algebra in which $1-T^k$ is invertible for $k=1, \ldots, n-1$. We formally define $L_X: K_0(X,g)\to K_0(X_g,g)\otimes_{\mathbb{Z}[g]}\mathcal{R}$ as $\lambda^{-1}_{-1}(N^\vee_{X/{X_g}})\cdot\tau$, the Lefschetz-Riemann-Roch theorem reads: the following diagram
\begin{align}\label{c1}
\xymatrix{ K_0(X,g) \ar[r]^-{L_X} \ar[d]^-{f_*} & K_0(X_g,g)\otimes_{\mathbb{Z}[g]}\mathcal{R} \ar[d]^-{{f_g}_*} \\
K_0(Y,g) \ar[r]^-{L_Y} & K_0(Y_g,g)\otimes_{\mathbb{Z}[g]}\mathcal{R}}
\end{align}
is commutative.

This commutative diagram (\ref{c1}) was presented by P. Donovan in \cite{Do}, and later it was generalized to singular varieties by P. Baum, W. Fulton and G. Quart in \cite{BFQ}. Notice that the settings in \cite{Do} and in \cite{BFQ} are more general than that in this introduction. The reasoning in the first paper runs similarly to the technique used in Borel-Serre's paper \cite{BS}, while the reasoning in the second paper relies on the deformation to the normal cone construction. These two processes are both traditional for producing the Grothendieck type Riemann-Roch theorem.

After Quillen and other mathematicians' work, algebraic K-groups are extended to higher degrees and the higher (equivariant) algebraic K-groups of $X$ are defined as the higher homotopy groups of the K-theory space associated to the category of (equivariant) locally free coherent sheaves on $X$. There are many methods to construct this ``K-theory space", but no matter which construction we choose, the tensor product of locally free coherent sheaves always induces a graded ring structure on $K_\bullet(X,g)$. In particular, each $K_m(X,g)$ is a $K_0(X,g)$-module. Moreover, the functor $K_\bullet(\cdot,g)$ is again covariant with respect to equivariant proper morphisms. Then, for any $m\geq 1$, the following diagram for higher algebraic K-groups which is similar to (\ref{c1}) does make sense:
\begin{align}\label{c2}
\xymatrix{ K_m(X,g) \ar[r]^-{L_X} \ar[d]^-{f_*} & K_m(X_g,g)\otimes_{\mathbb{Z}[g]}\mathcal{R} \ar[d]^-{{f_g}_*} \\
K_m(Y,g) \ar[r]^-{L_Y} & K_m(Y_g,g)\otimes_{\mathbb{Z}[g]}\mathcal{R}.}
\end{align}

The commutativity of diagram (\ref{c2}), which is named the Lefschetz-Riemann-Roch theorem for higher equivariant algebraic K-theory, was proved by B. K\"{o}ck in \cite{Ko1}. The main ingredient is an excess intersection formula whose proof also relies on the deformation to the normal cone construction. Moreover, it's worth indicating that the commutative diagram (\ref{c2}), combined with the Gillet's Riemann-Roch theorem for higher algebraic K-theory (cf. \cite{Gi}), implies a higher Lefschetz trace formula.

In the field of arithmetic geometry, one considers those noetherian and separated schemes $f: X\to {\rm Spec}(\mathbb{Z})$ over the ring of integers (actually over any excellent regular noetherian domain). In this context, it is possible to produce an analogue of the commutative diagram (\ref{c1}), by endowing $X$ with an action of the diagonalisable group scheme $\mu_n={\rm Spec}(\mathbb{Z}[\mathbb{Z}/{n\mathbb{Z}}])$ of $n$-th roots of unity rather than with the action of an automorphism of order $n$. Here, a $\mu_n$-action on $X$ is a morphism of schemes $m_X: \mu_n\times X\to X$ which satisfies the usual associativity property. The reason for this choice is that the fixed point subscheme $X_{\mu_n}$ of a regular scheme $X$ equipped with an action of $\mu_n$ is still regular and the natural inclusion $i_X: X_{\mu_n}\hookrightarrow X$ is a regular immersion, while the fixed point subscheme of a regular scheme under an automorphism of order $n$ can be very singular over the fibres lying above the primes dividing $n$. By a $\mu_n$-equivariant coherent sheaf $F$ on $X$, we understand a coherent $\mathcal{O}_X$-module $F$ together with an isomorphism 
$$m_F: m_X^*F\to {\rm pr}_X^*F$$
of $\mathcal{O}_{\mu_n\times X}$-modules which satisfies the following associativity property:
$$({\rm pr}_{2,3}^*m_F)\circ\big((1\times m_X)^*m_F\big)=(m_{\mu_n}\times 1)^*m_F.$$
Here, $m_{\mu_n}$ denotes the multiplication $\mu_n\times \mu_n\to \mu_n$, ${\rm pr}_X: \mu_n\times X\to X$ and ${\rm pr}_{2,3}: \mu_n\times \mu_n\times X\to \mu_n\times X$ denote the obvious projections. Under this situation, Baum-Fulton-Quart's method still works, so that the commutative diagram (\ref{c1}) holds for regular $\mu_n$-equivariant schemes over $\mathbb{Z}$.

In \cite{Th}, R. W. Thomason used another way to do the same thing and he even got a generalization of the commutative diagram (\ref{c2}) for regular $\mu_n$-equivariant schemes. Thomason's strategy was to use Quillen's localization sequence for higher equivariant algebraic K-groups to show a concentration theorem. This theorem states that, after a suitable localization, the equivariant algebraic K-group $K_m(X_{\mu_n},\mu_n)_\rho$ is isomorphic to $K_m(X,\mu_n)_\rho$ for any
$m\geq 0$, and the inverse map is exactly given by $\lambda^{-1}_{-1}(N^\vee_{X/{X_{\mu_n}}})\cdot i_X^*$. Here, $\rho$ is any prime ideal in $R(\mu_n):=K_0({\rm Spec}\mathbb{Z},\mu_n)\cong \Z[T]/{(1-T^n)}$ which doesn't contain the elements $1-T^k$ for $k=1,\ldots,n-1$. For instance, $\rho$ can be chosen to be the kernel of the natural morphism $\Z[T]/{(1-T^n)}\to \Z[T]/{(\Phi_n)}$ where $\Phi_n$ stands for the $n$-th cyclotomic polynomial. Then the Lefschetz-Riemann-Roch theorem for regular $\mu_n$-equivariant schemes
\begin{align}\label{c3}
\xymatrix{ K_m(X,\mu_n) \ar[r]^-{L_X} \ar[d]^-{f_*} & K_m(X_{\mu_n},\mu_n)_\rho \ar[d]^-{{f_{\mu_n}}_*} \\
K_m(Y,\mu_n) \ar[r]^-{L_Y} & K_m(Y_{\mu_n},\mu_n)_\rho}
\end{align}
follows from the covariant functoriality of $K_\bullet(\cdot,\mu_n)$ with respect to proper morphisms.

Now, let us turn to Arakelov geometry. Let $X$ be an arithmetic scheme over an arithmetic ring $(D, \Sigma, F_\infty)$ in the sense of Gillet-Soul\'{e} (cf. \cite{GS1}), then $X$ is quasi-projective over $D$ with smooth generic fibre. We denote $\mu_n:={\rm Spec}(D[\Z/{n\Z}])$ the diagonalisable group scheme over $D$ associated to a cyclic group $\Z/{n\Z}$. By saying $X$ is $\mu_n$-projective, we understand that $X$ is endowed with a projective $\mu_n$-action. That means $X$ is projective and there exists a very ample invertible $\mu_n$-sheaf on $X$. 

For each regular $\mu_n$-projective arithmetic scheme $X$, K. K\"{o}hler and D. Roessler have defined an equivariant arithmetic K$_0$-group
$\widehat{K_0}(X,\mu_n)$ in \cite{KR1}. This arithmetic $\text{K}_0$-group is a modified Grothendieck group of the category of equivariant hermitian vector bundles on $X$, it contains some smooth form class on $X_{\mu_n}(\C)$ as analytic datum. The same as the algebraic $\text{K}_0$-group $K_0(X,\mu_n)$, $\widehat{K_0}(X,\mu_n)$ has a ring structure and it is an $R(\mu_n)$-algebra. Moreover, direct image maps between equivariant arithmetic K$_0$-groups can be defined for an equivariant morphism which is smooth over the generic fibre, by using Bismut-K\"{o}hler-Ma's analytic torsion forms. Choose a K\"{a}hler metric for $X(\C)$, and let $\overline{N}_{X/{X_{\mu_n}}}$ be the normal bundle endowed with the quotient metric, then the main theorem in \cite{KR1} reads: the element
$\lambda_{-1}(\overline{N}_{X/{X_{\mu_n}}}^\vee)$ is a unit in
$\widehat{K_0}(X_{\mu_n},\mu_n)_\rho$ and the following diagram
\begin{align}\label{c4}
\xymatrix{
\widehat{K_0}(X,\mu_n) \ar[rr]^-{\Lambda_R\cdot \tau} \ar[d]_{f_*} && \widehat{K_0}(X_{\mu_n},\mu_n)_\rho \ar[d]^{{f_{\mu_n}}_*} \\
 \widehat{K_0}(D,\mu_n) \ar[rr]^-{\tau} && \widehat{K_0}(D,\mu_n)_\rho}
\end{align}
is commutative, where $\rho$ is any prime ideal in $R(\mu_n)$ which doesn't contain the elements $1-T^k$ for $k=1,\ldots,n-1$, $\Lambda_R$ is defined as $\big(1-R_g(N_{X/{X_{\mu_n}}})\big)\cdot\lambda_{-1}^{-1}(\overline{N}_{X/{X_{\mu_n}}}^\vee)$, and $R_g(\cdot)$ is the equivariant $R$-genus due to Bismut (see below).

Later, two refinements of (\ref{c4}) were presented by the author in \cite{T1} and in \cite{T2} respectively. In \cite{T1}, $D$ was replaced by a general regular $\mu_n$-projective scheme $Y$. In \cite{T2}, $X$ was allowed to have singularities on its finite fibres. The aim of this article is to show an arakelovian analogue of a special case of (\ref{c3}), in which the higher equivariant algebraic K-groups are replaced by the higher equivariant arithmetic K-groups. Hence, our work is a generalization of K\"{o}hler-Roessler's Lefschetz fixed point formula to the higher equivariant arithmetic K-theory.

Let us describe the main result more precisely. Firstly, notice that we have constructed a group endomorphism
$\otimes\lambda_{-1}(\overline{N}_{X/X_{\mu_n}}^\vee): \widehat{K}_m(X_{\mu_n},\mu_n)\to \widehat{K}_m(X_{\mu_n},\mu_n)$ and its formal inverse $\otimes\lambda^{-1}_{-1}(\overline{N}_{X/X_{\mu_n}}^\vee): \widehat{K}_m(X_{\mu_n},\mu_n)_\rho\to \widehat{K}_m(X_{\mu_n},\mu_n)_\rho$ in \cite[Section 5]{T3}. As what we stated before, $\rho$ is any prime ideal in $R(\mu_n):=K_0({\rm Spec}\mathbb{Z},\mu_n)\cong \Z[T]/{(1-T^n)}$ which doesn't contain the elements $1-T^k$ for $k=1,\ldots,n-1$. For instance, $\rho$ can be chosen to be the kernel of the natural morphism $\Z[T]/{(1-T^n)}\to \Z[T]/{(\Phi_n)}$ where $\Phi_n$ stands for the $n$-th cyclotomic polynomial. In this article, we shall further construct a group endomorphism $R_g(\overline{N}_{X/X_{\mu_n}}): \widehat{K}_m(X_{\mu_n},\mu_n)\to \widehat{K}_m(X_{\mu_n},\mu_n)$ and we shall prove that this endomorphism $R_g(\overline{N}_{X/X_{\mu_n}})$ is independent of the choice of the metric over $N_{X/X_{\mu_n}}$ after tensoring by $\Q$. So the expression $\Lambda_R=\big(1-R_g(N_{X/{X_{\mu_n}}})\big)\cdot\lambda_{-1}^{-1}(\overline{N}_{X/{X_{\mu_n}}}^\vee)$ still makes sense as an endomorphism of $\widehat{K}_m(X_{\mu_n},\mu_n)_\rho\otimes\Q$. Moreover, for any equivariant morphism $f: X\to Y$ between regular $\mu_n$-projective arithmetic schemes, which is smooth over the generic fibre, we shall prove that there exists a reasonable direct image map $f_*: \widehat{K}_m(X,\mu_n)\to \widehat{K}_m(Y,\mu_n)$ with $m\geq1$ and we discuss the transitivity property of the direct image maps up to torsion. Assume that the $\mu_n$-action on $Y$ is trivial and still use the notation $\tau$ to denote the morphism 
\begin{align*}
\widehat{K}_m\big((\cdot),\mu_n\big)&\to \widehat{K}_m\big((\cdot)_{\mu_n},\mu_n\big)_\rho\otimes\Q\\
x&\mapsto \tau(x)\otimes 1,
\end{align*}
Our main theorem reads: the following diagram
\begin{align}\label{c5}
\xymatrix{
\widehat{K}_m(X,\mu_n) \ar[rr]^-{\Lambda_R\cdot \tau} \ar[d]_{f_*} && \widehat{K}_m(X_{\mu_n},\mu_n)_\rho\otimes \Q \ar[d]^{{f_{\mu_n}}_*} \\
\widehat{K}_m(Y,\mu_n) \ar[rr]^-{\tau} && \widehat{K}_m(Y,\mu_n)_\rho\otimes\Q}
\end{align}
is commutative. In such a formulation, the equivariant $R$-genus again plays a crucial role.

To this aim, the definition of higher equivariant arithmetic K-groups and some reasonable technique that can be carried out for higher equivariant arithmetic K-theory should be clarified. We have settled these in \cite{T3}. In fact, we have defined the higher equivariant arithmetic K-groups via the simplicial description of the Beilinson's regulators (cf. \cite{BW}) and we have developed a localization sequence as well as an arithmetic concentration theorem. So, principally, we shall follow Thomason's approach to prove the commutativity of (\ref{c5}), but the fact that the direct image maps are only defined for the morphisms which are smooth over the generic fibres will lead to a big gap comparing with the purely algebraic case. Some highly non-trivial analytic machinery should be involved, such as the transitivity property of analytic torsion forms and the Bismut-Ma's immersion formula.

The K\"{o}hler-Roessler's arithmetic Lefschetz fixed point formula has fruitful applications in number theory and in arithmetic geometry. One important reason is that the equivariant $R$-genus is closely related to the logarithmic derivative of certain $L$-functions. K\"{o}hler-Roessler and Maillot-Roessler have shown in \cite{KR2} and in \cite{MR1} that the Faltings heights and the periods of C.M. abelian varieties can be expressed as a formula in terms of the special value of logarithmic derivative of $L$-functions at $0$. Further, in \cite{MR2}, Maillot-Roessler presented a series of conjectures about the relation between several invariants of arithmetic varieties and the special values of logarithmic derivative of Artin $L$-functions at negative integers. We hope that our Lefschetz-Riemann-Roch theorem for higher equivariant arithmetic K-groups would be helpful to understand these conjectures.

The structure of this article is as follows. In Section 2, we define the direct image maps between higher equivariant arithmetic K-groups. As an opportunity, we recall the analytic torsion for cubes of hermitian vector bundles introduced by D. Roessler in \cite{Roe}, actually our construction is slightly different to Roessler's construction. In Section 3, we discuss certain transitivity property of the direct image maps, the relation of equivariant analytic torsion forms with respect to families of submersions will be presented. In Section 4, we formulate and prove the commutativity of the diagram (\ref{c5}), an accurate computation via the deformation to the normal cone construction is given. In the last section, Section 5, we attach an appendix on some properties of equivariant analytic torsion forms and immersion formula. These purely analytic properties are crucial for the main arguments in this article, the author is very grateful to Prof. Ma Xiaonan for writing this appendix.

\section{Higher equivariant arithmetic K-theory}

\subsection{Bott-Chern forms and arithmetic K-groups}
Suggested by Soul\'{e} (cf. \cite{So}), and also by Deligne (cf. \cite{De}), the higher arithmetic K-groups of an arithmetic scheme $X$ can be defined as the homotopy groups of the homotopy fibre of Beilinson's regulator map so that one obtains a long exact sequence
$$\xymatrix{\cdots \ar[r] & \widehat{K}_m(X) \ar[r] & K_m(X) \ar[r]^-{{\rm ch}} & \bigoplus_{p\geq 0}H_{\mathcal{D}}^{2p-m}\big(X,\R(p)\big) \ar[r] & \widehat{K}_{m-1}(X) \ar[r] & \cdots,}$$
where $H_{\mathcal{D}}^{*}\big(X,\R(p)\big)$ is the real Deligne-Beilinson cohomology and ${\rm ch}$ is the Beilinson's regulator map. In order to do this, a simplicial description of Beilinson's regulator map is necessary. In \cite{BW}, such a simplicial description was given by Burgos and Wang by using the higher Bott-Chern forms. Recently, in \cite{T3}, we followed Burgos-Wang's approach to define the higher equivariant Bott-Chern forms and further the higher equivariant arithmetic K-theory. In this subsection, we shall recall some relevant constructions and definitions, for more details the reader is referred to \cite{BW} and \cite{T3}.

At first, let $X$ be a smooth algebraic variety over $\mathbb{C}$. In this subsection, we shall work with the analytic topology of $X$. Denote by $E_{\log}^*(X)$ the complex of differential forms on $X$ with logarithmic singularities along infinity (cf. \cite[Definition 2.1]{T3}), then $E_{\log}^*(X)$ has a natural bigrading $E_{\log}^n(X)=\bigoplus_{\stackrel{p+q=n}{}}E_{\log}^{p,q}(X)$ and this grading induces a Hodge filtration $F^pE_{\log}^n(X)=\bigoplus_{\stackrel{p'\geq p}{p'+q'=n}}E_{\log}^{p',q'}(X)$. Write $E_{\log,\R}^*(X,p):=(2\pi i)^pE_{\log,\R}^*(X)$ with $E_{\log,\R}^*(X)$ the subcomplex of $E_{\log}^*(X)$ consisting of real forms, then we have a decomposition $E_{\log}^*(X)=E_{\log,\R}^*(X,p)\oplus E_{\log,\R}^*(X,p-1)$ and the projection $\pi_p: E_{\log}^*(X)\to E_{\log,\R}^*(X,p)$ is given by $\pi_p(x)=\frac{1}{2}\big(x+(-1)^p\overline{x}\big)$. Moreover, for any $x\in E_{\log}^n(X)$, we define two filtered functions $$F^{k,k}x=\sum_{\stackrel{l\geq k,l'\geq k}{}}x^{l,l'}\quad\text{and}\quad F^{k}x=\sum_{\stackrel{l\geq k}{}}x^{l,l'}.$$ Then we set $\pi(x):=\pi_{p-1}(F^{n-p+1,n-p+1}x)$.

The main result in \cite[Section 2]{Bu1} states that the following Deligne complex
$$\mathfrak{D}^n\big(E_{\log}(X),p\big)=\left\{
      \begin{array}{ll}
        E_{\log,\R}^{n-1}(X,p-1)\bigcap\bigoplus_{\stackrel{p'+q'=n-1}{p'<p,q'<p}}E_{\log}^{p',q'}(X), & n<2p; \\
        E_{\log,\R}^{n}(X,p)\bigcap\bigoplus_{\stackrel{p'+q'=n}{p'\geq p,q'\geq p}}E_{\log}^{p',q'}(X), & n\geq 2p,
      \end{array}
    \right.$$
with differential
$$d_\mathcal{D}x=\left\{
                   \begin{array}{ll}
                     -\pi(dx), & n<2p-1; \\
                     -2\partial\overline{\partial}x, & n=2p-1; \\
                     dx, & n>2p-1.
                   \end{array}
                 \right.$$
computes the real Deligne-Beilinson cohomology of $X$. Namely, one has $$H_{\mathcal{D}}^n\big(X,\R(p)\big)=H^n\Big(\mathfrak{D}^*\big(E_{\log}(X),p\big)\Big).$$ We shall write $D^*(X,p):=\mathfrak{D}^*\big(E_{\log}(X),p\big)$ for short.

\begin{rem}\label{explain_d}
(i). According to the definition, the real Deligne-Beilinson cohomology of $X$ at degrees $2p$ and $2p-1$ are given by
$$H^{2p}\Big(\mathfrak{D}^*\big(E_{\log}(X),p\big)\Big)=\{x\in E_{\log}^{p,p}(X)\cap E_{\log,\R}^{2p}(X,p)\mid dx=0\}/{\Im(\partial\overline{\partial})}$$
and
$$H^{2p-1}\Big(\mathfrak{D}^*\big(E_{\log}(X),p\big)\Big)=\{x\in E_{\log}^{p-1,p-1}(X)\cap E_{\log,\R}^{2p-2}(X,p-1)\mid \partial\overline{\partial}x=0\}/{(\Im\partial+\Im\overline{\partial})}.$$

(ii). Let $x\in D^n(X,p)$ and $y\in D^m(X,q)$, we write $l=n+m$ and $r=p+q$. Then
$$x\bullet y=\left\{
               \begin{array}{ll}
                 (-1)^nr_p(x)\wedge y+x\wedge r_q(y), & n<2p, m<2q; \\
                 \pi(x\wedge y), & n<2p, m\geq 2q, l<2r; \\
                 F^{r,r}\big(r_p(x)\wedge y\big)+2\pi_r\partial\big((x\wedge y)^{r-1,l-r}\big), & n<2p, m\geq 2q, l\geq 2r; \\
                 x\wedge y, & n\geq 2p, m\geq 2q.
               \end{array}
             \right.$$
induces a product on $\bigoplus_p D^*(X,p)$ which is graded commutative and is associative up to chain homotopy. Here $r_px=2\pi_p(F^p dx)$ if $n\leq 2p-1$ and $r_px=x$ otherwise. At the level of cohomology groups, this product coincides with the product defined by Beilinson. Notice that if $x\in D^{2p}(X,p)$ is a cocycle, then for all $y,z$ we have $x\bullet y=y\bullet x$ and $y\bullet(x\bullet z)=(y\bullet x)\bullet z=x\bullet(y\bullet z)$.
\end{rem}

In order to introduce the higher Bott-Chern form, let us construct a new complex $\widetilde{D}^*(X,p)$ using the cocubical structure of the cartesian product of projective lines $(\mathbb{P}^1)^\cdot$. This complex $\widetilde{D}^*(X,p)$ has the same cohomology groups as $D^*(X,p)$. Firstly one notices that $D^*(X\times(\mathbb{P}^1)^\cdot,p)$ form a cubical complex with face and degeneracy maps
$$d_i^j=({\rm Id}\times d_j^i)^*\quad\text{and}\quad s_i=({\rm Id}\times s^i)^*,$$
where $$d_j^i: (\mathbb{P}^1)^k\to (\mathbb{P}^1)^{k+1},\quad i=1,\cdots,k,j=0,1,$$
$$s^i: (\mathbb{P}^1)^k\to (\mathbb{P}^1)^{k-1},\quad i=1,\cdots,k,$$
which are given by
$$d_0^i(x_1,\cdots,x_k)=(x_1,\cdots,x_{i-1},(0:1),x_i,\cdots,x_k),$$
$$d_1^i(x_1,\cdots,x_k)=(x_1,\cdots,x_{i-1},(1:0),x_i,\cdots,x_k),$$
$$s^i(x_1,\cdots,x_k)=(x_1,\cdots,x_{i-1},x_{i+1},\cdots,x_k)$$
are the coface and the codegeneracy maps of $(\mathbb{P}^1)^\cdot$. Then we write $D_{\mathbb{P}}^{r,k}(X,p)=D^r(X\times (\mathbb{P}^1)^{-k},p)$ and denote by $D_{\mathbb{P}}^{*,*}(X,p)$ the associated double complex with differentials
$$d'=d_{\mathcal{D}}\quad\text{and}\quad d''=\sum(-1)^{i+j-1}d_i^j.$$
Next, let $(x:y)$ be the homogeneous coordinates of $\mathbb{P}^1$, and let $\omega=\partial\overline{\partial}\log\frac{x\overline{x}+y\overline{y}}{x\overline{x}}\in (2\pi i)E_{\mathbb{P}^1,\R}^2$ be a K\"{a}hler form over $\mathbb{P}^1$. We shall write $\omega_i=p_i^*\omega\in E_{\log}^*(X\times(\mathbb{P}^1)^k)$ where $p_i: X\times(\mathbb{P}^1)^k\to \mathbb{P}^1, i=1,\cdots,k$ is the projection over the $i$-th projective line. The complex $\widetilde{D}^*(X,p)$ is constructed by killing the degenerate classes and the classes coming from the projective spaces.

\begin{defn}\label{201}
We define $\widetilde{D}^*(X,p)$ as the associated simple complex of the double complex $\widetilde{D}^{*,*}(X,p)$ which is given by $$\widetilde{D}^{r,k}(X,p)=D_{\mathbb{P}}^{r,k}(X,p)/{\sum_{i=1}^{-k}s_i\big(D_{\mathbb{P}}^{r,k+1}(X,p)\big)\oplus
\omega_i\wedge s_i\big(D_{\mathbb{P}}^{r-2,k+1}(X,p-1)\big)}.$$
The differential of this complex will be denoted by $d$.
\end{defn}

A repetition of the proofs of \cite[Proposition 1.2 and Lemma 1.3]{BW} gives that the natural morphism of complexes $$\iota: D^*(X,p)=\widetilde{D}^{*,0}(X,p)\to \widetilde{D}^*(X,p)$$ is a quasi-isomorphism.

Now, let $X$ be a smooth $\mu_n$-projective variety over $\mathbb{C}$ and denote by $\mathcal{U}:=\widehat{\mathcal{P}}(X,\mu_n)$ the exact category of $\mu_n$-equivariant vector bundles on $X$ equipped with $\mu_n$-invariant smooth hermitian metrics. We consider the exact cubes in the category $\mathcal{U}$. By definition, an exact $k$-cube in $\mathcal{U}$ is a functor $\mathcal{F}$ from $\langle-1,0,1\rangle^k$, the $k$-th power of the ordered set $\langle-1,0,1\rangle$, to $\mathcal{U}$ such that for any $\alpha\in \langle-1,0,1\rangle^{k-1}$ and $1\leq i \leq k$, the $1$-cube $\partial_i^\alpha$ defined by $$\mathcal{F}_{\alpha_1,\cdots,\alpha_{i-1},-1,\alpha_i,\cdots,\alpha_{k-1}}\to \mathcal{F}_{\alpha_1,\cdots,\alpha_{i-1},0,\alpha_i,\cdots,\alpha_{k-1}}\to \mathcal{F}_{\alpha_1,\cdots,\alpha_{i-1},1,\alpha_i,\cdots,\alpha_{k-1}}$$ which is called an edge of $\mathcal{F}$ is a short exact sequence. From now on, we shall write cubes instead of exact cubes for short. Let $\mathcal{F}$ be a $k$-cube in $\mathcal{U}$, for $1\leq i\leq k$ and $j\in \langle-1,0,1\rangle$, the $(k-1)$-cube $\partial_i^j\mathcal{F}$ defined by $(\partial_i^j\mathcal{F})_{\alpha_1,\cdots,\alpha_{k-1}}=\mathcal{F}_{\alpha_1,\cdots,\alpha_{i-1},j,\alpha_i,\cdots,\alpha_{k-1}}$ is called a face of $\mathcal{F}$. On the other hand, for any $1\leq i\leq k+1$, we denote by $S_i^1\mathcal{F}$ the $(k+1)$-cube
$$(S_i^1\mathcal{F})_{\alpha_1,\cdots,\alpha_{k+1}}=\left\{
                                                      \begin{array}{ll}
                                                        0, & \alpha_i=1; \\
                                                        \mathcal{F}_{\alpha_1,\cdots,\alpha_{i-1},\alpha_{i+1},\cdots,\alpha_{k+1}}, & \alpha_i\neq1,
                                                      \end{array}
                                                    \right.$$
such that the morphisms $(S_i^1\mathcal{F})_{\alpha_1,\cdots,\alpha_{i-1},-1,\alpha_{i+1},\cdots,\alpha_{k+1}}\to (S_i^1\mathcal{F})_{\alpha_1,\cdots,\alpha_{i-1},0,\alpha_{i+1},\cdots,\alpha_{k+1}}$ are the identities of $(S_i^1\mathcal{F})_{\alpha_1,\cdots,\alpha_{i-1},\alpha_{i+1},\cdots,\alpha_{k+1}}$. Similarly, we have $(k+1)$-cube $S_i^{-1}\mathcal{F}$.

Denote by $C_k\mathcal{U}$ the set of all $k$-cubes in $\mathcal{U}$, then we have the face maps $\partial_i^j: C_k\mathcal{U}\to C_{k-1}\mathcal{U}$ and the degeneracy maps $S_i^j: C_k\mathcal{U}\to C_{k+1}\mathcal{U}$. The cubes in the image of $S_i^j$ are said to be degenerate. Let $\Z C_k\mathcal{U}$ be the free abelian group generated by $C_k\mathcal{U}$ and $D_k$ be the subgroup of $\Z C_k{\mathcal{U}}$ generated by all degenerate $k$-cubes. Set $\widetilde{\Z}C_k\mathcal{U}=\Z C_k\mathcal{U}/{D_k}$ and
$$d=\sum_{i=1}^k\sum_{j=-1}^1(-1)^{i+j-1}\partial_i^j: \widetilde{\Z}C_k\mathcal{U}\to \widetilde{\Z}C_{k-1}\mathcal{U}.$$
Then $\widetilde{\Z}C_*\mathcal{U}=(\widetilde{\Z}C_k\mathcal{U},d)$ is a homological complex.

Assume that $\overline{E}$ is a hermitian $k$-cube in the category $\mathcal{U}=\widehat{\mathcal{P}}(X,\mu_n)$. If $\overline{E}$ is an emi-cube, namely the metrics on the quotient terms in all edges of $\overline{E}$ are induced by the metrics on the middle terms (cf. \cite[Definition 3.5]{BW}), one can follow \cite[(3.7)]{BW} to associate a hermitian locally free sheaf ${\rm tr}_k(\overline{E})$ on $X\times (\mathbb{P}^1)^k$. This ${\rm tr}_k(\overline{E})$ is called the $k$-transgression bundle of $\overline{E}$. If $k=1$, as an emi-$1$-cube, $\overline{E}$ is a short exact sequence
$$\xymatrix{ 0\ar[r] & \overline{E}_{-1}\ar[r]^-{i} & \overline{E}_0\ar[r] & \overline{E}_1 \ar[r] & 0},$$
where the metric of $\overline{E}_1$ is induced by the metric of $\overline{E}_0$. Then ${\rm tr}_1(\overline{E})$ is the cokernel with quotient metric of the map $E_{-1}\to E_{-1}\otimes\mathcal{O}(1)\oplus E_0\otimes\mathcal{O}(1)$ by the rule
$e_{-1}\mapsto e_{-1}\otimes \sigma_{\infty}\oplus i(e_{-1})\otimes \sigma_0$.
Here $\sigma_0$ (resp. $\sigma_{\infty}$) is the section of the tautological bundle $\mathcal{O}(1)$ on $\mathbb{P}^1$ which vanishes only at $0$ (resp. $\infty$), and $\mathcal{O}(1)$ is endowed with the Fubini-Study metric. If $k>1$, suppose that the transgression bundle is defined for $k-1$. Let ${\rm tr}_1(\overline{E})$ be the emi-$(k-1)$-cube over $X\times \mathbb{P}^1$ given by ${\rm tr}_1(\overline{E})_\alpha={\rm tr}_1\big(\partial^\alpha_1(\overline{E})\big)$, then ${\rm tr}_k(\overline{E})$ is defined as ${\rm tr}_{k-1}\big({\rm tr}_1(\overline{E})\big)$.

Moreover, according to \cite[Proposition 3.6]{BW}, for any hermitian cube $\overline{E}$ in the category $\mathcal{U}$, there is a unique way to change the metrics on $E_\alpha$ for $\alpha\nleq0$ such that the obtained new hermitian cube is emi. In fact, for $i=1,\ldots,k$, define $\lambda_i^1\overline{E}$ to be
$$(\lambda_i^1\overline{E})_\alpha=
\left\{
\begin{array}{ll}
(E_\alpha,h_\alpha), & \text{if } \alpha_i=-1,0; \\
(E_\alpha,h'_\alpha), & \text{if } \alpha_i=1,
\end{array}
\right.$$
where $h'_\alpha$ is the metric induced by $h_{\alpha_1,\ldots,\alpha_{i-1},0,\alpha_{i+1},\ldots,\alpha_k}$. Thus $\lambda_i^1\overline{E}$ has the same locally free sheaves as $\overline{E}$, but the metrics on the face $\partial_i^1E$ are induced by the metrics of the face $\partial_i^0\overline{E}$. To measure the difference between $\overline{E}$ and $\lambda_i^1\overline{E}$, let $\lambda_i^2(\overline{E})$ be the hermitian $k$-cube determined by $\partial_i^{-1}\lambda_i^2(\overline{E})=\partial_i^1\overline{E}$, $\partial_i^0\lambda_i^2(\overline{E})=\partial_i^1\lambda_i^1(\overline{E})$, and $\partial_i^1\lambda_i^2(\overline{E})=0$. Set $\lambda_i=\lambda_i^1+\lambda_i^2$, $\lambda=\lambda_k\circ\cdots\circ\lambda_1$ if $k\geq1$ and $\lambda={\rm Id}$ otherwise. Then the map $\lambda$ induces a morphism of complexes $$\widetilde{\Z}C_*\mathcal{U}\to \widetilde{\Z}C_*^{\rm emi}\mathcal{U}$$ which is the quasi-inverse of the inclusion $\widetilde{\Z}C_*^{\rm emi}\mathcal{U}\hookrightarrow \widetilde{\Z}C_*\mathcal{U}$. To specify the $\mu_n$-equivariant variety $X$, we shall write
$\widetilde{\Z}C_*(X,\mu_n):=\widetilde{\Z}C_*\mathcal{U}$.

\begin{defn}\label{202}
Fix a primitive $n$-th root of unity $\zeta_n$, the restriction of an equivariant hermitian vector bundle $\overline{F}\mid_{X_{\mu_n}}$ over the fixed point subvariety splits into a direct sum $\oplus_{l=1}^n\overline{F}_l$ where $F_l$ is the eigenbundle of $F\mid_{X_{\mu_n}}$ corresponding to the eigenvalue ${\zeta_n}^l$. Let $K_l$ be the curvature form with respect to the unique connection on $\overline{F}_l$ compatible with both the hermitian and the complex structure, the equivariant Chern-Weil form associated to $\overline{F}$ is defined as 
$${\rm ch}_g^0(\overline{F}):=\sum_{l=1}^n{\zeta_n}^l{\rm Tr}\big(\exp(-K_l)\big).$$
Define $R_n=\R$ if $n=1$ and $R_n=\C$ otherwise, denote $V\otimes_{\R}{R_n}$ by $V_{R_n}$ for any real vector space $V$, the equivariant higher Bott-Chern form associated to hermitian $k$-cube $\overline{E}$ is defined as $${\rm ch}_g^k(\overline{E}):={\rm ch}_g^0\Big({\rm tr}_k\big(\lambda(\overline{E})\big)\Big)\in \bigoplus_{p\geq0}\widetilde{D}^*(X_{\mu_n},p)[2p]_{R_n}.$$
\end{defn}

\begin{defn}
Let $\overline{F}\mid_{X_{\mu_n}}=\oplus_{l=1}^n\overline{F}_l$ be the restriction of an equivariant hermitian vector bundle over the fixed point subvariety, where $F_l$ is the eigenbundle of $F\mid_{X_{\mu_n}}$ corresponding to the eigenvalue ${\zeta_n}^l$ and $K_l$ is the curvature form of $\overline{F}_l$. The equivariant Todd form is defined as 
$${\rm Td}_g(\overline{F})={\rm det}\big(\frac{-K_n}{1-e^{K_n}}\big)\prod_{l\neq n}{\rm det}\big(\frac{1}{1-\zeta_n^{-l}e^{K_l}}\big).$$
\end{defn}

When $X$ is proper, Burgos and Wang gave in \cite[Section 6]{BW} a quasi-inverse $\varphi: \widetilde{D}^*(X,p)\to D^*(X,p)$ of the quasi-isomorphism $\iota: D^*(X,p)\to \widetilde{D}^*(X,p)$. By means of this quasi-inverse, the equivariant higher Bott-Chern form has another expression with value in $\bigoplus_{p\geq0}D^*(X_{\mu_n},p)[2p]_{R_n}$. To see this expression, let us set $z=x/y$ which defines the coordinate map $\mathbb{C}\to \mathbb{P}^1_{\mathbb{C}}$ by sending $z\to [z,1]$. Then $\log\mid z\mid$ defines an $L^1$ function on $\mathbb{P}^1_{\mathbb{C}}$, which can be considered as a current. We shall denote by $\log\mid z_1\mid,\cdots,\log\mid z_k\mid$ the corresponding currents on $(\mathbb{P}^1_{\mathbb{C}})^k$. These currents can be formally considered as elements in $D^1\big((\mathbb{P}^1_{\mathbb{C}})^k,1\big)$, and they satisfy the following differential equation
$$d_\mathcal{D}\log\mid z_j\mid=-2\partial\overline{\partial}\log\mid z_j\mid=-2i\pi(\delta_{\mathbb{P}^1_{\mathbb{C}}\times\mathbb{P}^1_{\mathbb{C}}\times\cdots\times\{\infty\}\times\cdots\times\mathbb{P}^1_{\mathbb{C}}}-
\delta_{\mathbb{P}^1_{\mathbb{C}}\times\mathbb{P}^1_{\mathbb{C}}\times\cdots\times\{0\}\times\cdots\times\mathbb{P}^1_{\mathbb{C}}})$$
where $\infty$ and $0$ stand at the $j$-th place. Let $u_1,\cdots,u_k$ be $k$ elements in $\bigoplus_{p\geq 0}D^{2p-1}(\cdot,p)$, we define an element in $\bigoplus_{p\geq 0}D^{2p-k}(\cdot,p)$ by the formula
$$C_k(u_1,\cdots,u_k):=-(-\frac{1}{2})^{k-1}\sum_{\sigma\in \mathfrak{S}_k}(-1)^\sigma u_{\sigma(1)}\bullet(u_{\sigma(2)}\bullet(\cdots u_{\sigma(k)})\cdots)$$ where $\mathfrak{S}_k$ stands for the $k$-th symmetric group. Then we have
\begin{align}\label{dc}
d_\mathcal{D}C_k(u_1,\cdots,u_k)&=(-\frac{1}{2})k\sum_{j=1}^k(-1)^{j-1}d_\mathcal{D}(u_j)\bullet C_{k-1}(u_1,\cdots,\widehat{u_j},\cdots,u_k)\notag \\
&=(-\frac{1}{2})k\sum_{j=1}^k(-1)^{j-1}d_\mathcal{D}(u_j)\wedge C_{k-1}(u_1,\cdots,\widehat{u_j},\cdots,u_k).
\end{align}
We refer to \cite[Lemma 2.9]{Roe} for a proof of these identities. With the above notations, the equivariant higher Bott-Chern form associated to a hermitian $k$-cube $\overline{E}$ with $k\geq 0$ is given by the expression
$$\varphi\big({\rm ch}_g^k(\overline{E})\big)=\frac{(-1)^k}{2k!(2\pi i)^k}\int_{(\mathbb{P}^1)^k}{\rm ch}_g^k(\overline{E})\wedge C_k(\log\mid z_1\mid^2,\cdots,\log\mid z_k\mid^2).$$

\begin{thm}\label{203}
The equivariant higher Bott-Chern forms induce a morphism of complexes
$$\xymatrix{ \widetilde{\Z}C^*(X,\mu_n) \ar[r]^-{\lambda} & \widetilde{\Z}C^*_{\rm emi}(X,\mu_n) \ar[r]^-{{\rm ch}_g^0\circ{\rm tr}_*} & \bigoplus_{p\geq0}\widetilde{D}^*(X_{\mu_n},p)[2p]_{R_n} \ar[r]^-{\varphi} & \bigoplus_{p\geq0}D^*(X_{\mu_n},p)[2p]_{R_n},}$$ which is denoted by ${\rm ch}_g$. Here, $\widetilde{\Z}C^*(X,\mu_n)$ and $\widetilde{\Z}C^*_{\rm emi}(X,\mu_n)$ are the (cohomological) complexes associated to the homological complexes $\widetilde{\Z}C_*(X,\mu_n)$ and $\widetilde{\Z}C_*^{\rm emi}(X,\mu_n)$.
\end{thm}

Specify to the case $k=1$, let $\bar{\varepsilon}: 0\to \bar{E}_{-1}\to \bar{E}_0\to \bar{E}_1\to 0$ be a hermitian $1$-cube, then 
\begin{align*}
d_\mathcal{D}{\rm ch}_g(\bar{\varepsilon})&=d_\mathcal{D}\bigg(\frac{1}{4\pi i}\int_{\mathbb{P}^1}{\rm ch}_g^0\Big({\rm tr}_1\big(\lambda(\bar{\varepsilon})\big)\Big)\log\mid z\mid^2\bigg)\\
&={\rm ch}_g(\bar{E}_{0})-{\rm ch}_g(\bar{E}_{-1})-{\rm ch}_g(\bar{E}_{1}).
\end{align*}
If $\bar{\varepsilon}$ is split, by replacing $z$ by $1/z$, we know that 
\begin{align*}
{\rm ch}_g(\bar{\varepsilon})&=\frac{1}{4\pi i}\int_{\mathbb{P}^1}{\rm ch}_g^0\Big({\rm tr}_1\big(\lambda(\bar{\varepsilon})\big)\Big)\log\mid z\mid^2\\
&=\frac{1}{4\pi i}\int_{\mathbb{P}^1}{\rm ch}_g^0\big(\bar{E}_{1}(1)\big)\log\mid z\mid^2+\frac{1}{4\pi i}\int_{\mathbb{P}^1}{\rm ch}_g^0\big(\frac{\bar{E}_{-1}(1)\oplus \bar{E}_{-1}(1)}{\bar{E}_{-1}}\big)\log\mid z\mid^2\\
&=0.
\end{align*}

Let ${\rm ch}'_g$ denote the usual equivariant Chern-Weil forms with the factor $2\pi i$ inside $${\rm ch}'_g(\bar E)=\sum_{l=1}^n{\zeta_n}^l{\rm Tr}\big(\exp(\frac{-K_l}{2\pi i})\big),$$ and let $\Phi$ be an operator acting on $2n$-forms by $\Phi(\alpha)=(2\pi i)^{-n}\alpha$. Then $$\Phi\big({\rm ch}_g(\bar E)\big)={\rm ch}'_g(\bar E)$$ and 
$$\frac{\bar{\partial}\partial}{2\pi i}\Big(2\Phi\big({\rm ch}_g(\bar{\varepsilon})\big)\Big)={\rm ch}'_g(\bar{E}_{0})-{\rm ch}'_g(\bar{E}_{-1})-{\rm ch}'_g(\bar{E}_{1}).$$
This means, after a rescaling, ${\rm ch}_g(\bar{\varepsilon})$ satisfies the axiomatic conditions for a theory of unique equivariant secondary Bott-Chern classes \cite[Theorem 3.4]{KR1} (See \cite[\S1, (f)]{BGS} for the non-equivariant case). Notice that in \cite{BGS}, the authors used the supertraces of Quillen's superconnections to define the non-equivariant secondary Bott-Chern form $\widetilde{\rm ch}$. Split $\bar{\varepsilon}\mid_{X_{\mu_n}}$ into a direct sum of short exact sequences of its eigenbundles $\oplus_{l=1}^n \bar{\varepsilon}_l$ and define
$$\widetilde{{\rm ch}}_g(\bar{\varepsilon}):=\sum_{l=1}^n {\zeta_n}^l \widetilde{\rm ch}(\bar{\varepsilon}_l).$$
Then we get another way, using the supertraces of Quillen's superconnections, to define the equivariant secondary Bott-Chern form $\widetilde{{\rm ch}}_g(\bar{\varepsilon})$ which satisfies the equation $$\frac{\bar{\partial}\partial}{2\pi i}\widetilde{{\rm ch}}_g(\bar{\varepsilon})={\rm ch}'_g(\bar{E}_{0})-{\rm ch}'_g(\bar{E}_{-1})-{\rm ch}'_g(\bar{E}_{1}).$$ So, $2\Phi\big({\rm ch}_g(\bar{\varepsilon})\big)$ must be equal to $\widetilde{{\rm ch}}_g(\bar{\varepsilon})$ modulo 
$\Im\partial+\Im\bar{\partial}$. Let us write $2\Phi\big({\rm ch}_g(\bar{\varepsilon})\big)-\widetilde{{\rm ch}}_g(\bar{\varepsilon})=\partial\Delta_\partial(\bar{\varepsilon})+\bar{\partial}\Delta_{\bar{\partial}}(\bar{\varepsilon})$, the following theorem states that $\Delta_\partial(\bar{\varepsilon})$ and $\Delta_{\bar{\partial}}(\bar{\varepsilon})$ can be chosen to admit some funtorial property.

\begin{thm}\label{explain_t}
Let notations and assumptions be as above. There is a functorial choice of the differential forms $\Delta_\partial(\bar{\varepsilon})$ and $\Delta_{\bar{\partial}}(\bar{\varepsilon})$ such that 
$$2\Phi\big({\rm ch}_g(\bar{\varepsilon})\big)-\widetilde{{\rm ch}}_g(\bar{\varepsilon})=\partial\Delta_\partial(\bar{\varepsilon})+\bar{\partial}\Delta_{\bar{\partial}}(\bar{\varepsilon})$$
and that
$\Delta_\partial(j^*\bar{\varepsilon})=j^*\Delta_\partial(\bar{\varepsilon}), \Delta_{\bar{\partial}}(j^*\bar{\varepsilon})=j^*\Delta_{\bar{\partial}}(\bar{\varepsilon})$ for any equivariant morphism $j: X'\to X$.
\end{thm}
\begin{proof}
For hermitian $1$-cube $\bar{\varepsilon}: 0\to \bar{E}_{-1}\to \bar{E}_0\to \bar{E}_1\to 0$, we divide it into two emi-$1$-cubes
$$\bar{\varepsilon}_1: 0\to \bar{E}_{-1}\to \bar{E}_0\to \bar{E}'_1\to 0$$ and 
$$\bar{\varepsilon}_2: 0\to \bar{E}_{1}\to \bar{E}'_1\to 0\to 0$$
where $\bar{E}'_1$ is $E_1$ endowed with the quotient metric. 
According to the definition of the morphism $\lambda$, the higher Bott-Chern form is additive ${\rm ch}_g(\bar{\varepsilon})={\rm ch}_g(\bar{\varepsilon}_1)+{\rm ch}_g(\bar{\varepsilon}_2)$. To study the secondary Bott-Chern form constructed by the supertraces of Quillen's superconnections, we write down a double complex
\begin{align}
\xymatrix{ & 0 \ar[d] & 0 \ar[d] & 0 \ar[d] & \\
0 \ar[r] & \bar{E}_{-1} \ar[r] \ar[d] & \bar{E}_0 \ar[r] \ar[d] & \bar{E}_1 \ar[r] \ar[d] & 0 \\
0 \ar[r] & \bar{E}_{-1} \ar[r] \ar[d] & \bar{E}_0 \ar[r] \ar[d] & \bar{E}'_1 \ar[r] \ar[d] & 0 \\
0 \ar[r] & 0 \ar[r] \ar[d] & 0 \ar[r] \ar[d] & 0 \ar[r] \ar[d] & 0 \\
& 0 & 0 & 0 &}
\end{align}

Restrict every bundle over $X_{\mu_n}$ and split it into the direct sum of eigenbundles, then one can immediately repeat the argument given in \cite[Theorem 1.20]{BGS} (where the non-equivariant bundles were dealt with) to write down a proof of the fact that $\widetilde{{\rm ch}}_g(\bar{\varepsilon})=\widetilde{{\rm ch}}_g(\bar{\varepsilon}_1)+\widetilde{{\rm ch}}_g(\bar{\varepsilon}_2)$ modulo $\Im\partial+\Im\bar{\partial}$.  In the proof of \cite[Theorem 1.20]{BGS}, the error terms were explicitly written down and were functorial (see \cite[(1.71) (1.72) (1.78) (1.81) (1.82)]{BGS}). That means one can fix a functorial choice of differential forms $\Delta'_\partial(\bar{\varepsilon})$ and $\Delta'_{\bar{\partial}}(\bar{\varepsilon})$ such that 
$$\widetilde{{\rm ch}}_g(\bar{\varepsilon})-\big(\widetilde{{\rm ch}}_g(\bar{\varepsilon}_1)+\widetilde{{\rm ch}}_g(\bar{\varepsilon}_2)\big)=\partial\Delta'_\partial(\bar{\varepsilon})+\bar{\partial}\Delta'_{\bar{\partial}}(\bar{\varepsilon}).$$
So we may reduce our proof to the case where $\bar{\varepsilon}$  is an emi-$1$-cube.

Now we consider the following exact sequence on $X\times \mathbb{P}^1$
$$\xymatrix{ \Psi:\quad 0 \ar[rr] && \bar{E}_{-1}  \ar[rr]^-{{\rm Id}\otimes \sigma_\infty\oplus i\otimes \sigma_0} && \bar{E}_{-1}(1)\oplus\bar{E}_0(1) \ar[rr] && {\rm tr}_1(\bar{\varepsilon}) \ar[rr] && 0, }
$$
we compute, using the fact that $\int_{\mathbb{P}^1}{\rm ch}_g(\bar{E}_{-1})\log\mid z\mid=0$, $\int_{\mathbb{P}^1}{\rm ch}_g\big(\bar{E}_{-1}(1)\oplus \bar{E}_{0}(1)\big)\log\mid z\mid=0$ and the Stokes  formula, 
\begin{align*}
2\Phi\big({\rm ch}_g(\bar{\varepsilon})\big)&=\int_{\mathbb{P}^1}\Phi\Big({\rm ch}_g\big({\rm tr}_1(\bar{\varepsilon})\big)\Big)\log\mid z\mid^2\\
&=-\int_{\mathbb{P}^1}\frac{\bar{\partial}\partial}{2\pi i}\widetilde{{\rm ch}}_g(\Psi)\log \mid z\mid^2\\
&=\frac{1}{2\pi i}\int_{\mathbb{P}^1}(\partial_X+\partial_z)(\bar{\partial}_X+\bar{\partial}_z)\widetilde{{\rm ch}}_g(\Psi)\log \mid z\mid^2\\
&=\partial_X\bar{\partial}_X(\frac{1}{2\pi i}\int_{\mathbb{P}^1}\widetilde{{\rm ch}}_g(\Psi)\log \mid z\mid^2)+\partial_X(\frac{1}{2\pi i}\int_{\mathbb{P}^1}\bar{\partial}_z\widetilde{{\rm ch}}_g(\Psi)\log \mid z\mid^2)\\
&\quad-\bar{\partial}_X(\frac{1}{2\pi i}\int_{\mathbb{P}^1}{\partial}_z\widetilde{{\rm ch}}_g(\Psi)\log \mid z\mid^2)+\frac{1}{2\pi i}\int_{\mathbb{P}^1}{\partial}_z\bar{\partial}_z\widetilde{{\rm ch}}_g(\Psi)\log \mid z\mid^2\\
&=\frac{1}{2\pi i}\int_{\mathbb{P}^1}\widetilde{{\rm ch}}_g(\Psi){\partial}_z\bar{\partial}_z\log \mid z\mid^2
+\partial_X\bar{\partial}_X(\frac{1}{2\pi i}\int_{\mathbb{P}^1}\widetilde{{\rm ch}}_g(\Psi)\log \mid z\mid^2)\\
&\quad+\partial_X(\frac{1}{2\pi i}\int_{\mathbb{P}^1}\bar{\partial}_z\widetilde{{\rm ch}}_g(\Psi)\log \mid z\mid^2)-\bar{\partial}_X(\frac{1}{2\pi i}\int_{\mathbb{P}^1}{\partial}_z\widetilde{{\rm ch}}_g(\Psi)\log \mid z\mid^2)
\end{align*}
Set
$$\Delta_{\partial}(\bar{\varepsilon})=\bar{\partial}_X(\frac{1}{2\pi i}\int_{\mathbb{P}^1}\widetilde{{\rm ch}}_g(\Psi)\log \mid z\mid^2)+\frac{1}{2\pi i}\int_{\mathbb{P}^1}\bar{\partial}_z\widetilde{{\rm ch}}_g(\Psi)\log \mid z\mid^2$$
and 
$$\Delta_{\bar{\partial}}(\bar{\varepsilon})=-\frac{1}{2\pi i}\int_{\mathbb{P}^1}{\partial}_z\widetilde{{\rm ch}}_g(\Psi)\log \mid z\mid^2.$$
We get 
$$2\Phi\big({\rm ch}_g(\bar{\varepsilon})\big)=\widetilde{{\rm ch}}_g(\Psi\mid_{X\times\{\infty\}})-\widetilde{{\rm ch}}_g(\Psi\mid_{X\times\{0\}})+\partial\Delta_{\partial}(\bar{\varepsilon})+\bar{\partial}\Delta_{\bar{\partial}}(\bar{\varepsilon}).$$

By the construction,  $\Psi\mid_{X\times\{0\}}$ is split and $\Psi\mid_{X\times\{\infty\}}$ is isometric to the direct sum of $\bar{\varepsilon}$ and a split exact sequence $0\to 0\to \bar{E}_{-1}\to \bar{E}_{-1}\to 0$, we finally have 
$$2\Phi\big({\rm ch}_g(\bar{\varepsilon})\big)-\widetilde{{\rm ch}}_g(\bar{\varepsilon})=\partial\Delta_{\partial}(\bar{\varepsilon})+\bar{\partial}\Delta_{\bar{\partial}}(\bar{\varepsilon}).$$
Since the $1$-transgression bundle construction $\Psi$ is functorial, the differential forms $\Delta_{\partial}(\bar{\varepsilon})$ and $\Delta_{\bar{\partial}}(\bar{\varepsilon})$ associated to $\widetilde{{\rm ch}}_g(\Psi)$ are also functorial, thus we complete the whole proof.
\end{proof}

\begin{rem}\label{explain_tt}
(i) According to Theorem~\ref{explain_t}, we can make a functorial choice of differential forms $\Delta_\partial(\bar{\varepsilon})$ and $\Delta_{\bar{\partial}}(\bar{\varepsilon})$ for any hermitian $1$-cube $\bar{\varepsilon}$ such that 
$$2\Phi\big({\rm ch}_g(\bar{\varepsilon})\big)-\widetilde{{\rm ch}}_g(\bar{\varepsilon})=\partial\Delta_\partial(\bar{\varepsilon})+\bar{\partial}\Delta_{\bar{\partial}}(\bar{\varepsilon}).$$
Set $\Delta(\bar{\varepsilon})=-\Phi^{-1}\big(\frac{\Delta_\partial(\bar{\varepsilon})+\Delta_{\bar{\partial}}(\bar{\varepsilon})}{2}\big)$, then $\Delta(\bar{\varepsilon})$ is functorial and by the definition of the Deligne complex $\mathfrak{D}^*\big(E_{\log}(X), p\big)$ we have
$${\rm ch}_g(\bar{\varepsilon})-\Phi^{-1}(\frac{\widetilde{{\rm ch}}_g(\bar{\varepsilon})}{2})=-\pi\big((\partial+\bar{\partial})\Delta(\bar{\varepsilon})\big)=d_\mathcal{D}\Delta(\bar{\varepsilon}).$$

(ii) It is easily seen from the proof of Theorem~\ref{explain_t} that if one uses another way to define the equivariant Bott-Chern form $\widetilde{{\rm ch}}_g$ which satisfies the axiomatic conditions in \cite[Theorem 1.29]{BGS} (at the level of differential forms) and which is additive for direct sum of short exact sequences, then one can also make a functorial choice of element $\Delta(\bar{\varepsilon})$ for any hermitian emi-$1$-cube $\bar{\varepsilon}$ such that 
$${\rm ch}_g(\bar{\varepsilon})-\Phi^{-1}(\frac{\widetilde{{\rm ch}}_g(\bar{\varepsilon})}{2})=d_\mathcal{D}\Delta(\bar{\varepsilon}).$$
\end{rem}

If $X$ is a regular $\mu_n$-projective arithmetic scheme over an arithmetic ring $(D,\Sigma,F_\infty)$, we shall denote $X_\R:=\big(X(\C),F_\infty\big)$ the real variety associated to $X$ where $F_\infty$ is the antiholomorphic involution of $X(\C)$ induced by the conjugate-linear involution $F_\infty$ over $(D,\Sigma,F_\infty)$. For any sheaf of complex vector spaces $V$ with a real structure over $X_\R$, we denote by $\sigma$ the involution given by $$\omega\mapsto \overline{F_\infty^*(\omega)}.$$
Write $D^*(X_\R,p):=D^*\big(X(\C),p\big)^\sigma$ for the subcomplex of $D^*\big(X(\C),p\big)$ consisting of the fixed elements under $\sigma$, we define the real Deligne-Beilinson cohomology of $X$ as $$H_D^*\big(X,\R(p)\big):=H^*\big(D^*(X_\R,p)\big).$$

Let us denote by $\widehat{\mathcal{P}}(X,\mu_n)$ the exact category of $\mu_n$-equivariant hermitian vector bundles on $X$, and by $\widehat{S}(X,\mu_n)$ the simplicial set associated to the Waldhausen $S$-construction of $\widehat{\mathcal{P}}(X,\mu_n)$ (cf. \cite[Section 2.3]{T3}). The forgetful functor (forget about the metrics) $\pi: \widehat{\mathcal{P}}(X,\mu_n)\to \mathcal{P}(X,\mu_n)$ induces an equivalence of categories, so we have homotopy equivalence $$\mid \widehat{S}(X,\mu_n)\mid\simeq \mid S(X,\mu_n)\mid$$ and isomorphisms of abelian groups
$$K_m(X,\mu_n)\cong \pi_{m+1}(\mid \widehat{S}(X,\mu_n)\mid,0)$$ for any $m\geq 0$. To give the simplicial description of the equivariant regulator maps, we associate to each element in $S_k\widehat{\mathcal{P}}(X,\mu_n)$ a hermitian $(k-1)$-cube. Firstly, notice that an element $A$ in $S_k\widehat{\mathcal{P}}(X,\mu_n)$ is a family of injections
$$A_{0,1}\rightarrowtail A_{0,2}\rightarrowtail \cdots\rightarrowtail A_{0,k}$$
of $\mu_n$-equivariant hermitian vector bundles on $X$ with quotients $A_{i,j}\simeq A_{0,j}/{A_{0,i}}$ for each $i<j$. For $k=1$, we write
$${\rm Cub}(A_{0,1})=A_{0,1}.$$ Suppose that the map ${\rm Cub}$ is defined for all $l<k$, then ${\rm Cub}A$ is the $(k-1)$-cube with
\begin{align*}
\partial_1^{-1}{\rm Cub}A & =s_{k-2}^1\cdots s_1^1(A_{0,1}), \\
\partial_1^{1}{\rm Cub}A & ={\rm Cub}(\partial_0A).
\end{align*}

Let $\Z\widehat{S}_*(X,\mu_n)$ be the simplicial abelian group generated by the simplicial set $\widehat{S}(X,\mu_n)$, and let $\mathcal{N}\big(\Z\widehat{S}_*(X,\mu_n)\big)$ be the Moore complex associated to $\Z\widehat{S}_*(X,\mu_n)$ with differential $d=\sum_{i=0}^k(-1)^i\partial_i$ where $\partial_i$ is the face map of $\widehat{S}(X,\mu_n)$. Then, according to \cite[Corollary 4.8]{BW}, the map ${\rm Cub}$ defined above extends by linearity to a morphism of homological complexes
$${\rm Cub}:\quad \mathcal{N}\big(\Z\widehat{S}_*(X,\mu_n)\big)\to \widetilde{\Z}C_*(X,\mu_n)[-1],$$ and hence one gets a simplicial map $${\rm Cub}:\quad \Z\widehat{S}_*(X,\mu_n)\to \mathcal{K}\big(\widetilde{\Z}C_*(X,\mu_n)[-1]\big)$$ where $\mathcal{K}$ is the Dold-Puppe functor.

\begin{defn}\label{204}
Let notations and assumptions be as above. We denote by $D^{2p-*}(X_{\mu_n},p)$ the homological complex associated to the complex $\tau_{\leq0}\big(D^*(X_{\mu_n},p)[2p]\big)$ which is the canonical truncation of $D^*(X_{\mu_n},p)[2p]$ at degree $0$. We define a simplicial map
$$\xymatrix{\widetilde{{\rm ch}}_g: \widehat{S}(X,\mu_n) \ar[r]^-{\rm Hu} & \Z\widehat{S}_*(X,\mu_n) \ar[d]^-{\rm Cub} & \\ & \mathcal{K}\big(\widetilde{\Z}C_*(X,\mu_n)[-1]\big) \ar[r]^-{\mathcal{K}({\rm ch}_g)} & \mathcal{K}\big(\bigoplus_{p\geq0}D^{2p-*}(X_{\mu_n},p)[-1]_{R_n}\big),}$$
where ${\rm Hu}$ is the Hurewicz map.
\end{defn}

\begin{defn}\label{205}
Let $X$ be a regular $\mu_n$-projective scheme over an arithmetic ring $(D,\Sigma,F_\infty)$, and let $X_{\mu_n}$ be the fixed point subscheme. The higher equivariant arithmetic K-groups of $X$ are defined as $$\widehat{K}_m(X,\mu_n):=\pi_{m+1}\big(\text{homotopy fibre of }\mid\widetilde{{\rm ch}}_g\mid\big)\quad\text{for}\quad m\geq1,$$ and the equivariant regulator maps $${\rm ch}_g:\quad K_m(X,\mu_n)\to \bigoplus_{p\geq0}H_D^{2p-m}\big(X_{\mu_n},\R(p)\big)_{R_n}$$ are defined as the homomorphisms induced by $\widetilde{{\rm ch}}_g$ at the level of homotopy groups.
\end{defn}

\begin{rem}\label{206}
(i). We have the long exact sequence
$$\cdots \to \widehat{K}_m(X,\mu_n) \to K_m(X,\mu_n) \to \bigoplus_{p\geq0}H_{\mathcal{D}}^{2p-m}\big(X_{\mu_n},\R(p)\big)_{R_n} \to
\widehat{K}_{m-1}(X,\mu_n) \to \cdots$$ ending with
$$\xymatrix{\cdots\to K_1(X,\mu_n) \ar[r] & \bigoplus_{p\geq0}H_{\mathcal{D}}^{2p-1}\big(X_{\mu_n},\R(p)\big)_{R_n}
\ar[d] & \\
& \pi_1\big(\text{homotopy fibre of }\widetilde{{\rm ch}}_g\big) \ar[r] & K_0(X,\mu_n)\to \bigoplus_{p\geq0}H_{\mathcal{D}}^{2p}\big(X_{\mu_n},\R(p)\big)_{R_n}.}$$

(ii). When $n=1$, the equivariant higher Bott-Chern forms given in Definition~\ref{202} coincide with the higher Bott-Chern forms defined in \cite{BW} for non-equivariant proper varieties. So, in this case,
$${\rm ch}_g:\quad K_m(X,\mu_1)\to \bigoplus_{p\geq0}H_D^{2p-m}\big(X,\R(p)\big)$$
is the Beilinson's regulator map.

(iii). The higher equivariant arithmetic K-groups $\widehat{K}_m(X,\mu_n)$ can be defined for non-proper $X$, for details, see \cite[Section 2]{T3}.

(iv). Let $s({\rm ch}_g)$ denote the simple complex associated to the chain morphism
$$\xymatrix{ {\rm ch}_g: & \widetilde{\Z}C_*(X,\mu_n) \ar[r]^-{{\rm ch}_g} & \bigoplus_{p\geq0}D^{2p-*}(X_{\mu_n},p)_{R_n}.}$$
Then, for any $m\geq1$, there is an isomorphism $$\widehat{K}_m(X,\mu_n)_\Q\cong H_m\big(s({\rm ch}_g),\Q\big).$$

(v). A $\mu_n$-equivariant hermitian sheaf on $X$ is a $\mu_n$-equivariant coherent sheaf on $X$ which is locally free on $X(\C)$ and is equipped with a $\mu_n$-invariant hermitian metric. To a $\mu_n$-equivariant hermitian sheaf, the higher equivariant Bott-Chern form can still be defined in the same way. Denote by $\widehat{\mathcal{P}}'(X,\mu_n)$ the category of $\mu_n$-equivariant hermitian sheaves on $X$, then instead of $\widehat{\mathcal{P}}(X,\mu_n)$ one may define a new arithmetic K-theory $\widehat{K}'_*(X,\mu_n)$ which is called the equivariant arithmetic K$'$-theory. Since $\widehat{\mathcal{P}}'(X,\mu_n)$ and $\widehat{\mathcal{P}}(X,\mu_n)$ define the same algebraic K-theory when $X$ is regular, it is easily seen from the Five-lemma that the natural inclusion $\widehat{\mathcal{P}}(X,\mu_n)\subset\widehat{\mathcal{P}}'(X,\mu_n)$ induces isomorphisms $\widehat{K}_m(X,\mu_n)\cong \widehat{K}'_m(X,\mu_n)$ for any $m\geq1$.
\end{rem}

\subsection{Equivariant analytic torsion for hermitian cubes}
In \cite{BK}, J.-M. Bismut and K. K\"{o}hler developed a theory of higher analytic torsion forms
for holomorphic submersions of complex manifolds. The higher analytic torsion form solves a differential equation which gives a refinement of the Grothendieck-Riemann-Roch theorem at the level of characteristic forms. Later, in \cite{Ma1}, X. Ma generalized J.-M. Bismut and K. K\"{o}hler's results to the equivariant case. Considering the higher K-theory and the Deligne-Beilinson cohomology, to make a refinement of the Riemann-Roch theorem at the level of higher Bott-Chern forms representing the regulator maps, one needs an extension of higher analytic torsion for hermitian cubes, this has been done in \cite{Roe}. In this subsection, we do the equivariant case by using Ma's equivariant analytic torsion forms. Our construction is slightly different to Roessler's construction.

Let $X, Y$ be two smooth $\mu_n$-projective varieties over $\mathbb{C}$, and let $f: X\to Y$ be an equivariant and smooth morphism. A K\"{a}hler fibration structure on $f$ is a real closed $(1,1)$-form $\omega$ on $X$ which induces K\"{a}hler metrics on the fibres of $f$ (cf. \cite[Def. 1.1, Thm. 1.2]{BK}). For instance, we may fix a $\mu_n$-invariant K\"{a}hler metric on $X$ and choose corresponding K\"{a}hler form $\omega$ as a K\"{a}hler fibration structure on $f$. Let $(E,h^E)$ be a $\mu_n$-equivariant hermitian vector bundle on $X$ such that $E$ is $f$-acyclic i.e. the higher direct image $R^qf_*E$ vanishes for $q>0$. The equivariant analytic torsion form $T_g(f,\omega,h^E)$ is an element of $\bigoplus_{p\geq0}D^{2p-1}(Y_{\mu_n},p)_{R_n}$, which depends on $f,\omega$ and $(E,h^E)$ and satisfies the differential equation
$$d_\mathcal{D}T_g(f,\omega,h^E)={\rm ch}_g(f_*E,f_*h^E)-\frac{1}{(2\pi i)^r}\int_{X_{\mu_n}/{Y_{\mu_n}}}{\rm Td}_g(Tf,h^{Tf}){\rm ch}_g(E,h^E)$$
where $h^{Tf}$ is the hermitian metric induced by $\omega$ on the holomorphic tangent bundle $Tf$, $r$ is the rank of the bundle $Tf_{\mu_n}$, and $f_*h^E$ is the $L^2$-metric on $f_*E$ (see the end of \cite[Section 2.2]{Roe} for a definition). By definition, for elements $u, v\in (f_*E)_y$ of the fibre of $f_*E$ over a point $y\in Y$, the $L^2$-hermitian product is given by 
$$\langle u, v\rangle_{L^2}=\frac{1}{(2\pi)^b} \int_{f^{-1}y} \langle u, v\rangle_E\frac{\omega^b}{b!}$$
where $b$ is the relative dimension of $X$ over $Y$. 

We would like to caution the reader that the equivariant analytic torsion form we use here coincides with Ma's definition only up to a rescaling. If we denote by $T'_g(f,\omega,h^E)$ Ma's equivariant torsion form, then the equality $2\Phi\big(T_g(f,\omega,h^E)\big)=T'_g(f,\omega,h^E)$ holds. From now on, we shall write $T_g(\omega,h^E)$ or $T_g(h^E)$ for $T_g(f,\omega,h^E)$, if there is no ambiguity about the underlying map or K\"{a}hler form. Now, let $\omega'$ be the form associated to another K\"{a}hler fibration structure on $f: X\to Y$ and let $h'^{Tf}$ be the metric on
$Tf$ induced by this new fibration. Let $\widetilde{{\rm Td}}_g(Tf,h'^{Tf},h^{Tf})$ be the equivariant secondary Todd form used in the Appendix \big(Section 5.1 (\ref{Ma2.3.2})\big), and set ${\rm Td}_g(Tf,h'^{Tf},h^{Tf})=\Phi^{-1}\big(\frac{\widetilde{{\rm Td}}_g(Tf,h'^{Tf},h^{Tf})}{2}\big)$. So 
$$d_\mathcal{D}{\rm Td}_g(Tf,h'^{Tf},h^{Tf})={\rm Td}_g(Tf,h^{Tf})-{\rm Td}_g(Tf,h'^{Tf}).$$

The following anomaly formula is useful for our later discussion.

\begin{thm}\label{207}
Let notations and assumptions be as above. The following identity holds in $\bigoplus_{p\geq 0}\big(D^{2p-1}(Y_{\mu_n},p)/{\Im d_\mathcal{D}}\big)$:
$$T_g(\omega,h^E)-T_g(\omega',h^E)={\rm
ch}_g(f_*E,h'^{f_*E},h^{f_*E})
-\frac{1}{(2\pi i)^r}\int_{X_{\mu_n}/{Y_{\mu_n}}}{\rm
Td}_g(Tf,h'^{Tf},h^{Tf}){\rm ch}_g(E,h^E)$$
where $(f_*E,h'^{f_*E},h^{f_*E})$ stands for the emi-$1$-cubes of hermitian vector bundles
$$\xymatrix{ 0\ar[r] & (f_*E,h'^{f_*E}) \ar[r]^-{\rm Id} & (f_*E,h^{f_*E}) \ar[r] & 0 \ar[r] & 0.}$$
\end{thm}
\begin{proof}
This is a translation of \cite[Theorem 2.13]{Ma1}, see also Theorem~\ref{A10} in the Appendix. Considering the relation between the equivariant analytic torsion forms $T_g(\omega,h^E), T_g(\omega',h^E)$ and the ones used in Ma's paper, we only need to show 
$${\rm ch}_g(f_*E,h'^{f_*E},h^{f_*E})=\Phi^{-1}\big(\frac{\widetilde{{\rm ch}}_g(f_*E,h'^{f_*E},h^{f_*E})}{2}\big)\in \bigoplus_{p\geq 0}\big(D^{2p-1}(Y_{\mu_n},p)/{\Im d_\mathcal{D}}\big).$$ But this is the content of Remark~\ref{explain_tt}.
\end{proof}

According to Remark~\ref{explain_tt} and Theorem~\ref{A1} in the Appendix, there exists a functorial choice of the differential form which measures the difference
$$T_g(\omega,h^E)-T_g(\omega',h^E)-{\rm ch}_g(f_*E,h'^{f_*E},h^{f_*E})+\frac{1}{(2\pi i)^r}\int_{X_{\mu_n}/{Y_{\mu_n}}}{\rm
Td}_g(Tf,h'^{Tf},h^{Tf}){\rm ch}_g(E,h^E)$$
in Theorem~\ref{207}. With the same notations as in Remark~\ref{explain_tt} and Theorem~\ref{A1}, we set
$$\Delta(f,\overline{E},\omega,\omega'):=-\Phi^{-1}\big(\frac{\Delta^0(f,\overline{E},\omega,\omega')+\Delta_0(f,\overline{E},\omega,\omega')}{2}\big)+\Delta(f_*E,h'^{f_*E},h^{f_*E}),$$ it satisfies the differential equation
\begin{multline*}
d_\mathcal{D}\Delta(f,\overline{E},\omega,\omega')=T_g(\omega,h^E)-T_g(\omega',h^E)-{\rm ch}_g(f_*E,h'^{f_*E},h^{f_*E})\\+\frac{1}{(2\pi i)^r}\int_{X_{\mu_n}/{Y_{\mu_n}}}{\rm
Td}_g(Tf,h'^{Tf},h^{Tf}){\rm ch}_g(E,h^E).
\end{multline*}
We consider the following setting. Let $Z$ be a compact K\"{a}hler manifold and let $Z_1$ be a closed submanifold of $Z$. Choose a K\"{a}hler metric on $Z$ and endow $Z_1$ with the restricted metric. Let $f_Z: X\times Z\to Y\times Z$ be the induced map and let $\omega,\omega'$ be the K\"{a}hler forms of the product metrics on $X\times Z$ with respect to two K\"{a}hler fibrations on $f: X\to Y$. Similarly, let $f_{Z_1}: X\times Z_1\to Y\times Z_1$ be the induced map and let $\omega_1,\omega'_1$ be the K\"{a}hler forms of the product metrics on $X\times Z_1$ with respect to the same two K\"{a}hler fibrations on $f: X\to Y$. We shall denote by $j$ (resp. $i$) the natural embedding $X\times Z_1\to X\times Z$ (resp. $Y\times Z_1\to Y\times Z$). Then $j^*\omega=\omega_1$ and $j^*\omega'=\omega'_1$. Let $\overline{E}$ be an $f_Z$-acyclic hermitian bundle on $X\times Z$, we have the following result.

\begin{lem}\label{delta}
The identity $i_{\mu_n}^*\Delta(f_Z,\overline{E},\omega,\omega')=\Delta(f_{Z_1},j^*\overline{E},\omega_1,\omega'_1)$ holds.
\end{lem}
\begin{proof}
This is a consequence of Theorem~\ref{A1} in the Appendix.
\end{proof}

\begin{defn}\label{homotopy-diagram}
By a chain homotopy of a diagram of homological complexes
\begin{align*}
\xymatrix{ A_* \ar[r]^-{i} \ar[d]^-{f} & B_* \ar[d]^-{l} \\
C_* \ar[r]^-{j} & D_*,}
\end{align*}
we understand a chain homotopy between the complex morphisms $j\circ f$ and $l\circ i$.
\end{defn}

Roughly speaking, the equivariant analytic torsion for hermitian cubes is a chain homotopy of the following diagram
\begin{align}\label{atc}
\xymatrix{ \widetilde{\Z}C_*^{f-\text{ac}}(X,\mu_n) \ar[r]^-{{\rm ch}_g} \ar[d]^-{f_*} & \bigoplus_{p\geq0}D^{2p-*}(X_{\mu_n},p)_{R_n} \ar[d]^-{{f_{\mu_n}}_*\circ {\rm Td}_g(\overline{Tf})\bullet(\cdot)} \\
\widetilde{\Z}C_*(Y,\mu_n) \ar[r]^-{{\rm ch}_g} & \bigoplus_{p\geq0}D^{2p-*}(Y_{\mu_n},p)_{R_n}}
\end{align}
where $\widetilde{\Z}C_*^{f-\text{ac}}(X,\mu_n)$ is the subcomplex of $\widetilde{\Z}C_*(X,\mu_n)$ made of $f$-acyclic bundles. Since the Waldhausen K-theory space of $\widehat{\mathcal{P}}(X,\mu_n)$ is homotopy equivalent to the Waldhausen K-theory space of the full subcategory of $\widehat{\mathcal{P}}(X,\mu_n)$ consisting of $f$-acyclic bundles, we shall always work with acyclic bundles. 

Like the non-equivariant case treated in \cite{Roe}, the equivariant analytic torsion for hermitian cubes induces a commutative diagram at the level of homology groups and hence one gets an analytic proof of the equivariant version of Gillet's Riemann-Roch theorem for higher algebraic K-theory.

To construct a chain homotopy of (\ref{atc}), let us move in two steps. Notice that the equivariant higher Bott-Chern form factors as
$$\xymatrix{ \widetilde{\Z}C_*(X,\mu_n) \ar[r]^-{\lambda} & \widetilde{\Z}C_*^{\rm emi}(X,\mu_n) \ar[r]^-{{\rm ch}_g^0\circ {\rm tr}_*} & \bigoplus_{p\geq0}D^{2p-*}(X_{\mu_n},p)_{R_n},}$$
we firstly clarify the difference between $f_*\big({\rm tr}\circ\lambda(\cdot)\big)$ and ${\rm tr}\circ\lambda\big(f_*(\cdot)\big)$. Let $\overline{E}$ be a $f$-acyclic hermitian $k$-cube in $\widehat{\mathcal{P}}(X,\mu_n)$. The hermitian bundles $f_*\big({\rm tr}_k\circ\lambda(\overline{E})\big)$ and ${\rm tr}_k\circ\lambda\big(f_*(\overline{E})\big)$ are canonically isomorphic as bundles, but carry in general different metrics. For instance, assume that $\overline{E}$ is a hermitian emi-$1$-cube, then $f_*\big({\rm tr}_1(\overline{E})\big)$ and  ${\rm tr}_1\big(f_*(\overline{E})\big)$ fit into the following two exact sequences
\begin{align*}
\xymatrix{ 0 \ar[r] & f_*(p_X^*\bar{E}_{-1})  \ar[r] & f_*(p_X^*\bar{E}_{-1})(1)\oplus f_*(p_X^*\bar{E}_0)(1) \ar[r] & f_*\big({\rm tr}_1(\overline{E})\big)  \ar[r] & 0, }
\end{align*}
and 
\begin{align*}
\xymatrix{ 0 \ar[r] & p_Y^*\big(f_*(\bar{E}_{-1})\big)  \ar[r] & p_Y^*\big(f_*(\bar{E}_{-1})\big)(1)\oplus p_Y^*\big(f_*(\bar{E}_0)(1)\big) \ar[r] & {\rm tr}_1\big(f_*(\overline{E})\big)  \ar[r] & 0. }
\end{align*}
Here $p_X$ (resp. $p_Y$) stands for the obvious projection $X\times \mathbb{P}^1\to X$ (resp. $Y\times \mathbb{P}^1\to Y$). By the definition of the $L^2$-metric, over the point $(y, t)$ in $Y\times \mathbb{P}^1$, the hermitian product on $f_*\big({\rm tr}_1(\overline{E})\big)_{(y, t)}$ relies on the integral of certain power of the K\"{a}hler form $\omega_{X\times \mathbb{P}^1}$ over the fibre $f_t$ and hence relies on $t$. But the pull-bak hermitian products on $p_Y^*\big(f_*(\bar{E}_0)_{(y, t)}$ and on $p_Y^*\big(f_*(\bar{E}_{-1})_{(y, t)}$ equal the hermitian products on $f_*(\bar{E}_0)_y$ and on $f_*(\bar{E}_{-1})_y$ which don't rely on $t$, therefore the induced hermitian product on ${\rm tr}_1\big(f_*(\overline{E})_{(y, t)}$ doesn't rely on $t$ neither. So in general, $f_*\big({\rm tr}_1(\overline{E})\big)$ and  ${\rm tr}_1\big(f_*(\overline{E})\big)$ carry different metrics.

In the following, we shall write $H(\overline{E})$ for the short exact sequence
$$\xymatrix{ 0\ar[r] & f_*\big({\rm tr}_k\circ\lambda(\overline{E})\big) \ar[r]^-{\rm Id} & {\rm tr}_k\circ\lambda\big(f_*(\overline{E})\big) \ar[r] & 0 \ar[r] & 0}$$
which is an emi-$1$-cube of hermitian bundles on $Y\times (\mathbb{P}^1)^k$. The transgression bundle of $H(\overline{E})$ is a hermitian bundle on $Y\times (\mathbb{P}^1)^{k+1}=Y\times (\mathbb{P}^1)^k\times \mathbb{P}^1$. But here we change the order of the $\mathbb{P}^1$, let $p_1$ be the first projection from $Y\times\mathbb{P}^1\times (\mathbb{P}^1)^k$ to $Y\times (\mathbb{P}^1)^k$, we apply the transgression bundle construction to the short exact sequence $H(\overline{E})$ with respect to the projection $p_1$ to get a hermitian bundle on $Y\times (\mathbb{P}^1)^{k+1}$. With some abuse of notation, we still denote this hermitian bundle by ${\rm tr}_1\big(H(\overline{E})\big)$ and it satisfies the following relations:
$${\rm tr}_1\big(H(\overline{E})\big)\mid_{Y\times \{0\}\times (\mathbb{P}^1)^k}={\rm tr}_k\circ\lambda\big(f_*(\overline{E})\big),\quad {\rm tr}_1\big(H(\overline{E})\big)\mid_{Y\times \{\infty\}\times (\mathbb{P}^1)^k}=f_*\big({\rm tr}_k\circ\lambda(\overline{E})\big)$$
and
$${\rm tr}_1\big(H(\overline{E})\big)\mid_{Y\times (\mathbb{P}^1)^{i}\times \{0\}\times (\mathbb{P}^1)^{k-i}}={\rm tr}_1\big(H(\partial_i^0\overline{E})\big),$$
$${\rm tr}_1\big(H(\overline{E})\big)\mid_{Y\times (\mathbb{P}^1)^{i}\times \{\infty\}\times (\mathbb{P}^1)^{k-i}}={\rm tr}_1\big(H(\partial_i^{-1}\overline{E})\big)\oplus {\rm tr}_1\big(H(\partial_i^{1}\overline{E})\big)$$
for $i=1,\cdots,k$. Now we define
$$\Pi'_k(\overline{E}):=\frac{(-1)^{k+1}}{2(k+1)!(2\pi i)^{k+1}}\int_{(\mathbb{P}^1)^{k+1}}{\rm ch}_g^0\Big({\rm tr}_1\big(H(\overline{E})\big)\Big)\wedge C_{k+1}(\log\mid z_1\mid^2,\cdots,\log\mid z_{k+1}\mid^2).$$
The same reasoning as in \cite[Lemma 3.3]{Roe} proves that $\Pi'_k$ vanishes on degenerate $k$-cubes, and hence we obtain a map $\Pi'_k: \widetilde{\Z}C_k^{f-\text{ac}}(X,\mu_n)\to \bigoplus_{p\geq0}D^{2p-k-1}(Y_{\mu_n},p)_{R_n}$ by linear extension.

\begin{prop}\label{208}
The equality
\begin{multline*}
d_\mathcal{D}\circ \Pi'_k(\overline{E})+\Pi'_{k-1}\circ d(\overline{E})\\
={\rm ch}_g(f_*\overline{E})-\frac{(-1)^k}{2k!(2\pi i)^k}\int_{(\mathbb{P}^1)^k}{\rm ch}_g^0\Big(f_*\big({\rm tr}_k\circ\lambda(\overline{E})\big)\Big)\wedge C_{k}(\log\mid z_1\mid^2,\cdots,\log\mid z_{k}\mid^2)
\end{multline*}
holds.
\end{prop}
\begin{proof}
We compute
\begin{multline*}
d_\mathcal{D}\circ \Pi'_k(\overline{E})=\frac{(-1)^{k+1}}{2(k+1)!(2\pi i)^{k+1}}\int_{(\mathbb{P}^1)^{k+1}}{\rm ch}_g^0\Big({\rm tr}_1\big(H(\overline{E})\big)\Big)\wedge d_\mathcal{D}C_{k+1}(\log\mid z_1\mid^2,\cdots,\log\mid z_{k+1}\mid^2)\\
=\frac{(-1)^{k+1}}{2(k+1)!(2\pi i)^{k+1}}\int_{(\mathbb{P}^1)^{k+1}}{\rm ch}_g^0\Big({\rm tr}_1\big(H(\overline{E})\big)\Big)\wedge\big((-\frac{1}{2})(k+1)\sum_{j=1}^{k+1}(-1)^{j-1}(-4\pi i)(\delta_{z_j=\infty}-\delta_{z_j=0})\\
{\wedge C_{k}(\log\mid z_1\mid^2,\cdots,\widehat{\log\mid z_j\mid^2},\cdots,\log\mid z_{k+1}\mid^2)\big)}\\
=\frac{(-1)^{k+1}}{2(k+1)!(2\pi i)^{k+1}}\int_{(\mathbb{P}^1)^{k+1}}{\rm ch}_g^0\Big({\rm tr}_1\big(H(\overline{E})\big)\Big)\wedge\big((-\frac{1}{2})(k+1)\sum_{j=2}^{k+1}(-1)^{j-1}(-4\pi i)(\delta_{z_j=\infty}-\delta_{z_j=0})\\
{\wedge C_{k}(\log\mid z_1\mid^2,\cdots,\widehat{\log\mid z_j\mid^2},\cdots,\log\mid z_{k+1}\mid^2)\big)}\\
+\frac{(-1)^k}{2k!(2\pi i)^k}\int_{(\mathbb{P}^1)^k}{\rm ch}_g^0\Big({\rm tr}_k\circ\lambda\big(f_*(\overline{E})\big)\Big)\wedge C_{k}(\log\mid z_1\mid^2,\cdots,\log\mid z_{k}\mid^2)\\
-\frac{(-1)^k}{2k!(2\pi i)^k}\int_{(\mathbb{P}^1)^k}{\rm ch}_g^0\Big(f_*\big({\rm tr}_k\circ\lambda(\overline{E})\big)\Big)\wedge C_{k}(\log\mid z_1\mid^2,\cdots,\log\mid z_{k}\mid^2)\\
=\frac{(-1)^{k+1}}{2k!(2\pi i)^{k}}\int_{(\mathbb{P}^1)^{k}}\Bigg(\bigg(\sum_{j=2}^{k+1}(-1)^{j-1}{\rm ch}_g^0\Big({\rm tr_1}\big(H(\partial_j^{-1}\overline{E}\oplus \partial_j^{1}\overline{E})\big)\Big)-{\rm ch}_g^0\Big({\rm tr}_1\big(H(\partial_i^0\overline{E})\big)\Big)\bigg)\hfill\\
{\wedge C_{k}(\log\mid z_1\mid^2,\cdots,\log\mid z_{k}\mid^2)\Bigg)+{\rm ch}_g(f_*\overline{E})}\\
-\frac{(-1)^k}{2k!(2\pi i)^k}\int_{(\mathbb{P}^1)^k}{\rm ch}_g^0\Big(f_*\big({\rm tr}_k\circ\lambda(\overline{E})\big)\Big)\wedge C_{k}(\log\mid z_1\mid^2,\cdots,\log\mid z_{k}\mid^2)\\
=\frac{(-1)^k}{2k!(2\pi i)^{k}}\int_{(\mathbb{P}^1)^{k}}{\rm ch}_g^0\Big({\rm tr}_1\big(H(-d\overline{E})\big)\Big)\wedge C_{k}(\log\mid z_1\mid^2,\cdots,\log\mid z_{k}\mid^2)+{\rm ch}_g(f_*\overline{E})\hfill\\
-\frac{(-1)^k}{2k!(2\pi i)^k}\int_{(\mathbb{P}^1)^k}{\rm ch}_g^0\Big(f_*\big({\rm tr}_k\circ\lambda(\overline{E})\big)\Big)\wedge C_{k}(\log\mid z_1\mid^2,\cdots,\log\mid z_{k}\mid^2)\\
=-\Pi'_{k-1}\circ d(\overline{E})+{\rm ch}_g(f_*\overline{E})-\frac{(-1)^k}{2k!(2\pi i)^k}\int_{(\mathbb{P}^1)^k}{\rm ch}_g^0\Big(f_*\big({\rm tr}_k\circ\lambda(\overline{E})\big)\Big)\wedge C_{k}(\log\mid z_1\mid^2,\cdots,\log\mid z_{k}\mid^2).\hfill
\end{multline*}
So we are done.
\end{proof}

On the other hand, we equip $X\times (\mathbb{P}^1)^k$ with the product metric and we define
$$\Pi''_k(\overline{E})=\frac{(-1)^{k+1}}{(k+1)!(2\pi i)^k}\int_{(\mathbb{P}^1)^k}C_{k+1}\Big(T_g\big({\rm tr}_k\circ\lambda(\overline{E})\big),\log\mid z_1\mid^2,\cdots,\log\mid z_{k}\mid^2\Big)$$
where $T_g\big({\rm tr}_k\circ\lambda(\overline{E})\big)$ is the equivariant higher analytic torsion of the hermitian bundle ${\rm tr}_k\circ\lambda(\overline{E})$ with respect to the fibration $f: X\times (\mathbb{P}^1)^k\to Y\times (\mathbb{P}^1)^k$. By \cite[Lemma 3.5]{Roe}, the map $\Pi''_k$ vanishes on degenerate $k$-cubes and hence we obtain a map $\Pi''_k: \widetilde{\Z}C_k^{f-\text{ac}}(X,\mu_n)\to \bigoplus_{p\geq0}D^{2p-k-1}(Y_{\mu_n},p)_{R_n}$ by linear extension.

\begin{thm}\label{209}
Set $\Pi_k=\Pi'_k+\Pi''_k$, then $\Pi_k$ defines a chain homotopy of the diagram (\ref{atc}). This map $\Pi_k: \widetilde{\Z}C_k^{f-\text{ac}}(X,\mu_n)\to \bigoplus_{p\geq0}D^{2p-k-1}(Y_{\mu_n},p)_{R_n}$ is called the equivariant higher analytic torsion for hermitian cubes.
\end{thm}
\begin{proof}
Let $\overline{E}$ be a hermitian $k$-cube in $\widetilde{\Z}C_k^{f-\text{ac}}(X,\mu_n)$, we compute
\begin{multline*}
d_\mathcal{D}\circ\Pi_k(\overline{E})+\Pi_{k-1}\circ d(\overline{E})=d_\mathcal{D}\circ\Pi'_k(\overline{E})+\Pi'_{k-1}\circ d(\overline{E})+d_\mathcal{D}\circ\Pi''_k(\overline{E})+\Pi''_{k-1}\circ d(\overline{E})\\
={\rm ch}_g(f_*\overline{E})-\frac{(-1)^k}{2k!(2\pi i)^k}\int_{(\mathbb{P}^1)^k}{\rm ch}_g^0\Big(f_*\big({\rm tr}_k\circ\lambda(\overline{E})\big)\Big)\wedge C_{k}(\log\mid z_1\mid^2,\cdots,\log\mid z_{k}\mid^2)+d_\mathcal{D}\circ\Pi''_k(\overline{E})+\Pi''_{k-1}\circ d(\overline{E}).
\end{multline*}
and
\begin{multline*}
d_\mathcal{D}\circ\Pi''_k(\overline{E})=\frac{(-1)^{k+1}}{(k+1)!(2\pi i)^k}\int_{(\mathbb{P}^1)^k}d_\mathcal{D}C_{k+1}\Big(T_g\big({\rm tr}_k\circ\lambda(\overline{E})\big),\log\mid z_1\mid^2,\cdots,\log\mid z_{k}\mid^2\Big)\\
=\frac{(-1)^{k+1}}{(k+1)!(2\pi i)^k}\int_{(\mathbb{P}^1)^k}(-\frac{1}{2})(k+1)\bigg(d_\mathcal{D}T_g\big({\rm tr}_k\circ\lambda(\overline{E})\big)\bullet C_{k}(\log\mid z_1\mid^2,\cdots,\log\mid z_{k}\mid^2)\hfill\\
+\sum_{j=1}^{k}(-1)^{j}(-4\pi i)(\delta_{z_j=\infty}-\delta_{z_j=0}){\wedge C_{k}\Big(T_g\big({\rm tr}_k\circ\lambda(\overline{E})\big),\log\mid z_1\mid^2,\cdots,\widehat{\log\mid z_j\mid^2},\cdots,\log\mid z_k\mid^2\Big)}\bigg)\\
=\frac{(-1)^{k}}{k!(2\pi i)^{k-1}}\int_{(\mathbb{P}^1)^{k-1}}\sum_{j=1}^k(-1)^{j}\bigg(C_{k}\Big(T_g\big({\rm tr}_{k-1}\circ\lambda(\partial_j^{0}\overline{E})\big),\log\mid z_1\mid^2,\cdots,\log\mid z_{k-1}\mid^2\Big)\hfill\\
-C_{k}\Big(T_g\big({\rm tr}_{k-1}\circ\lambda(\partial_j^{-1}\overline{E})\oplus {\rm tr}_{k-1}\circ\lambda(\partial_j^{1}\overline{E})\big),\log\mid z_1\mid^2,\cdots,\log\mid z_{k-1}\mid^2\Big)\bigg)\\
+\frac{(-1)^k}{2k!(2\pi i)^k}\int_{(\mathbb{P}^1)^k}\Bigg(\bigg({\rm ch}_g^0\Big(f_*\big({\rm tr}_k\circ\lambda(\overline{E})\big)\Big)-\frac{1}{(2\pi i)^r}\int_{{X_{\mu_n}\times (\mathbb{P}^1)^k}/{Y_{\mu_n}\times (\mathbb{P}^1)^k}}{\rm Td}_g(\overline{Tf}){\rm ch}_g^0\big({\rm tr}_k\circ\lambda(\overline{E})\big)\bigg)\\
{\bullet C_{k}(\log\mid z_1\mid^2,\cdots,\log\mid z_{k}\mid^2)\Bigg)}\\
=-\Pi''_{k-1}\circ d(\overline{E})+\frac{(-1)^k}{2k!(2\pi i)^k}\int_{(\mathbb{P}^1)^k}{\rm ch}_g^0\Big(f_*\big({\rm tr}_k\circ\lambda(\overline{E})\big)\Big)\wedge C_{k}(\log\mid z_1\mid^2,\cdots,\log\mid z_{k}\mid^2)\hfill\\
-\frac{1}{(2\pi i)^r}\int_{X_{\mu_n}/{Y_{\mu_n}}}{\rm Td}_g(\overline{Tf})\bullet {\rm ch}_g(\overline{E}).
\end{multline*}
Combining these two computations, we finally get
$$d_\mathcal{D}\circ\Pi_k(\overline{E})+\Pi_{k-1}\circ d(\overline{E})={\rm ch}_g(f_*\overline{E})-\frac{1}{(2\pi i)^r}\int_{X_{\mu_n}/{Y_{\mu_n}}}{\rm Td}_g(\overline{Tf})\bullet {\rm ch}_g(\overline{E}).$$
So we are done.
\end{proof}

If we are given another fibration structure $\omega'$, then for any $f$-acyclic hermitian $k$-cube $\overline{E}$ in $\widehat{\mathcal{P}}(X,\mu_n)$, the short exact sequence
\begin{align*}
\xymatrix{0 \ar[r] & (f_*E,h'^{f_*E}) \ar[r]^-{\rm Id} & (f_*E,h^{f_*E}) \ar[r] & 0 \ar[r] & 0}
\end{align*}
forms a hermitian $(k+1)$-cube $H_f(\overline{E})$ on $Y$ such that the transgression bundle ${\rm tr}_{k+1}\Big(\lambda\big(H_f(\overline{E})\big)\Big)$ satisfies the relations
$${\rm tr}_{k+1}\Big(\lambda\big(H_f(\overline{E})\big)\Big)\mid_{Y\times \{0\}\times (\mathbb{P}^1)^k}={\rm tr}_{k}\big(\lambda(f_*E,h^{f_*E})\big),$$
$${\rm tr}_{k+1}\Big(\lambda\big(H_f(\overline{E})\big)\Big)\mid_{Y\times \{\infty\}\times (\mathbb{P}^1)^k}={\rm tr}_{k}\big(\lambda(f_*E,h'^{f_*E})\big)$$
and
$${\rm tr}_{k+1}\Big(\lambda\big(H_f(\overline{E})\big)\Big)\mid_{Y\times (\mathbb{P}^1)^{i}\times \{0\}\times (\mathbb{P}^1)^{k-i}}={\rm tr}_k\Big(\lambda\big(H_f(\partial_i^0\overline{E})\big)\Big),$$
$${\rm tr}_{k+1}\Big(\lambda\big(H_f(\overline{E})\big)\Big)\mid_{Y\times (\mathbb{P}^1)^{i}\times \{\infty\}\times (\mathbb{P}^1)^{k-i}}={\rm tr}_k\Big(\lambda\big(H_f(\partial_i^{-1}\overline{E})\big)\Big)\oplus {\rm tr}_k\Big(\lambda\big(H_f(\partial_i^{1}\overline{E})\big)\Big)$$
for $i=1,\cdots,k$. Therefore, the following map
$$\Pi^{(1)}_k(\overline{E})=\frac{(-1)^{k+1}}{2(k+1)!(2\pi i)^{k+1}}\int_{(\mathbb{P}^1)^{k+1}}{\rm ch}_g^0\Big({\rm tr}_{k+1}\circ\lambda\big(H_f(\overline{E})\big)\Big)\wedge C_{k+1}(\log\mid z_1\mid^2,\cdots,\log\mid z_{k+1}\mid^2)$$
which vanishes on degenerate cubes provides a chain homotopy of homological complexes between the maps ${\rm ch}_g\circ f_*$ and ${\rm ch}_g\circ f'_*$ where $f'_*(\overline{E}):=(f_*E,h'^{f_*E})$ is the push-forward with respect to the new fibration $\omega'$. Similarly, by projection formula, the map
\begin{multline*}
\Pi^{(3)}_k(\overline{E}):=\frac{(-1)^{k}}{2k!(2\pi i)^k}\int_{(\mathbb{P}^1)^k}\bigg(\Big(\frac{1}{(2\pi i)^r}\int_{{X_{\mu_n}\times (\mathbb{P}^1)^k}/{Y_{\mu_n}\times (\mathbb{P}^1)^k}}{\rm Td}_g(Tf,h'^{Tf},h^{Tf}){\rm ch}_g^0\big({\rm tr}_k\circ\lambda(\overline{E})\big)\Big)\\
{\bullet C_{k}(\log\mid z_1\mid^2,\cdots,\log\mid z_{k}\mid^2)\bigg)}
\end{multline*}
gives a chain homotopy of homological complexes between the maps ${f_{\mu_n}}_*\circ\big({\rm Td}_g(Tf,h^{Tf})\bullet{\rm ch}_g\big)$ and ${f_{\mu_n}}_*\circ\big({\rm Td}_g(Tf,h'^{Tf})\bullet {\rm ch}_g\big)$. Finally we write $\Pi^{(2)}_k=\Pi'^{(2)}_k+\Pi''^{(2)}_k$ for the chain homotopy defined in Theorem~\ref{209} between the maps ${\rm ch}_g\circ f'_*$ and ${f_{\mu_n}}_*\circ\big({\rm Td}_g(Tf,h'^{Tf})\bullet {\rm ch}_g\big)$ with respect to the new fibration $\omega'$. Then $\Pi^{(1)}_k+\Pi^{(2)}_k-\Pi^{(3)}_k$ defines a chain homotopy between ${\rm ch}_g\circ f_*$ and ${f_{\mu_n}}_*\circ\big({\rm Td}_g(Tf,h^{Tf})\bullet {\rm ch}_g\big)$. At the end of this subsection, we compare this homotopy $\Pi^{(1)}_k+\Pi^{(2)}_k-\Pi^{(3)}_k$ with $\Pi_k$ constructed in Theorem~\ref{209}.

\begin{defn}\label{homotopy}
Let $f, l$ be two morphisms of homological complexes $A_*\to B_*$, and let $h_1,h_2$ be two chain homotopies between $f$ and $l$. We say that $h_1$ is homotopic to $h_2$ if there exists a map $H: A_*\to B_{*+2}$ satisfying the condition that $Hd-dH=h_1-h_2$.
\end{defn}

Now, we denote by $H_f^{f'}(\overline{E})$ the following emi-$2$-cube of hermitian bundles on $Y\times (\mathbb{P}^1)^k$
$$\xymatrix{f'_*\big({\rm tr}_k\circ\lambda(\overline{E})\big) \ar[r]^-{\rm Id} \ar[d]^-{\rm Id} & {\rm tr}_k\circ\lambda\big(f'_*(\overline{E})\big) \ar[r] \ar[d]^-{\rm Id} & 0 \ar[d]\\
f_*\big({\rm tr}_k\circ\lambda(\overline{E})\big) \ar[r]^-{\rm Id} \ar[d] & {\rm tr}_k\circ\lambda\big(f_*(\overline{E})\big) \ar[r] \ar[d] & 0 \ar[d]\\
0 \ar[r] & 0 \ar[r] & 0.}$$
Changing the order of the $\mathbb{P}^1\times \mathbb{P}^1$ in $(\mathbb{P}^1)^{k+2}=(\mathbb{P}^1)^k\times \mathbb{P}^1\times \mathbb{P}^1$ so that $(\mathbb{P}^1)^{k+2}=\mathbb{P}^1\times \mathbb{P}^1\times (\mathbb{P}^1)^k$, we construct a hermitian bundle ${\rm tr}_2\big(H_f^{f'}(\overline{E})\big)$ on $Y\times (\mathbb{P}^1)^{k+2}$ as the second transgression bundle of $H_f^{f'}(\overline{E})$ such that it satisfies the following relations:
$${\rm tr}_2\big(H_f^{f'}(\overline{E})\big)\mid_{Y\times \{0\}\times (\mathbb{P}^1)^{k+1}}={\rm tr}_{k+1}\Big(\lambda\big(H_f(\overline{E})\big)\Big),$$
$${\rm tr}_2\big(H_f^{f'}(\overline{E})\big)\mid_{Y\times \{\infty\}\times (\mathbb{P}^1)^{k+1}}={\rm tr}_1\Big(H_f\big({\rm tr}_k\circ\lambda(\overline{E})\big)\Big),$$
$${\rm tr}_2\big(H_f^{f'}(\overline{E})\big)\mid_{Y\times \mathbb{P}^1\times \{0\}\times (\mathbb{P}^1)^k}={\rm tr}_1\big(H(\overline{E})\big),\quad {\rm tr}_2\big(H_f^{f'}(\overline{E})\big)\mid_{Y\times \mathbb{P}^1\times \{\infty\}\times (\mathbb{P}^1)^k}={\rm tr}_1\big(H'(\overline{E})\big)$$
and
$${\rm tr}_2\big(H_f^{f'}(\overline{E})\big)\mid_{Y\times (\mathbb{P}^1)^{i+1}\times \{0\}\times (\mathbb{P}^1)^{k-i}}={\rm tr}_2\big(H_f^{f'}(\partial_i^0\overline{E})\big),$$
$${\rm tr}_2\big(H_f^{f'}(\overline{E})\big)\mid_{Y\times (\mathbb{P}^1)^{i+1}\times \{\infty\}\times (\mathbb{P}^1)^{k-i}}={\rm tr}_2\big(H_f^{f'}(\partial_i^{-1}\overline{E})\big)\oplus {\rm tr}_2\big(H_f^{f'}(\partial_i^{1}\overline{E})\big)$$
for $i=1,\cdots,k$. We set
$$\Pi_{f,k}^{f'}(\overline{E}):=\frac{(-1)^{k+2}}{2(k+2)!(2\pi i)^{k+2}}\int_{(\mathbb{P}^1)^{k+2}}{\rm ch}_g^0\Big({\rm tr}_2\big(H_f^{f'}(\overline{E})\big)\Big)\wedge C_{k+2}(\log\mid z_1\mid^2,\cdots,\log\mid z_{k+2}\mid^2).$$
Then $\Pi_{f,k}^{f'}$ vanishes on degenerate $k$-cubes, and we obtain a map $$\Pi_{f,k}^{f'}: \widetilde{\Z}C_k^{f-\text{ac}}(X,\mu_n)\to \bigoplus_{p\geq0}D^{2p-k-2}(Y_{\mu_n},p)_{R_n}$$ by linear extension.

\begin{prop}\label{210}
Let notations and assumptions be as above. Then the chain homotopy $\Pi_k$ is homotopic to the chain homotopy $\Pi^{(1)}_k+\Pi^{(2)}_k-\Pi^{(3)}_k$.
\end{prop}
\begin{proof}
Firstly, we set
\begin{multline*}
\Pi^{(3')}_k(\overline{E}):=\frac{(-1)^{k+1}}{(k+1)!(2\pi i)^k}\int_{(\mathbb{P}^1)^k}C_{k+1}\Big(\frac{1}{(2\pi i)^r}\int_{{X_{\mu_n}\times (\mathbb{P}^1)^k}/{Y_{\mu_n}\times (\mathbb{P}^1)^k}}{\rm
Td}_g(Tf,h'^{Tf},h^{Tf}){\rm ch}_g^0\big({\rm tr}_k\circ\lambda(\overline{E})\big)\\
,\log\mid z_1\mid^2,\cdots,\log\mid z_{k}\mid^2\Big).
\end{multline*}
It also defines a chain homotopy between the maps ${f_{\mu_n}}_*\circ\big({\rm Td}_g(Tf,h^{Tf})\bullet{\rm ch}_g\big)$ and ${f_{\mu_n}}_*\circ\big({\rm Td}_g(Tf,h'^{Tf})\bullet {\rm ch}_g\big)$. Since the product $\bullet$ on Deligne complex is graded commutative and is associative up to homotopy, we claim that $\Pi^{(3')}_k(\overline{E})$ is homotopic to $\Pi^{(3)}_k(\overline{E})$ so that we are left to show that $\Pi_k$ is homotopic to $\Pi^{(1)}_k+\Pi^{(2)}_k-\Pi^{(3')}_k$. Actually, our claim follows from the fact that $d_\mathcal{D}\Pi^{(3)}_k(\overline{E})-d_\mathcal{D}\Pi^{(3')}_k(\overline{E})=\Pi^{(3)}_{k-1}(-d\overline{E})-\Pi^{(3')}_{k-1}(-d\overline{E})$ and \cite[Remark 2.4, Lemma 2.5]{T3}.

Now, let $\overline{E}$ be a hermitian $k$-cube in $\widehat{\mathcal{P}}(X,\mu_n)$ which is $f$-acyclic. We compute
\begin{multline*}
d_\mathcal{D}\circ\Pi_{f,k}^{f'}(\overline{E})=\frac{(-1)^{k+2}}{2(k+2)!(2\pi i)^{k+2}}\int_{(\mathbb{P}^1)^{k+2}}{\rm ch}_g^0\Big({\rm tr}_2\big(H_f^{f'}(\overline{E})\big)\Big)\wedge d_\mathcal{D}C_{k+2}(\log\mid z_1\mid^2,\cdots,\log\mid z_{k+2}\mid^2)\\
=\frac{(-1)^{k+2}}{2(k+2)!(2\pi i)^{k+2}}\int_{(\mathbb{P}^1)^{k+2}}{\rm ch}_g^0\Big({\rm tr}_2\big(H_f^{f'}(\overline{E})\big)\Big)\wedge\big((-\frac{1}{2})(k+2)\sum_{j=1}^{k+2}(-1)^{j-1}(-4\pi i)(\delta_{z_j=\infty}-\delta_{z_j=0})\hfill\\
{\wedge C_{k+1}(\log\mid z_1\mid^2,\cdots,\widehat{\log\mid z_j\mid^2},\cdots,\log\mid z_{k+2}\mid^2)\big)}\\
=\Pi_{f,k-1}^{f'}\circ d(\overline{E})-\frac{(-1)^{k+1}}{2(k+1)!(2\pi i)^{k+1}}\int_{(\mathbb{P}^1)^{k+1}}\Bigg[\bigg({\rm ch}_g^0\Big({\rm tr}_1\big(H(\overline{E})\big)\Big)-{\rm ch}_g^0\Big({\rm tr}_1\big(H'(\overline{E})\big)\Big)\bigg)\hfill\\
{\wedge C_{k+1}(\log\mid z_1\mid^2,\cdots,\log\mid z_{k+1}\mid^2)\Bigg]}\\
+\frac{(-1)^{k+1}}{2(k+1)!(2\pi i)^{k+1}}\int_{(\mathbb{P}^1)^{k+1}}\Bigg[\Bigg({\rm ch}_g^0\bigg({\rm tr}_{k+1}\Big(\lambda\big(H_f(\overline{E})\big)\Big)\bigg)-{\rm ch}_g^0\bigg({\rm tr}_1\Big(H_f\big({\rm tr}_k\circ\lambda(\overline{E})\big)\Big)\bigg)\Bigg)\\
{\wedge C_{k+1}(\log\mid z_1\mid^2,\cdots,\log\mid z_{k+1}\mid^2)\Bigg]}\\
=\Pi_{f,k-1}^{f'}\circ d(\overline{E})-\Pi'_k(\overline{E})+\Pi'^{(2)}_k(\overline{E})+\Pi^{(1)}_k(\overline{E})\hfill\\
-\frac{(-1)^{k+1}}{2(k+1)!(2\pi i)^{k+1}}\int_{(\mathbb{P}^1)^{k+1}}{\rm ch}_g^0\bigg({\rm tr}_1\Big(H_f\big({\rm tr}_k\circ\lambda(\overline{E})\big)\Big)\bigg)\wedge
C_{k+1}(\log\mid z_1\mid^2,\cdots,\log\mid z_{k+1}\mid^2).
\end{multline*}

On the other hand, according to the anomaly formula Theorem~\ref{207}, we have
\begin{multline*}
\Pi''_k(\overline{E})-\Pi''^{(2)}_k(\overline{E})\\
=\frac{(-1)^{k+1}}{(k+1)!(2\pi i)^k}\int_{(\mathbb{P}^1)^k}C_{k+1}\Big(T_g\big({\rm tr}_k\circ\lambda(\overline{E})\big),\log\mid z_1\mid^2,\cdots,\log\mid z_{k}\mid^2\Big)\hfill\\
-\frac{(-1)^{k+1}}{(k+1)!(2\pi i)^k}\int_{(\mathbb{P}^1)^k}C_{k+1}\Big(T'_g\big({\rm tr}_k\circ\lambda(\overline{E})\big),\log\mid z_1\mid^2,\cdots,\log\mid z_{k}\mid^2\Big)\\
=\frac{(-1)^{k+1}}{(k+1)!(2\pi i)^k}\int_{(\mathbb{P}^1)^k}C_{k+1}\Bigg(\frac{1}{4\pi i}\int_{{Y_{\mu_n}\times (\mathbb{P}^1)^{k+1}}/{Y_{\mu_n}\times (\mathbb{P}^1)^{k}}}{\rm ch}_g^0\bigg({\rm tr}_1\Big(H_f\big({\rm tr}_k\circ\lambda(\overline{E})\big)\Big)\bigg)\log \mid z_{0}\mid^2\hfill\\
,\log\mid z_1\mid^2,\cdots,\log\mid z_{k}\mid^2\Bigg)\\
-\frac{(-1)^{k+1}}{(k+1)!(2\pi i)^k}\int_{(\mathbb{P}^1)^k}C_{k+1}\Big(\frac{1}{(2\pi i)^r}\int_{{X_{\mu_n}\times (\mathbb{P}^1)^k}/{Y_{\mu_n}\times (\mathbb{P}^1)^k}}{\rm
Td}_g(Tf,h'^{Tf},h^{Tf}){\rm ch}_g^0\big({\rm tr}_k\circ\lambda(\overline{E})\big)\\
,\log\mid z_1\mid^2,\cdots,\log\mid z_{k}\mid^2\Big)\\
+\frac{(-1)^{k+1}}{(k+1)!(2\pi i)^k}\int_{(\mathbb{P}^1)^k}C_{k+1}\Big(d_\mathcal{D}\Delta\big(f,{\rm tr}_k\circ\lambda(\overline{E}),\omega,\omega'\big),\log\mid z_1\mid^2,\cdots,\log\mid z_{k}\mid^2\Big)\\
=\frac{(-1)^{k+1}}{2(k+1)!(2\pi i)^{k+1}}\int_{(\mathbb{P}^1)^{k+1}}{\rm ch}_g^0\bigg({\rm tr}_1\Big(H_f\big({\rm tr}_k\circ\lambda(\overline{E})\big)\Big)\bigg)\wedge C_{k+1}(\log\mid z_1\mid^2,\cdots,\log\mid z_{k+1}\mid^2)\hfill\\
+\frac{(-1)^{k+1}}{(k+1)!(2\pi i)^k}\int_{(\mathbb{P}^1)^k}C_{k+1}\Big(d_\mathcal{D}\Delta\big(f,{\rm tr}_k\circ\lambda(\overline{E}),\omega,\omega'\big),\log\mid z_1\mid^2,\cdots,\log\mid z_{k}\mid^2\Big)-\Pi^{(3')}_k(\overline{E}).
\end{multline*}

We formally define a product $C_{k+1}\Big(\Delta\big(f,{\rm tr}_k\circ\lambda(\overline{E}),\omega,\omega'\big),\log\mid z_1\mid^2,\cdots,\log\mid z_{k}\mid^2\Big)$ in a similar way to $C_{k+1}(\cdot,\ldots,\cdot)$ like follows.
\begin{align}\label{NewC}
&C_{k+1}\Big(\Delta\big(f,{\rm tr}_k\circ\lambda(\overline{E}),\omega,\omega'\big),\log\mid z_1\mid^2,\cdots,\log\mid z_{k}\mid^2\Big)\notag\\
=&-(-\frac{1}{2})^{k}\sum_{\sigma\in \mathfrak{S}_k}(-1)^\sigma \Delta\bullet(\log\mid z_{\sigma(1)}\mid^2\bullet(\log\mid z_{\sigma(2)}\mid^2\bullet(\cdots \log\mid z_{\sigma(k)}\mid^2)\cdots)\notag\\
&-(-\frac{1}{2})^{k}\sum_{\sigma\in \mathfrak{S}_k}(-1)^\sigma \log\mid z_{\sigma(1)}\mid^2\bullet(\Delta\bullet(\log\mid z_{\sigma(2)}\mid^2\bullet(\cdots \log\mid z_{\sigma(k)}\mid^2)\cdots)\notag\\
&\cdots\notag\\
&-(-\frac{1}{2})^{k}\sum_{\sigma\in \mathfrak{S}_k}(-1)^\sigma \log\mid z_{\sigma(1)}\mid^2\bullet(\log\mid z_{\sigma(2)}\mid^2\bullet(\cdots \log\mid z_{\sigma(k)}\mid^2\bullet\Delta)\cdots)
\end{align}

Then we set
$$\Delta_k(\overline{E})=\frac{(-1)^{k}}{(k+1)!(2\pi i)^k}\int_{(\mathbb{P}^1)^k}C_{k+1}\Big(\Delta\big(f,{\rm tr}_k\circ\lambda(\overline{E}),\omega,\omega'\big),\log\mid z_1\mid^2,\cdots,\log\mid z_{k}\mid^2\Big),$$
and it is readily checked by Lemma~\ref{delta} that
$$\Delta_{k-1}(d\overline{E})-d_\mathcal{D}\Delta_k(\overline{E})=\frac{(-1)^{k+1}}{(k+1)!(2\pi i)^k}\int_{(\mathbb{P}^1)^k}C_{k+1}\Big(d_\mathcal{D}\Delta\big(f,{\rm tr}_k\circ\lambda(\overline{E}),\omega,\omega'\big),\log\mid z_1\mid^2,\cdots,\log\mid z_{k}\mid^2\Big).$$

Combing all the above computations, we finally get
\begin{multline*}
(\Pi_{f,k-1}^{f'}+\Delta_{k-1})\circ d(\overline{E})-d_\mathcal{D}\circ(\Pi_{f,k}^{f'}+\Delta_k)(\overline{E})\\
=-\Pi'^{(2)}_k(\overline{E})+\Pi'_k(\overline{E})-\Pi^{(1)}_k(\overline{E})+\Delta_{k-1}(d\overline{E})-d_\mathcal{D}\Delta_k(\overline{E})\hfill\\
+\frac{(-1)^{k+1}}{2(k+1)!(2\pi i)^{k+1}}\int_{(\mathbb{P}^1)^{k+1}}{\rm ch}_g^0\bigg({\rm tr}_1\Big(H_f\big({\rm tr}_k\circ\lambda(\overline{E})\big)\Big)\bigg)\wedge C_{k+1}(\log\mid z_1\mid^2,\cdots,\log\mid z_{k+1}\mid^2)\\
=-\Pi'^{(2)}_k(\overline{E})+\Pi'_k(\overline{E})-\Pi^{(1)}_k(\overline{E})-\Pi''^{(2)}_k(\overline{E})+\Pi''_k(\overline{E})+\Pi^{(3')}_k(\overline{E})\\
=\Pi_k(\overline{E})-\big(\Pi^{(1)}_k(\overline{E})+\Pi^{(2)}_k(\overline{E})-\Pi^{(3')}_k(\overline{E})\big).
\end{multline*}
So we are done.
\end{proof}

\subsection{Direct image map between arithmetic K-groups}
In this subsection, we define the direct image map between arithmetic K-groups of regular $\mu_n$-projective arithmetic schemes by means of the equivariant higher analytic torsion for hermitian cubes constructed in last subsection.

Let now $X$ and $Y$ be two regular $\mu_n$-projective schemes over an arithmetic ring $(D,\Sigma,F_\infty)$. Assume that $f: X\to Y$ is an equivariant and flat morphism from $X$ to $Y$ such that $f$ is smooth over the generic fibre. Notice that the chain homotopy $$\Pi_*: \widetilde{\Z}C_*^{f-\text{ac}}\big(X(\C),\mu_n\big)\to \bigoplus_{p\geq0}D^{2p-*-1}(Y(\C)_{\mu_n},p)_{R_n}$$ is $\sigma$-invariant and the following diagrams
$$\xymatrix{\widehat{S}^{f-\text{ac}}(X,\mu_n) \ar[r]^-{\rm Hu} \ar[d]^-{f_*} & \Z\widehat{S}_*^{f-\text{ac}}(X,\mu_n) \ar[r]^-{\rm Cub} \ar[d]^-{f_*} & \mathcal{K}\big(\widetilde{\Z}C_*^{f-\text{ac}}(X,\mu_n)[-1]\big) \ar[d]^-{f_*} \\
\widehat{S}(Y,\mu_n) \ar[r]^-{\rm Hu} & \Z\widehat{S}_*(Y,\mu_n) \ar[r]^-{\rm Cub} & \mathcal{K}\big(\widetilde{\Z}C_*(Y,\mu_n)[-1]\big)}$$
are commutative, the chain homotopy $\Pi_*$ induces a simplicial homotopy between the maps $\widetilde{{\rm ch}}_g\circ f_*$ and ${f_{\mu_n}}_*\circ {\rm Td}_g(\overline{Tf})\bullet(\cdot)\circ \widetilde{{\rm ch}}_g$ in the following square
$$\xymatrix{ \widehat{S}^{f-\text{ac}}(X,\mu_n) \ar[rr]^-{\widetilde{{\rm ch}}_g} \ar[d]^-{f_*} && \mathcal{K}\big(\bigoplus_{p\geq0}D^{2p-*}(X_{\mu_n},p)[-1]_{R_n}\big) \ar[d]^-{{f_{\mu_n}}_*\circ {\rm Td}_g(\overline{Tf})\bullet(\cdot)}\\
\widehat{S}(Y,\mu_n) \ar[rr]^-{\widetilde{{\rm ch}}_g} && \mathcal{K}\big(\bigoplus_{p\geq0}D^{2p-*}(Y_{\mu_n},p)[-1]_{R_n}\big).}$$
To see the construction of this simplicial homotopy and general theory on homotopies in the category of simplical abelian groups, the reader is referred to \cite[Section 2.1, Section 2.3, Section 3.2]{GJ}, especially \cite[p160, p162 Prop. 2.18, p72 Prop. 1.8 Cor. 1.9]{GJ}

We remark that, according to the construction given in \cite{GJ}, the resulting simplicial homotopy is unique up to a homotopy in a strong sense: let $h_1, h_2$ be two simplicial homotopies arising from $\Pi_*$, then there exists a homotopy 
$$\xymatrix{\widetilde{H}:\quad\widehat{S}^{f-\text{ac}}(X,\mu_n)\times \Delta^1\times \Delta^1 \ar[rr] && \mathcal{K}\big(\bigoplus_{p\geq0}D^{2p-*}(Y_{\mu_n},p)[-1]_{R_n}\big)}$$ such that 
$\widetilde{H}(\cdot, \cdot, 0)=h_1$, $\widetilde{H}(\cdot, \cdot, 1)=h_2$, $\widetilde{H}(\cdot, 0, \cdot)$ is the constant homotopy on $\widetilde{{\rm ch}}_g\circ f_*$ and $\widetilde{H}(\cdot, 1, \cdot)$ is the constant homotopy on ${f_{\mu_n}}_*\circ {\rm Td}_g(\overline{Tf})\bullet(\cdot)\circ \widetilde{{\rm ch}}_g$ (cf. \cite[Prop. 3.8]{GJ}). Thus, applying the geometric realization construction to the above simplicial square, we get a continuous map between homotopy fibres $$\mid f\mid: \text{   homotopy fibre of }\mid \widetilde{{\rm ch}}_g^X\mid\longrightarrow \text{homotopy fibre of }\mid \widetilde{{\rm ch}}_g^Y\mid$$ which is unique up to a homotopy. So we may have a well-defined direct image map between arithmetic $K$-groups as follows.

\begin{defn}\label{211}
For $m\geq1$, the direct image map $f_*: \widehat{K}_m(X,\mu_n)\to \widehat{K}_m(Y,\mu_n)$ is defined as the homomorphism of abelian groups induced by the map $\mid f\mid$ at the level of homotopy groups. 
\end{defn}

\begin{rem}\label{flat}
The condition ``flatness" of the map $f$ is only used to guarantee that the direct image of a $f$-acyclic bundle is locally free. By introducing the arithmetic K$'$-theory and using the isomorphisms $\widehat{K}_m(X,\mu_n)\cong \widehat{K}'_m(X,\mu_n)$ which hold for regular schemes, the condition ``flatness" can certainly be removed.
\end{rem}

To study the direct image map up to torsion, we need the following lemma.

\begin{lem}\label{NewCA}
Consider the following diagram of homological complexes
$$\xymatrix{ A_* \ar@/_/[d]_-{f_1} \ar@/^/[d]^-{f_2} \ar[rr]^-{i} && B_* \ar@/_/[d]_-{l_1} \ar@/^/[d]^-{l_2} \\
C_* \ar[rr]^-{j} && D_*.}$$
Assume that $j\circ {f_1}$ (resp. $j\circ {f_2}$) is homotopic to $l_1\circ i$ (resp. $l_2\circ i$) via the chain homotopy $h_1$ (resp. $h_2$), and that $f_1$ (resp. $l_1$) is homotopic to $f_2$ (resp. $l_2$) via the chain homotopy $\pi_f$ (resp. $\pi_l$). Suppose that the chain homotopy $j\circ \pi_f+h_2-\pi_l\circ i$ is homotopic to the chain homotopy $h_1$, then the morphism on simple complexes
$$(f_1,l_1,h_1): s_*(i: A_*\to B_*)\to s_*(j: C_*\to D_*)$$ is chain homotopic to $(f_2,l_2,h_2)$.
\end{lem}
\begin{proof}
Let $(a,b)\in A_k\bigoplus B_{k+1}$, the morphism $(f_1,l_1,h_1)$ (resp. $(f_2,l_2,h_2)$) sends $(a,b)$ to $\big(f_1(a),l_1(b)+h_1(a)\big)$ (resp. $\big(f_2(a),l_2(b)+h_2(a)\big)$). Let $H: A_*\to D_{*+2}$ be the homotopy such that
$$Hd-dH=h_1-(j\circ \pi_f+h_2-\pi_l\circ i),$$
and we define $\widetilde{H}(a,b)=\big(\pi_f(a),-\pi_l(b)+H(a)\big)$. Then we compute
\begin{align*}
d\widetilde{H}(a,b)=&d\big(\pi_f(a),-\pi_l(b)+H(a)\big)\\
=&\big(d\pi_f(a),j\circ \pi_f(a)+d\pi_l(b)-dH(a)\big)\\
=&\big(f_1(a)-f_2(a)-\pi_f(da),l_1(b)-l_2(b)-\pi_l(db)-Hd(a)+h_1(a)-h_2(a)+\pi_l\circ i(a)\big)\\
=&\big(f_1(a),l_1(b)+h_1(a)\big)-\big(f_2(a),l_2(b)+h_2(a)\big)-\big(\pi_fd(a),\pi_l(db)-\pi_l\circ i(a)+Hd(a)\big)\\
=&\big(f_1(a),l_1(b)+h_1(a)\big)-\big(f_2(a),l_2(b)+h_2(a)\big)-\widetilde{H}(da,i(a)-db)\\
=&\big(f_1(a),l_1(b)+h_1(a)\big)-\big(f_2(a),l_2(b)+h_2(a)\big)-\widetilde{H}d(a,b).
\end{align*}
So we are done.
\end{proof}

\begin{cor}\label{212}
Let notations and assumptions be as above, then the direct image map $f_*: \widehat{K}_m(X,\mu_n)_\Q\to \widehat{K}_m(Y,\mu_n)_\Q$ without torsion is independent of the choice of the K\"{a}hler fibration structure.
\end{cor}
\begin{proof}
This follows from Remark~\ref{206} (iv), Theorem~\ref{210} and Lemma~\ref{NewCA}.
\end{proof}

\section{Transitivity of the direct image maps}
Let $f: X\to Y$, $h: Y\to Z$ and $l: X\to Z$ be three equivariant morphisms between regular $\mu_n$-projective schemes, which are all smooth over the generic fibres. Assume that $l=h\circ f$, in this section, we shall compare the direct image map $l_*$ with the composition $h_*\circ f_*$. To this aim, we shall firstly discuss the functoriality of the equivariant analytic torsion forms with respect to a composition of submersions.

\subsection{Analytic torsion forms and families of submersions}
Let $W, V$ and $S$ be three smooth $\mu_n$-equivariant algebraic varieties over $\C$ with $S=S_{\mu_n}$. Suppose that $f: W\to V$ and $h: V\to S$ are two proper smooth morphisms, then passing to their analytifications the maps $f: W(\C)\to V(\C)$ and $h: V(\C)\to S(\C)$ are holomorphic submersions with compact fibres. Set $l=h\circ f$, it is also a proper smooth morphism and $l: W(\C)\to S(\C)$ is a holomorphic submersion with compact fibre as well.

Let $\omega^W$ and $\omega^V$ be two $\mu_n$-invariant K\"{a}hler forms on $W$ and on $V$. As before, $\omega^W$ and $\omega^V$ imply K\"{a}hler fibration structures on the morphisms $f, h$ and $l$ and they induce $\mu_n$-invariant hermitian metrics on relative tangent bundles $Tf, Th$ and $Tl$. Consider the following short exact sequence of hermitian vector bundles
$$\overline{T}(f,h,h\circ f):\quad 0\to \overline{Tf}\to \overline{Tl}\to f^*\overline{Th}\to 0,$$
denote by ${\rm Td}_g\big(\overline{T}(f,h,h\circ f)\big)=\Phi^{-1}\Big(\frac{\widetilde{{\rm Td}}_g\big(\overline{T}(f,h,h\circ f)\big)}{2}\Big)$ (see Section 5.2 in the Appendix) the equivariant secondary Todd form such that
$$d_\mathcal{D}{\rm Td}_g\big(\overline{T}(f,h,h\circ f)\big)={\rm Td}_g(\overline{Tl})-f_{\mu_n}^*{\rm Td}_g(\overline{Th}){\rm Td}_g(\overline{Tf}).$$

Now let $\overline{E}$ be a hermitian vector bundle on $W$, we shall assume that $E$ is $f$-acyclic and $l$-acyclic. Then the Leray spectral sequence
$E_2^{i,j}=R^ih_*(R^jf_*E)$ degenerates at $E_2$ so that $f_*E=R^0f_*(E)$ is $h$-acyclic and $l_*E\cong h_*f_*E$. Clearly, $l_*E$ and $h_*f_*E$ carry in general different $L^2$-metrics (See Section 5.2 in the Appendix). Consider the following short exact sequence of hermitian vector bundles
$$\overline{E}(f,h,h\circ f):\quad 0\to h_*f_*\overline{E} \to l_*\overline{E} \to 0\to 0,$$
it can be regarded as an emi-$1$-cube of hermitian bundles on $S$. Then the equivariant higher Bott-Chern form ${\rm ch}_g\big(\overline{E}(f,h,h\circ f)\big)$ satisfies the differential equation
$$d_\mathcal{D}{\rm ch}_g\big(\overline{E}(f,h,h\circ f)\big)={\rm ch}_g(l_*\overline{E})-{\rm ch}_g(h_*f_*\overline{E}).$$

The main result in this subsection is the following.

\begin{thm}\label{301}
Let notations and assumptions be as above. Then the following identity holds in $\bigoplus_{p\geq 0}\big(D^{2p-1}(S,p)/{\Im d_\mathcal{D}}\big)$:
\begin{multline*}
T_g(l,\omega^W,h^E)-T_g(h,\omega^V,h^{f_*E})-\frac{1}{(2\pi i)^{r_h}}\int_{V_{\mu_n}/S}{\rm Td}_g(\overline{Th})T_g(f,\omega^W,h^E)\\
={\rm ch}_g\big(\overline{E}(f,h,h\circ f)\big)-\frac{1}{(2\pi i)^{r_l}}\int_{W_{\mu_n}/S}{\rm Td}_g\big(\overline{T}(f,h,h\circ f)\big){\rm ch}_g(\overline{E})
\end{multline*}
where $r_h$ and $r_l$ are the relative dimensions of $V_{\mu_n}/S$ and $W_{\mu_n}/S$ respectively.
\end{thm}
\begin{proof}
This is a translation of Theorem~\ref{A20} in the Appendix.
\end{proof}

\begin{lem}\label{deltac}
With the same notations as in Remark~\ref{explain_tt} and Theorem~\ref{A2} in the Appendix, we set 
$$\Delta(f,h,\omega^W,\omega^V,\overline{E}):=-\Phi^{-1}\big(\frac{\Delta^0(f,h,\omega^W,\omega^V,\overline{E})+\Delta_0(f,h,\omega^W,\omega^V,\overline{E})}{2}\big)+\Delta\big(\overline{E}(f,h,h\circ f)\big).$$ 
Then $d_\mathcal{D}\Delta(f,h,\omega^W,\omega^V,\overline{E})$ measures the difference
\begin{multline*}
T_g(l,\omega^W,h^E)-T_g(h,\omega^V,h^{f_*E})-\frac{1}{(2\pi i)^{r_h}}\int_{V_{\mu_n}/S}{\rm Td}_g(\overline{Th})T_g(f,\omega^W,h^E)\\
-{\rm ch}_g\big(\overline{E}(f,h,h\circ f)\big)+\frac{1}{(2\pi i)^{r_l}}\int_{W_{\mu_n}/S}{\rm Td}_g\big(\overline{T}(f,h,h\circ f)\big){\rm ch}_g(\overline{E})
\end{multline*}
in Theorem~\ref{301}. Assume that we are in the same situation described before Lemma~\ref{delta}. Call $l: S\times Z_1\to S\times Z$ the natural inclusion, then similar to Lemma~\ref{delta}, we have that
$$l^*\Delta(f_Z,h_Z,\omega^W,\omega^V,\overline{E})=\Delta(f_{Z_1},h_{Z_1},\omega^W_1,\omega^V_1,j^*\overline{E}).$$
\end{lem}
\begin{proof}
This is a consequence of Theorem~\ref{A2} in the Appendix.
\end{proof}

\subsection{The transitivity property}
In this subsection, we present certain transitivity property of direct image maps between equivariant higher arithmetic K-groups. To do this, we firstly write down the following diagram of homological complexes
\begin{align}\label{atc-comp}
\xymatrix{ \widetilde{\Z}C_*^{(f, l)-\text{ac}}(X,\mu_n) \ar[r]^-{{\rm ch}_g} \ar[d]^-{f_*} & \bigoplus_{p\geq0}D^{2p-*}(X_{\mu_n},p)_{R_n} \ar[d]^-{{f_{\mu_n}}_*\circ {\rm Td}_g(\overline{Tf})\bullet(\cdot)} \\
\widetilde{\Z}C_*^{f-\text{ac}}(Y,\mu_n) \ar[r]^-{{\rm ch}_g} \ar[d]^-{h_*} & \bigoplus_{p\geq0}D^{2p-*}(Y_{\mu_n},p)_{R_n} \ar[d]^-{{h_{\mu_n}}_*\circ {\rm Td}_g(\overline{Th})\bullet(\cdot)} \\
\widetilde{\Z}C_*(Z,\mu_n) \ar[r]^-{{\rm ch}_g} & \bigoplus_{p\geq0}D^{2p-*}(Z_{\mu_n},p)_{R_n}}
\end{align}
where $l$ is $h\circ f$ and $\widetilde{\Z}C_*^{(f, l)-\text{ac}}(X,\mu_n)$ is the subcomplex of $\widetilde{\Z}C_*(X,\mu_n)$ made of those bundles which are $f$-acyclic and $l$-acyclic simultaneously.

Let $\overline{E}$ be a hermitian $k$-cube in $\widehat{\mathcal{P}}(X,\mu_n)$ which is $f$-acyclic and $l$-acyclic, the short exact sequence
\begin{align*}
\xymatrix{0 \ar[r] & h_*f_*\overline{E} \ar[r]^-{\rm Id} & l_*\overline{E} \ar[r] & 0 \ar[r] & 0}
\end{align*}
can be regarded as a hermitian $(k+1)$-cube $H_{h\circ f}(\overline{E})$ on $Z$ such that the transgression bundle ${\rm tr}_{k+1}\Big(\lambda\big(H_{h\circ f}(\overline{E})\big)\Big)$ satisfies the relations
$${\rm tr}_{k+1}\Big(\lambda\big(H_{h\circ f}(\overline{E})\big)\Big)\mid_{Z\times \{0\}\times (\mathbb{P}^1)^k}={\rm tr}_{k}\big(\lambda(l_*\overline{E})\big),$$
$${\rm tr}_{k+1}\Big(\lambda\big(H_{h\circ f}(\overline{E})\big)\Big)\mid_{Z\times \{\infty\}\times (\mathbb{P}^1)^k}={\rm tr}_{k}\big(\lambda(h_*f_*\overline{E})\big)$$
and
$${\rm tr}_{k+1}\Big(\lambda\big(H_{h\circ f}(\overline{E})\big)\Big)\mid_{Z\times (\mathbb{P}^1)^{i}\times \{0\}\times (\mathbb{P}^1)^{k-i}}={\rm tr}_k\Big(\lambda\big(H_{h\circ f}(\partial_i^0\overline{E})\big)\Big),$$
$${\rm tr}_{k+1}\Big(\lambda\big(H_{h\circ f}(\overline{E})\big)\Big)\mid_{Z\times (\mathbb{P}^1)^{i}\times \{\infty\}\times (\mathbb{P}^1)^{k-i}}={\rm tr}_k\Big(\lambda\big(H_{h\circ f}(\partial_i^{-1}\overline{E})\big)\Big)\oplus {\rm tr}_k\Big(\lambda\big(H_{h\circ f}(\partial_i^{1}\overline{E})\big)\Big)$$
for $i=1,\cdots,k$.

\begin{prop}\label{302}
The following map
$$\Pi^{(1)}_k(\overline{E})=\frac{(-1)^{k+1}}{2(k+1)!(2\pi i)^{k+1}}\int_{(\mathbb{P}^1)^{k+1}}{\rm ch}_g^0\Big({\rm tr}_{k+1}\circ\lambda\big(H_{h\circ f}(\overline{E})\big)\Big)\wedge C_{k+1}(\log\mid z_1\mid^2,\cdots,\log\mid z_{k+1}\mid^2)$$
which vanishes on degenerate cubes provides a chain homotopy of homological complexes between the maps ${\rm ch}_g\circ l_*$ and ${\rm ch}_g\circ (h_*\circ f_*)$.
\end{prop}
\begin{proof}
Using the above relations that the transgression bundle ${\rm tr}_{k+1}\Big(\lambda\big(H_{h\circ f}(\overline{E})\big)\Big)$ satisfies and the expression of $d_\mathcal{D}C_{k+1}$, the proof is straightforward. This can be also seen from the fact that $H_{h\circ f}(\overline{E})$ provides a chain homotopy between $l_*$ and $h_*\circ f_*$.
\end{proof}

\begin{prop}\label{303}
The composition ${h_{\mu_n}}_*\circ {\rm Td}_g(\overline{Th})\bullet\big({f_{\mu_n}}_*\circ {\rm Td}_g(\overline{Tf})\bullet(\cdot)\big)$ is equal to ${l_{\mu_n}}_*\circ f_{\mu_n}^*{\rm Td}_g(\overline{Th}){\rm Td}_g(\overline{Tf})\bullet(\cdot)$. The following maps
\begin{multline*}
\Pi^{(3)}_k(\overline{E}):=\frac{(-1)^{k}}{2k!(2\pi i)^k}\int_{(\mathbb{P}^1)^k}\bigg(\Big(\frac{1}{(2\pi i)^{r_l}}\int_{{X_{\mu_n}\times (\mathbb{P}^1)^k}/{Z_{\mu_n}\times(\mathbb{P}^1)^k}}{\rm Td}_g\big(\overline{T}(f,h,h\circ f)\big){\rm ch}_g^0\big({\rm tr}_k\circ\lambda(\overline{E})\big)\Big)\\
{\bullet C_{k}(\log\mid z_1\mid^2,\cdots,\log\mid z_{k}\mid^2)\bigg)}
\end{multline*}
and
\begin{multline*}
\Pi^{(3')}_k(\overline{E}):=\frac{(-1)^{k+1}}{(k+1)!(2\pi i)^k}\int_{(\mathbb{P}^1)^k}C_{k+1}\Big(\frac{1}{(2\pi i)^{r_l}}\int_{{X_{\mu_n}\times (\mathbb{P}^1)^k}/{Z_{\mu_n}\times(\mathbb{P}^1)^k}}{\rm Td}_g\big(\overline{T}(f,h,h\circ f)\big){\rm ch}_g^0\big({\rm tr}_k\circ\lambda(\overline{E})\big)\\
,\log\mid z_1\mid^2,\cdots,\log\mid z_{k}\mid^2\Big)
\end{multline*}
give two chain homotopies of homological complexes between the maps ${l_{\mu_n}}_*\circ {\rm Td}_g(\overline{Tg})\bullet\big({\rm ch}_g(\cdot)\big)$ and ${l_{\mu_n}}_*\circ f_{\mu_n}^*{\rm Td}_g(\overline{Th}){\rm Td}_g(\overline{Tf})\bullet\big({\rm ch}_g(\cdot)\big)$. Moreover, $\Pi^{(3)}_k(\overline{E})$ and $\Pi^{(3')}_k(\overline{E})$ are homotopic to each other.
\end{prop}
\begin{proof}
The first statement follows from the projection formula, the second statement follows from a straightforward computation and the third follows from \cite[Remark 2.4, Lemma. 2.5]{T3}.
\end{proof}

Now we write $\Pi^f_k=\Pi'^f_k+\Pi''^f_k$ for the chain homotopy of the upper square in (\ref{atc-comp}) and $\Pi^h_k=\Pi'^h_k+\Pi''^h_k$ for the chain homotopy of the lower square in (\ref{atc-comp}). Then $\Pi^{(1)}_k+{h_{\mu_n}}_*\circ \big({\rm Td}_g(\overline{Th})\bullet\Pi^f_k\big)+\Pi^h_k\circ f_*-\Pi^{(3)}_k$ defines a chain homotopy between maps ${\rm ch}_g\circ l_*$ and ${l_{\mu_n}}_*\circ {\rm Td}_g(\overline{Tg})\bullet\big({\rm ch}_g(\cdot)\big)$. Suppose that the $\mu_n$-action on $Z$ is trivial, it's the main result of this subsection that the chain homotopy $\Pi^{(1)}_k+{h_{\mu_n}}_*\circ \big({\rm Td}_g(\overline{Th})\bullet\Pi^f_k\big)+\Pi^h_k\circ f_*-\Pi^{(3)}_k$ is homotopic to the chain homotopy $\Pi^l_k=\Pi'^l_k+\Pi''^l_k$ for the whole square in (\ref{atc-comp}). According to Proposition~\ref{303}, it is equivalent to show that $\Pi^{(1)}_k+{h_{\mu_n}}_*\circ \big({\rm Td}_g(\overline{Th})\bullet\Pi^f_k\big)+\Pi^h_k\circ f_*-\Pi^{(3')}_k$ is homotopic to $\Pi^l_k$.

To see this, we firstly denote by $H_{h\circ f}^{l}(\overline{E})$ the following emi-$2$-cube of hermitian bundles on $Z\times (\mathbb{P}^1)^k$
$$\xymatrix{h_*f_*\big({\rm tr}_k\circ\lambda(\overline{E})\big) \ar[r]^-{\rm Id} \ar[d]^-{\rm Id} & {\rm tr}_k\circ\lambda\big(h_*f_*(\overline{E})\big) \ar[r] \ar[d]^-{\rm Id} & 0 \ar[d]\\
l_*\big({\rm tr}_k\circ\lambda(\overline{E})\big) \ar[r]^-{\rm Id} \ar[d] & {\rm tr}_k\circ\lambda\big(l_*(\overline{E})\big) \ar[r] \ar[d] & 0 \ar[d]\\
0 \ar[r] & 0 \ar[r] & 0.}$$
Then, like before, we construct a hermitian bundle ${\rm tr}_2\big(H_{h\circ f}^{l}(\overline{E})\big)$ on $(\mathbb{P}^1)^{k+2}$ as the second transgression bundle of $H_{h\circ f}^{l}(\overline{E})$ such that it satisfies the following relations:
$${\rm tr}_2\big(H_{h\circ f}^{l}(\overline{E})\big)\mid_{Z\times \{0\}\times (\mathbb{P}^1)^{k+1}}={\rm tr}_{k+1}\Big(\lambda\big(H_{h\circ f}(\overline{E})\big)\Big),$$
$${\rm tr}_2\big(H_{h\circ f}^{l}(\overline{E})\big)\mid_{Z\times \{\infty\}\times (\mathbb{P}^1)^{k+1}}={\rm tr}_1\Big(H_{h\circ f}\big({\rm tr}_k\circ\lambda(\overline{E})\big)\Big),$$
$${\rm tr}_2\big(H_{h\circ f}^{l}(\overline{E})\big)\mid_{Z\times \mathbb{P}^1\times \{0\}\times (\mathbb{P}^1)^k}={\rm tr}_1\big(H(\overline{E},l_*)\big),$$
$${\rm tr}_2\big(H_{h\circ f}^{l}(\overline{E})\big)\mid_{Z\times \mathbb{P}^1\times \{\infty\}\times (\mathbb{P}^1)^k}={\rm tr}_1\big(H(\overline{E},h_*f_*)\big)$$
and
$${\rm tr}_2\big(H_{h\circ f}^{l}(\overline{E})\big)\mid_{Z\times (\mathbb{P}^1)^{i+1}\times \{0\}\times (\mathbb{P}^1)^{k-i}}={\rm tr}_2\big(H_{h\circ f}^{l}(\partial_i^0\overline{E})\big),$$
$${\rm tr}_2\big(H_{h\circ f}^{l}(\overline{E})\big)\mid_{Z\times (\mathbb{P}^1)^{i+1}\times \{\infty\}\times (\mathbb{P}^1)^{k-i}}={\rm tr}_2\big(H_{h\circ f}^{l}(\partial_i^{-1}\overline{E})\big)\oplus {\rm tr}_2\big(H_{h\circ f}^{l}(\partial_i^{1}\overline{E})\big)$$
for $i=1,\cdots,k$. We set
$$\mathbf{H}_{1,k}(\overline{E}):=\frac{(-1)^{k+2}}{2(k+2)!(2\pi i)^{k+2}}\int_{(\mathbb{P}^1)^{k+2}}{\rm ch}_g^0\Big({\rm tr}_2\big(H_{h\circ f}^{l}(\overline{E})\big)\Big)\wedge C_{k+2}(\log\mid z_1\mid^2,\cdots,\log\mid z_{k+2}\mid^2).$$
Then $\mathbf{H}_{1,k}$ vanishes on degenerate $k$-cubes, and we obtain a map $$\mathbf{H}_{1,k}: \widetilde{\Z}C_k^{(f, l)-\text{ac}}(X,\mu_n)\to \bigoplus_{p\geq0}D^{2p-k-2}(Z,p)_{R_n}$$ by linear extension. This map satisfies the following differential equation
\begin{multline*}
d_\mathcal{D}\circ\mathbf{H}_{1,k}(\overline{E})=\frac{(-1)^{k+2}}{2(k+2)!(2\pi i)^{k+2}}\int_{(\mathbb{P}^1)^{k+2}}{\rm ch}_g^0\Big({\rm tr}_2\big(H_{h\circ f}^{l}(\overline{E})\big)\Big)\wedge d_\mathcal{D}C_{k+2}(\log\mid z_1\mid^2,\cdots,\log\mid z_{k+2}\mid^2)\\
=\frac{(-1)^{k+2}}{2(k+2)!(2\pi i)^{k+2}}\int_{(\mathbb{P}^1)^{k+2}}{\rm ch}_g^0\Big({\rm tr}_2\big(H_{h\circ f}^{l}(\overline{E})\big)\Big)\wedge\big((-\frac{1}{2})(k+2)\sum_{j=1}^{k+2}(-1)^{j-1}(-4\pi i)(\delta_{z_j=\infty}-\delta_{z_j=0})\hfill\\
{\wedge C_{k+1}(\log\mid z_1\mid^2,\cdots,\widehat{\log\mid z_j\mid^2},\cdots,\log\mid z_{k+2}\mid^2)\big)}\\
=\mathbf{H}_{1,k-1}\circ d(\overline{E})-\frac{(-1)^{k+1}}{2(k+1)!(2\pi i)^{k+1}}\int_{(\mathbb{P}^1)^{k+1}}\Bigg[\bigg({\rm ch}_g^0\Big({\rm tr}_1\big(H(\overline{E},l_*)\big)\Big)-{\rm ch}_g^0\Big({\rm tr}_1\big(H(\overline{E},h_*f_*)\big)\Big)\bigg)\hfill\\
{\wedge C_{k+1}(\log\mid z_1\mid^2,\cdots,\log\mid z_{k+1}\mid^2)\Bigg]}\\
+\frac{(-1)^{k+1}}{2(k+1)!(2\pi i)^{k+1}}\int_{(\mathbb{P}^1)^{k+1}}\Bigg[\Bigg({\rm ch}_g^0\bigg({\rm tr}_{k+1}\Big(\lambda\big(H_{h\circ f}(\overline{E})\big)\Big)\bigg)-{\rm ch}_g^0\bigg({\rm tr}_1\Big(H_{h\circ f}\big({\rm tr}_k\circ\lambda(\overline{E})\big)\Big)\bigg)\Bigg)\\
{\wedge C_{k+1}(\log\mid z_1\mid^2,\cdots,\log\mid z_{k+1}\mid^2)\Bigg]}\\
=\mathbf{H}_{1,k-1}\circ d(\overline{E})-\Pi'^l_k(\overline{E})+\Pi^{(1)}_k(\overline{E})\hfill\\
+\frac{(-1)^{k+1}}{2(k+1)!(2\pi i)^{k+1}}\int_{(\mathbb{P}^1)^{k+1}}{\rm ch}_g^0\Big({\rm tr}_1\big(H(\overline{E},h_*f_*)\big)\Big){\wedge C_{k+1}(\log\mid z_1\mid^2,\cdots,\log\mid z_{k+1}\mid^2)}\\
-\frac{(-1)^{k+1}}{2(k+1)!(2\pi i)^{k+1}}\int_{(\mathbb{P}^1)^{k+1}}{\rm ch}_g^0\bigg({\rm tr}_1\Big(H_{h\circ f}\big({\rm tr}_k\circ\lambda(\overline{E})\big)\Big)\bigg)\wedge C_{k+1}(\log\mid z_1\mid^2,\cdots,\log\mid z_{k+1}\mid^2).
\end{multline*}

Secondly, we denote by $H'^l_{h\circ f}(\overline{E})$ the following emi-$2$-cube of hermitian bundles on $Z\times(\mathbb{P}^1)^k$
$$\xymatrix{h_*f_*\big({\rm tr}_k\circ\lambda(\overline{E})\big) \ar[r]^-{\rm Id} \ar[d]^-{\rm Id} & h_*{\rm tr}_k\circ\lambda\big(f_*(\overline{E})\big) \ar[r] \ar[d]^-{\rm Id} & 0 \ar[d]\\
h_*f_*\big({\rm tr}_k\circ\lambda(\overline{E})\big) \ar[r]^-{\rm Id} \ar[d] & {\rm tr}_k\circ\lambda\big(h_*f_*(\overline{E})\big) \ar[r] \ar[d] & 0 \ar[d]\\
0 \ar[r] & 0 \ar[r] & 0.}$$
Again, we construct a hermitian bundle ${\rm tr}_2\big(H'^l_{h\circ f}(\overline{E})\big)$ on $Z\times(\mathbb{P}^1)^{k+2}$ as the second transgression bundle of $H'^l_{h\circ f}(\overline{E})$ such that it satisfies the following relations:
$${\rm tr}_2\big(H'^l_{h\circ f}(\overline{E})\big)\mid_{Z\times \{0\}\times (\mathbb{P}^1)^{k+1}}={\rm tr}_1\big(H(f_*\overline{E},h_*)\big),$$
$${\rm tr}_2\big(H'^l_{h\circ f}(\overline{E})\big)\mid_{Z\times \{\infty\}\times (\mathbb{P}^1)^{k+1}}={\rm tr}_1\Big(h_*f_*\big({\rm tr}_k\circ\lambda(\overline{E})\big)\to h_*f_*\big({\rm tr}_k\circ\lambda(\overline{E})\big)\Big),$$
$${\rm tr}_2\big(H'^l_{h\circ f}(\overline{E})\big)\mid_{Z\times \mathbb{P}^1\times \{0\}\times (\mathbb{P}^1)^k}={\rm tr}_1\big(H(\overline{E},h_*f_*)\big),$$
$${\rm tr}_2\big(H'^l_{h\circ f}(\overline{E})\big)\mid_{Z\times \mathbb{P}^1\times \{\infty\}\times (\mathbb{P}^1)^k}={\rm tr}_1\big(h_*H(\overline{E},f_*)\big)$$
and
$${\rm tr}_2\big(H'^l_{h\circ f}(\overline{E})\big)\mid_{Z\times (\mathbb{P}^1)^{i+1}\times \{0\}\times (\mathbb{P}^1)^{k-i}}={\rm tr}_2\big(H'^l_{h\circ f}(\partial_i^0\overline{E})\big),$$
$${\rm tr}_2\big(H'^l_{h\circ f}(\overline{E})\big)\mid_{Z\times (\mathbb{P}^1)^{i+1}\times \{\infty\}\times (\mathbb{P}^1)^{k-i}}={\rm tr}_2\big(H'^l_{h\circ f}(\partial_i^{-1}\overline{E})\big)\oplus {\rm tr}_2\big(H'^l_{h\circ f}(\partial_i^{1}\overline{E})\big)$$
for $i=1,\cdots,k$. We set
$$\mathbf{H}_{2,k}(\overline{E}):=\frac{(-1)^{k+2}}{2(k+2)!(2\pi i)^{k+2}}\int_{(\mathbb{P}^1)^{k+2}}{\rm ch}_g^0\Big({\rm tr}_2\big(H'^l_{h\circ f}(\overline{E})\big)\Big)\wedge C_{k+2}(\log\mid z_1\mid^2,\cdots,\log\mid z_{k+2}\mid^2).$$
Then $\mathbf{H}_{2,k}$ defines a map $$\mathbf{H}_{2,k}: \widetilde{\Z}C_k^{(f, l)-\text{ac}}(X,\mu_n)\to \bigoplus_{p\geq0}D^{2p-k-2}(Z,p)_{R_n}$$ which satisfies the following differential equation
\begin{multline*}
d_\mathcal{D}\circ\mathbf{H}_{2,k}(\overline{E})=\frac{(-1)^{k+2}}{2(k+2)!(2\pi i)^{k+2}}\int_{(\mathbb{P}^1)^{k+2}}{\rm ch}_g^0\Big({\rm tr}_2\big(H'^l_{h\circ f}(\overline{E})\big)\Big)\wedge d_\mathcal{D}C_{k+2}(\log\mid z_1\mid^2,\cdots,\log\mid z_{k+2}\mid^2)\\
=\frac{(-1)^{k+2}}{2(k+2)!(2\pi i)^{k+2}}\int_{(\mathbb{P}^1)^{k+2}}{\rm ch}_g^0\Big({\rm tr}_2\big(H'^l_{h\circ f}(\overline{E})\big)\Big)\wedge\big((-\frac{1}{2})(k+2)\sum_{j=1}^{k+2}(-1)^{j-1}(-4\pi i)(\delta_{z_j=\infty}-\delta_{z_j=0})\hfill\\
{\wedge C_{k+1}(\log\mid z_1\mid^2,\cdots,\widehat{\log\mid z_j\mid^2},\cdots,\log\mid z_{k+2}\mid^2)\big)}\\
=\mathbf{H}_{2,k-1}\circ d(\overline{E})-\frac{(-1)^{k+1}}{2(k+1)!(2\pi i)^{k+1}}\int_{(\mathbb{P}^1)^{k+1}}\Bigg[\bigg({\rm ch}_g^0\Big({\rm tr}_1\big(H(\overline{E},h_*f_*)\big)\Big)-{\rm ch}_g^0\Big({\rm tr}_1\big(h_*H(\overline{E},f_*)\big)\Big)\bigg)\hfill\\
{\wedge C_{k+1}(\log\mid z_1\mid^2,\cdots,\log\mid z_{k+1}\mid^2)\Bigg]}\\
+\frac{(-1)^{k+1}}{2(k+1)!(2\pi i)^{k+1}}\int_{(\mathbb{P}^1)^{k+1}}{\rm ch}_g^0\Big({\rm tr}_1\big(H(f_*\overline{E},h_*)\big)\Big)\wedge C_{k+1}(\log\mid z_1\mid^2,\cdots,\log\mid z_{k+1}\mid^2)\\
=\mathbf{H}_{2,k-1}\circ d(\overline{E})+\Pi'^h_k(f_*\overline{E})\hfill\\
-\frac{(-1)^{k+1}}{2(k+1)!(2\pi i)^{k+1}}\int_{(\mathbb{P}^1)^{k+1}}{\rm ch}_g^0\Big({\rm tr}_1\big(H(\overline{E},h_*f_*)\big)\Big)\wedge C_{k+1}(\log\mid z_1\mid^2,\cdots,\log\mid z_{k+1}\mid^2)\\
+\frac{(-1)^{k+1}}{2(k+1)!(2\pi i)^{k+1}}\int_{(\mathbb{P}^1)^{k+1}}{\rm ch}_g^0\Big({\rm tr}_1\big(h_*H(\overline{E},f_*)\big)\Big)\wedge C_{k+1}(\log\mid z_1\mid^2,\cdots,\log\mid z_{k+1}\mid^2).
\end{multline*}

Thirdly, notice that the short exact sequence
$$\xymatrix{0 \ar[r] & h_*{\rm tr}_1\big(H(\overline{E},f_*)\big) \ar[r]^-{\rm Id} & {\rm tr}_1\big(h_*H(\overline{E},f_*)\big) \ar[r] & 0 \ar[r] & 0}$$
forms an emi-$1$-cube of hermitian bundles on $Z\times\mathbb{P}^1\times (\mathbb{P}^1)^k$, we denote it by $\widetilde{H}_{h\circ f}(\overline{E})$. Using the same construction as before, we construct a transgression bundle ${\rm tr}_1\big(\widetilde{H}_{h\circ f}(\overline{E})\big)$ on $Z\times\mathbb{P}^1\times \mathbb{P}^1\times (\mathbb{P}^1)^k$ satisfying
$${\rm tr}_1\big(\widetilde{H}_{h\circ f}(\overline{E})\big)\mid_{Z\times\{0\}\times (\mathbb{P}^1)^{k+1}}={\rm tr}_1\big(h_*H(\overline{E},f_*)\big),$$
$${\rm tr}_1\big(\widetilde{H}_{h\circ f}(\overline{E})\big)\mid_{Z\times\{\infty\}\times (\mathbb{P}^1)^{k+1}}=h_*{\rm tr}_1\big(H(\overline{E},f_*)\big),$$
$${\rm tr}_1\big(\widetilde{H}_{h\circ f}(\overline{E})\big)\mid_{Z\times\mathbb{P}^1\times \{0\}\times (\mathbb{P}^1)^k}={\rm tr}_1\big(h_*{\rm tr}_k\circ \lambda(f_*\overline{E})\to h_*{\rm tr}_k\circ \lambda(f_*\overline{E})\to 0\big),$$
$${\rm tr}_1\big(\widetilde{H}_{h\circ f}(\overline{E})\big)\mid_{Z\times\mathbb{P}^1\times \{\infty\}\times (\mathbb{P}^1)^k}={\rm tr}_1\big(h_*f_*{\rm tr}_k\circ \lambda(\overline{E})\to h_*f_*{\rm tr}_k\circ \lambda(\overline{E})\to 0\big)$$
and
$${\rm tr}_1\big(\widetilde{H}_{h\circ f}(\overline{E})\big)\mid_{Z\times(\mathbb{P}^1)^{i+1}\times \{0\}\times (\mathbb{P}^1)^{k-i}}={\rm tr}_1\big(\widetilde{H}_{h\circ f}(\partial_i^0\overline{E})\big),$$
$${\rm tr}_1\big(\widetilde{H}_{h\circ f}(\overline{E})\big)\mid_{Z\times(\mathbb{P}^1)^{i+1}\times \{\infty\}\times (\mathbb{P}^1)^{k-i}}={\rm tr}_1\big(\widetilde{H}_{h\circ f}(\partial_i^{-1}\overline{E})\big)\oplus {\rm tr}_1\big(\widetilde{H}_{h\circ f}(\partial_i^{1}\overline{E})\big)$$
for $i=1,\cdots,k$. So if we set
$$\mathbf{H}_{3,k}(\overline{E}):=\frac{(-1)^{k+2}}{2(k+2)!(2\pi i)^{k+2}}\int_{(\mathbb{P}^1)^{k+2}}{\rm ch}_g^0\Big({\rm tr}_1\big(\widetilde{H}_{h\circ f}(\overline{E})\big)\Big)\wedge C_{k+2}(\log\mid z_1\mid^2,\cdots,\log\mid z_{k+2}\mid^2),$$
it satisfies the differential equation
\begin{multline*}
d_\mathcal{D}\circ\mathbf{H}_{3,k}(\overline{E})=\frac{(-1)^{k+2}}{2(k+2)!(2\pi i)^{k+2}}\int_{(\mathbb{P}^1)^{k+2}}{\rm ch}_g^0\big({\rm tr}_1(\widetilde{H}_{h\circ f})\big)\wedge d_\mathcal{D}C_{k+2}(\log\mid z_1\mid^2,\cdots,\log\mid z_{k+2}\mid^2)\\
=\mathbf{H}_{3,k-1}\circ d(\overline{E})+\frac{(-1)^{k+1}}{2(k+1)!(2\pi i)^{k+1}}\int_{(\mathbb{P}^1)^{k+1}}{\rm ch}_g^0\Big({\rm tr}_1\big(h_*H(\overline{E},f_*)\big)\Big)\wedge
C_{k+1}(\log\mid z_1\mid^2,\cdots,\log\mid z_{k+1}\mid^2)\hfill\\
-\frac{(-1)^{k+1}}{2(k+1)!(2\pi i)^{k+1}}\int_{(\mathbb{P}^1)^{k+1}}{\rm ch}_g^0\big(h_*{\rm tr}_1\big(H(\overline{E},f_*)\big)\wedge C_{k+1}(\log\mid z_1\mid^2,\cdots,\log\mid z_{k+1}\mid^2)
\end{multline*}

Finally, we set
$$\mathbf{H}_{4,k}(\overline{E}):=\frac{(-1)^{k+2}}{(k+2)!(2\pi i)^{k+1}}\int_{(\mathbb{P}^1)^{k+1}}C_{k+2}\bigg(T_g\Big(h,h^{{\rm tr}_1\big(H(\overline{E},f_*)\big)}\Big),\log\mid z_1\mid^2,\cdots,\log\mid z_{k+1}\mid^2\bigg),$$
then it satisfies
\begin{multline*}
d_\mathcal{D}\circ\mathbf{H}_{4,k}(\overline{E})=\frac{(-1)^{k+2}}{(k+2)!(2\pi i)^{k+1}}\int_{(\mathbb{P}^1)^{k+1}}C_{k+2}\bigg(T_g\Big(h,h^{{\rm tr}_1\big(H(\overline{E},f_*)\big)}\Big),\log\mid z_1\mid^2,\cdots,\log\mid z_{k+1}\mid^2\bigg)\\
=\mathbf{H}_{4,k-1}\circ d(\overline{E})+\frac{(-1)^{k+1}}{2(k+1)!(2\pi i)^{k+1}}\int_{(\mathbb{P}^1)^{k+1}}{\rm ch}_g^0\Big(h_*{\rm tr}_1\big(H(\overline{E},f_*)\big)\Big)\wedge C_{k+1}(\log\mid z_1\mid^2,\cdots,\log\mid z_{k+1}\mid^2)\hfill\\
-\frac{(-1)^{k+1}}{2(k+1)!(2\pi i)^{k+1}}\int_{(\mathbb{P}^1)^{k+1}}\bigg(\frac{1}{(2\pi i)^{r_h}}\int_{Y_{\mu_n}}{\rm Td}_g(\overline{Th}){\rm ch}_g^0\Big({\rm tr}_1\big(H(\overline{E},f_*)\big)\Big)\bigg)\wedge C_{k+1}(\log\mid z_1\mid^2,\cdots,\log\mid z_{k+1}\mid^2)\hfill\\
-\frac{(-1)^{k+1}}{(k+1)!(2\pi i)^{k}}\int_{(\mathbb{P}^1)^{k}}C_{k+1}\Big(T_g\big(h,h^{{\rm tr}_k\circ\lambda(f_*\overline{E})}\big),\log\mid z_1\mid^2,\cdots,\log\mid z_{k}\mid^2\Big)\\
+\frac{(-1)^{k+1}}{(k+1)!(2\pi i)^{k}}\int_{(\mathbb{P}^1)^{k}}C_{k+1}\Big(T_g\big(h,h^{f_*{\rm tr}_k\circ\lambda(\overline{E})}\big),\log\mid z_1\mid^2,\cdots,\log\mid z_{k}\mid^2\Big)\\
=\mathbf{H}_{4,k-1}\circ d(\overline{E})-{h_{\mu_n}}_*\circ\big({\rm Td}_g(\overline{Th})\bullet \Pi'^f_k(\overline{E})\big)\hfill\\
+\frac{(-1)^{k+1}}{2(k+1)!(2\pi i)^{k+1}}\int_{(\mathbb{P}^1)^{k+1}}{\rm ch}_g^0\Big(h_*{\rm tr}_1\big(H(\overline{E},f_*)\big)\Big)\wedge C_{k+1}(\log\mid z_1\mid^2,\cdots,\log\mid z_{k+1}\mid^2)\\
-\Pi''^h_k(f_*\overline{E})+\frac{(-1)^{k+1}}{(k+1)!(2\pi i)^{k}}\int_{(\mathbb{P}^1)^{k}}C_{k+1}\Big(T_g\big(h,h^{f_*{\rm tr}_k\circ\lambda(\overline{E})}\big),\log\mid z_1\mid^2,\cdots,\log\mid z_{k}\mid^2\Big)
\end{multline*}

\begin{prop}\label{304}
Let notations and assumptions be as above, then the chain homotopy $\Pi^l_k=\Pi'^l_k+\Pi''^l_k$ is homotopic to $\Pi^{(1)}_k+{h_{\mu_n}}_*\circ \big({\rm Td}_g(\overline{Th})\bullet\Pi^f_k\big)+\Pi^h_k\circ f_*-\Pi^{(3')}_k$.
\end{prop}
\begin{proof}
Let $\overline{E}$ be a hermitian $k$-cube in $\widehat{\mathcal{P}}(X,\mu_n)$ which is $f$-acyclic and $l$-acyclic. Using the above differential equations concerning $\mathbf{H}_{i,k}$, we obtain that
\begin{multline*}
(\mathbf{H}_{1,k-1}+\mathbf{H}_{2,k-1}-\mathbf{H}_{3,k-1}-\mathbf{H}_{4,k-1})\circ d(\overline{E})-d_\mathcal{D}\circ (\mathbf{H}_{1,k}+\mathbf{H}_{2,k}-\mathbf{H}_{3,k}-\mathbf{H}_{4,k})(\overline{E})\\
=\Pi'^l_k(\overline{E})-\Pi^{(1)}_k(\overline{E})-\Pi'^h_k(f_*\overline{E})-\Pi''^h_k(f_*\overline{E})-{h_{\mu_n}}_*\circ\big({\rm Td}_g(\overline{Th})\bullet \Pi'^f_k(\overline{E})\big)\hfill\\
+\frac{(-1)^{k+1}}{2(k+1)!(2\pi i)^{k+1}}\int_{(\mathbb{P}^1)^{k+1}}{\rm ch}_g^0\bigg({\rm tr}_1\Big(H_{h\circ f}\big({\rm tr}_k\circ\lambda(\overline{E})\big)\Big)\bigg)\wedge C_{k+1}(\log\mid z_1\mid^2,\cdots,\log\mid z_{k+1}\mid^2)\\
+\frac{(-1)^{k+1}}{(k+1)!(2\pi i)^{k}}\int_{(\mathbb{P}^1)^{k}}C_{k+1}\Big(T_g\big(h,h^{f_*{\rm tr}_k\circ\lambda(\overline{E})}\big),\log\mid z_1\mid^2,\cdots,\log\mid z_{k}\mid^2\Big)
\end{multline*}
On the other hand, according to Theorem~\ref{301}, we have
\begin{multline*}
\frac{(-1)^{k+1}}{(k+1)!(2\pi i)^{k}}\int_{(\mathbb{P}^1)^{k}}C_{k+1}\Big(T_g\big(g,h^{{\rm tr}_k\circ\lambda(\overline{E})}\big),\log\mid z_1\mid^2,\cdots,\log\mid z_{k}\mid^2\Big)\\
-\frac{(-1)^{k+1}}{(k+1)!(2\pi i)^{k}}\int_{(\mathbb{P}^1)^{k}}C_{k+1}\Big(T_g\big(h,h^{f_*{\rm tr}_k\circ\lambda(\overline{E})}\big),\log\mid z_1\mid^2,\cdots,\log\mid z_{k}\mid^2\Big)-{h_{\mu_n}}_*\circ\big({\rm Td}_g(\overline{Th})\bullet \Pi''^f_k(\overline{E})\big)\hfill\\
=\frac{(-1)^{k+1}}{2(k+1)!(2\pi i)^{k+1}}\int_{(\mathbb{P}^1)^{k+1}}{\rm ch}_g^0\bigg({\rm tr}_1\Big(H_{h\circ f}\big({\rm tr}_k\circ\lambda(\overline{E})\big)\Big)\bigg)\wedge C_{k+1}(\log\mid z_1\mid^2,\cdots,\log\mid z_{k+1}\mid^2)-\Pi^{(3')}_k(\overline{E})\hfill\\
+\frac{(-1)^{k+1}}{(k+1)!(2\pi i)^k}\int_{(\mathbb{P}^1)^k}C_{k+1}\Big(d_\mathcal{D}\Delta\big(f,h,\omega^X,\omega^Y,{\rm tr}_k\circ\lambda(\overline{E})\big),\log\mid z_1\mid^2,\cdots,\log\mid z_{k}\mid^2\Big).
\end{multline*}

We then formally define a product $$C_{k+1}\Big(\Delta\big(f,h,\omega^X,\omega^Y,{\rm tr}_k\circ\lambda(\overline{E})\big),\log\mid z_1\mid^2,\cdots,\log\mid z_{k}\mid^2\Big)$$ in the same way as (\ref{NewC}), and we set
$$\Delta_k(\overline{E})=\frac{(-1)^{k}}{(k+1)!(2\pi i)^k}\int_{(\mathbb{P}^1)^k}C_{k+1}\Big(\Delta\big(f,h,\omega^X,\omega^Y,{\rm tr}_k\circ\lambda(\overline{E})\big),\log\mid z_1\mid^2,\cdots,\log\mid z_{k}\mid^2\Big).$$
It is readily checked by Lemma~\ref{deltac} that
\begin{align*}
&\Delta_{k-1}(d\overline{E})-d_\mathcal{D}\Delta_k(\overline{E})\\
=&\frac{(-1)^{k+1}}{(k+1)!(2\pi i)^k}\int_{(\mathbb{P}^1)^k}C_{k+1}\Big(d_\mathcal{D}\Delta\big(f,h,\omega^X,\omega^Y,{\rm tr}_k\circ\lambda(\overline{E})\big),\log\mid z_1\mid^2,\cdots,\log\mid z_{k}\mid^2\Big).
\end{align*}

Combing all the above computations, we finally get
\begin{multline*}
(\mathbf{H}_{1,k-1}+\mathbf{H}_{2,k-1}-\mathbf{H}_{3,k-1}-\mathbf{H}_{4,k-1}+\Delta_{k-1})(d\overline{E})-d_\mathcal{D}
(\mathbf{H}_{1,k}+\mathbf{H}_{2,k}-\mathbf{H}_{3,k}-\mathbf{H}_{4,k}+\Delta_k)(\overline{E})\\
=\Pi'^l_k(\overline{E})-\Pi^{(1)}_k(\overline{E})-\Pi'^h_k(f_*\overline{E})-\Pi''^h_k(f_*\overline{E})+\Pi^{(3')}_k(\overline{E})+\Pi''^l_k(\overline{E})\hfill\\
-{h_{\mu_n}}_*\circ\big({\rm Td}_g(\overline{Th})\bullet \Pi'^f_k(\overline{E})\big)-{h_{\mu_n}}_*\circ\big({\rm Td}_g(\overline{Th})\bullet \Pi''^f_k(\overline{E})\big)\\
=\Pi^l_k(\overline{E})-\Big(\Pi^{(1)}_k(\overline{E})+{h_{\mu_n}}_*\circ \big({\rm Td}_g(\overline{Th})\bullet\Pi^f_k(\overline{E})\big)+\Pi^h_k(f_*\overline{E})-\Pi^{(3')}_k(\overline{E})\Big)\hfill
\end{multline*}
So we are done.
\end{proof}

\begin{cor}\label{305}
Let $f: X\to Y$, $h: Y\to Z$ and $l: X\to Z$ be three equivariant morphisms between regular $\mu_n$-projective schemes, which are all smooth over the generic fibres. Assume that $l=h\circ f$ and that the $\mu_n$-action on $Z$ is trivial. Then the direct image map $l_*$ is equal to the composition $h_*\circ f_*$ from $\widehat{K}_m(X,\mu_n)_{\Q}$ to $\widehat{K}_m(Z,\mu_n)_{\Q}$ for any $m\geq 1$.
\end{cor}

\section{The Lefschetz-Riemann-Roch theorem}

\subsection{The statement}
In order to formulate the Lefschetz-Riemann-Roch theorem for higher equivariant arithmetic K-groups, we need to introduce the equivariant $R$-genus due to Bismut. Let $X$ be a $\mu_n$-equivariant smooth algebraic variety over $\C$, and let $\overline{E}$ be a $\mu_n$-equivariant hermitian vector bundle on $X$. For $\zeta\in \mu_n(\C)$ and $s>1$, we consider the following Lerch zeta function
$$L(\zeta,s)=\sum_{k=1}^\infty\frac{\zeta^k}{k^s}$$
and its meromorphic continuation to the whole complex plane.
Define a formal power series in the variable $x$ as
$$\widetilde{R}(\zeta,x):=\sum_{n=0}^\infty\big(\frac{\partial L}{\partial s}(\zeta,-n)+L(\zeta,-n)\sum_{j=1}^n\frac{1}{2j}\big)\frac{x^n}{n!}.$$

\begin{defn}\label{401}
The Bismut's equivariant $R$-genus of an equivariant hermitian vector bundle $\overline{E}$ with
$\overline{E}\mid_{X_{\mu_n}}=\sum_{\zeta\in \mu_n(\C)}\overline{E}_\zeta$ is defined as
$$R_g(\overline{E}):=\sum_{\zeta\in \mu_n(\C)}\big({\rm Tr}\widetilde{R}(\zeta,-\Omega^{\overline{E}_\zeta})-{\rm
Tr}\widetilde{R}(1/\zeta,\Omega^{\overline{E}_\zeta})\big),$$
where $\Omega^{\overline{E}_\zeta}$ is the curvature form associated to $\overline{E}_\zeta$.
\end{defn}

Now, let $X$ be a regular $\mu_n$-projective arithmetic scheme over an arithmetic ring $(D,\Sigma,F_\infty)$ and we construct a naive commutative diagram of homological complexes
\begin{align}\label{genus}
\xymatrix{ \widetilde{\Z}C_*(X,\mu_n) \ar[r]^-{{\rm ch}_g} \ar[d]^-{0} & \bigoplus_{p\geq0}D^{2p-*}(X_{\mu_n},p)_{R_n} \ar[d]^-{0} \\
\widetilde{\Z}C_*(X,\mu_n) \ar[r]^-{{\rm ch}_g} & \bigoplus_{p\geq0}D^{2p-*}(X_{\mu_n},p)_{R_n}}
\end{align}
where $0$ stands for the zero map. Let $\overline{N}$ be a $\mu_n$-equivariant hermitian vector bundle on $X$, we shall formally regard the $R$-genus $R_g(\overline{N})$ as an element in $\bigoplus_{p\geq0}D^{2p-1}(X,p)$. It is a $d$-closed form. Denote by $p_0$ the projection from $X\times (\mathbb{P}^1)^\cdot$ to $X$.
For any hermitian $k$-cube $\overline{E}$ in $\widehat{\mathcal{P}}(X,\mu_n)$, we set
$$\Pi_R(\overline{E})=\frac{(-1)^{k}}{(k+1)!(2\pi i)^k}\int_{(\mathbb{P}^1)^k}C_{k+1}\Big(R_g(p_0^*\overline{N}){\rm ch}_g^0\big({\rm tr}_k\circ\lambda(\overline{E})\big),\log\mid z_1\mid^2,\cdots,\log\mid z_{k}\mid^2\Big).$$
It is clear that $\Pi_R(\overline{E})$ extends to be a map $\Pi_R: \widetilde{\Z}C_k(X,\mu_n)\to \bigoplus_{p\geq0}D^{2p-k-1}(X_{\mu_n},p)_{R_n}$ which provides a chain homotopy of the square (\ref{genus}). Therefore, we get an endomorphism of $\widehat{K}_m(X,\mu_n)$ for any $m\geq1$. This endomorphism will be denoted by $\otimes R_g(\overline{N})$.

Again, by \cite[Remark 2.4, Lemma 2.5]{T3}, the chain homotopy $\Pi_R$ is homotopic to the chain homotopy $\Pi'_R$ defined by
$$\Pi'_R(\overline{E})=\frac{(-1)^{k+1}}{2k!(2\pi i)^k}\int_{(\mathbb{P}^1)^k}R_g(p_0^*\overline{N})\bullet{\rm ch}_g^0\big({\rm tr}_k\circ\lambda(\overline{E})\big)\bullet C_{k}(\log\mid z_1\mid^2,\cdots,\log\mid z_{k}\mid^2),$$
and hence is homotopic to $-R_g(\overline{N})\bullet{\rm ch}_g(\overline{E})$ by the projection formula. Let $(x,\alpha)$ be an element in $\widehat{K}_m(X,\mu_n)_\Q$, then $dx=0$ and ${\rm ch}_g(x)$ is a $d_\mathcal{D}$-closed form. Let $(0,\alpha)$ and $(0,\alpha')$ be two elements in $\widehat{K}_m(X,\mu_n)_\Q$, then $(0,\alpha)=(0,\alpha')$ if $\alpha$ and $\alpha'$ have the same cohomology class in $\bigoplus_{p\geq0}H_\mathcal{D}^*\big(X_{\mu_n},\mathbb{R}(p)\big)_{R_n}$. Notice that the product $\bullet$ on the Deligne-Beilinson complex induces the product on the real Deligne-Beilinson cohomology. Then, modulo torsion, the endomorphism $\otimes R_g(\overline{N})$ is independent of the choice of the metric on $N$ and it can be written as $\otimes R_g(N)$.

Assume that $\rho$ is any prime ideal in $R(\mu_n):=K_0({\rm Spec}\mathbb{Z},\mu_n)\cong \Z[T]/{(1-T^n)}$ which doesn't contain the elements $1-T^k$ for $k=1,\ldots,n-1$. For instance, $\rho$ can be chosen to be the kernel of the natural morphism $\Z[T]/{(1-T^n)}\to \Z[T]/{(\Phi_n)}$ where $\Phi_n$ stands for the $n$-th cyclotomic polynomial. Let $X_{\mu_n}$ be the fixed point subscheme of $X$, and let $\overline{N}_{X/{X_{\mu_n}}}$ be the normal bundle of $X_{\mu_n}$ in $X$ with some $\mu_n$-invariant hermitian metric. We set
$$\Lambda_R:=\big({\rm Id}-\otimes R_g(N_{X/{X_{\mu_n}}})\big)\circ\otimes \lambda_{-1}^{-1}(\overline{N}_{X/{X_{\mu_n}}}^\vee),$$
it is a well-defined endomorphism of $\widehat{K}_m(X_{\mu_n},\mu_n)_\rho\otimes\Q$. Then the arithmetic Lefschetz-Riemann-Roch theorem for higher equivariant arithmetic K-groups can be formulated as follows.

\begin{thm}\label{402}(arithmetic Lefschetz-Riemann-Roch)
Let $f: X\to Y$ be an equivariant morphism between two regular $\mu_n$-projective arithmetic schemes, which is smooth over the generic fibre. Suppose that the $\mu_n$-action on the base $Y$ is trivial. Then, for any $m\geq1$, the following diagram
\begin{align*}
\xymatrix{
\widehat{K}_m(X,\mu_n) \ar[rr]^-{\Lambda_R\circ \tau} \ar[d]_{f_*} && \widehat{K}_m(X_{\mu_n},\mu_n)_\rho\otimes \Q \ar[d]^{{f_{\mu_n}}_*} \\
\widehat{K}_m(Y,\mu_n) \ar[rr]^-{\tau} && \widehat{K}_m(Y,\mu_n)_\rho\otimes\Q}
\end{align*}
where $\tau$ is the restriction map, is commutative.
\end{thm}

The proof of Theorem~\ref{402} will be given in next two subsections.

\subsection{Arithmetic K-theoretic form of Bismut-Ma's immersion formula}
Let $Y\hookrightarrow X$ be a $\mu_n$-equivariant closed immersion of regular $\mu_n$-projective arithmetic schemes over $(D,\Sigma,F_{\infty})$. In \cite[Section 4]{T3}, we have proved an arithmetic purity theorem $$\widehat{K}_m(Y,\mu_n)\cong \widehat{K}_{Y,m}(X,\mu_n)$$ for any integer $m\geq1$. As a byproduct, we get an embedding morphism $\widehat{K}_m(Y,\mu_n)\to \widehat{K}_m(X,\mu_n)$. This embedding morphism is realized by constructing an explicit chain homotopy of the square
\begin{align}\label{bcc}
\xymatrix{ \widetilde{\Z}C_*(Y,\mu_n) \ar[r]^-{{\rm ch}_g} \ar[d]^-{i_*} & \bigoplus_{p\geq0}{'D}^{2p-*}(Y_{\mu_n},p)_{R_n} \ar[d]^-{{i_{\mu_n}}_!\circ {\rm Td}_g^{-1}(\overline{N}_{X/Y})\bullet(\cdot)} \\
\widetilde{\Z}C_*(P,\mu_n) \ar[r]^-{{\rm ch}_g} & \bigoplus_{p\geq0}{'D}^{2p-*}(P_{\mu_n},p)_{R_n},}
\end{align}
where ${'D}^{2p-*}(\cdot,p)$ stands for the Deligne complex of currents computing the Deligne homology groups, $({i_{\mu_n}}_!T)(\eta)=T(i_{\mu_n}^*\eta)$ for a current $T$ and a test form $\eta$, $i: Y\hookrightarrow P:=\mathbb{P}(N_{X/Y}\oplus\mathcal{O}_Y)$ is the associated zero section embedding with projection $\pi: P\to Y$ and
$$i_*: \widetilde{\Z}C_*(Y,\mu_n)\to \widetilde{\Z}C_*(P,\mu_n)$$ is the complex morphism defined by sending a hermitian cube $\overline{E}$ to $\sum_{j=0}^n(-1)^j\overline{Q}^\vee\otimes \pi^*\overline{E}$ provided the Koszul resolution
$$K(\overline{E},\overline{N}_{X/Y}):\quad 0\to \wedge^n\overline{Q}^\vee\otimes \pi^*\overline{E} \to \cdots \to \wedge\overline{Q}^\vee\otimes \pi^*\overline{E}\to \pi^*\overline{E}\to i_*\overline{E}\to 0.$$
For any hermitian $k$-cube $\overline{E}$, one chain homotopy $\mathbf{H}_k(\overline{E})$ of (\ref{bcc}) is given by the formula
$$\mathbf{H}_k(\overline{E})=T_g\big(K(\overline{\mathcal{O}}_Y,\overline{N}_{X/Y})\big)\bullet{\rm ch}_g(\pi^*\overline{E})$$
where $T_g\big(K(\overline{\mathcal{O}}_Y,\overline{N}_{X/Y})\big)$ is the equivariant Bott-Chern singular current associated to the Koszul resolution which satisfies
$$d_\mathcal{D}T_g\big(K(\overline{\mathcal{O}}_Y,\overline{N}_{X/Y})\big)=\sum_{j=1}^n(-1)^j{\rm ch}_g(\wedge^n\overline{Q}^\vee)-{i_{\mu_n}}_!\big({\rm ch}_g(\overline{O}_Y){\rm Td}_g(\overline{N}_{X/Y})\big).$$ 
For more details the reader is referred to \cite[Section 4.2]{T3}.

It is clear that if we choose another resolution
$$0\to \overline{F}_n \to \cdots\to \overline{F}_1\to \overline{F}_0\to i_*\overline{O}_Y\to 0$$
with respect to the zero section embedding $i: Y\hookrightarrow \mathbb{P}(N_{X/Y}\oplus\mathcal{O}_Y)$ such that the metrics on $F.$ satisfy the Bismut's assumption (A), we may construct a different homotopy of (\ref{bcc}) and we shall get a different embedding morphism $i_*: \widehat{K}_m(Y,\mu_n)\to \widehat{K}_m(P,\mu_n)$. Our first result in this subsection is the following.

\begin{prop}\label{403}
The embedding morphism over rational arithmetic K-groups $$i_*: \widehat{K}_m(Y,\mu_n)_{\Q}\to \widehat{K}_m(P,\mu_n)_{\Q}$$ is independent of the choice of the resolution of $i_*\overline{\mathcal{O}}_Y$ on $\mathbb{P}(N_{X/Y}\oplus\mathcal{O}_Y)$ which satisfies the Bismut's assumption (A).
\end{prop}
\begin{proof}
Since any two resolutions of $i_*\overline{\mathcal{O}}_Y$ on $\mathbb{P}(N_{X/Y}\oplus\mathcal{O}_Y)$ are dominated by a third one, we may assume that $\overline{F}.$ and $\wedge^.\overline{Q}^\vee$ fit into the following diagram
$$\xymatrix{ & 0 \ar[d] & 0 \ar[d] & 0 \ar[d] & \\
0 \ar[r] & \overline{A}_n \ar[r] \ar[d] & \overline{F}_n \ar[r] \ar[d] & \wedge^n\overline{Q}^\vee \ar[r] \ar[d] & 0\\
 & \vdots \ar[d] & \vdots \ar[d] & \vdots \ar[d] & \\
0 \ar[r] & \overline{A}_1 \ar[r] \ar[d] & \overline{F}_1 \ar[r] \ar[d] & \wedge\overline{Q}^\vee \ar[r] \ar[d] & 0\\
0 \ar[r] & \overline{A}_0 \ar[r] \ar[d] & \overline{F}_0 \ar[r] \ar[d] & \overline{\mathcal{O}}_P \ar[r] \ar[d] & 0\\
 & 0 \ar[r]  & i_*\overline{\mathcal{O}}_Y \ar[r]  &  i_*\overline{\mathcal{O}}_Y & }$$
where $\overline{A}.$ is an exact sequence of hermitian vector bundles on $P$. We endow $A.$ with the metrics coming form $\overline{F}.$ via the natural inclusion. We split $\overline{A}.$ into a family of short exact sequence of hermitian bundles from $j=1$ to $n-1$
$$\xymatrix{\chi_j: \quad 0 \ar[r] & \Ker d_j \ar[r] & \overline{A}_j \ar[r]^-{d_j} & \Ker d_{j-1} \ar[r] & 0}.$$
Moreover, we denote by $\varepsilon_j$ the short exact sequence $$\xymatrix{0 \ar[r] & \overline{A}_j \ar[r] & \overline{F}_j \ar[r] & \wedge^j\overline{Q}^\vee \ar[r] & 0}$$ from $j=0$ to $n$. Write $i_*$ (resp. $i'_*$) for the morphism $\widetilde{\Z}C_*(Y,\mu_n)\to \widetilde{\Z}C_*(P,\mu_n)$ with respect to the Koszul resolution $K(\overline{\mathcal{O}}_Y,\overline{N}_{X/Y})$ (resp. the resolution $\overline{F}.$). Then, for any hermitian $k$-cube $\overline{E}$ on $Y$, the assignment $$H_i(\overline{E}):=\sum_{j=0}^n(-1)^j\varepsilon_j\otimes \pi^*\overline{E}+\sum_{j=1}^{n-1}(-1)^j\chi_j\otimes\pi^*\overline{E}\in \widetilde{\Z}C_{k+1}(P,\mu_n)$$ provides a chain homotopy between $i'_*$ and $i_*$. Consequently, the formula
$$\mathbf{H}^{(1)}_k(\overline{E})=\big(\sum_{j=0}^n(-1)^j{\rm ch}_g(\varepsilon_j)+\sum_{j=1}^{n-1}(-1)^j{\rm ch}_g(\chi_j)\big){\rm ch}_g(\pi^*\overline{E})$$ defines a chain homotopy between ${\rm ch}_g\circ i'_*$ and ${\rm ch}_g\circ i_*$. We claim that there exists a homotopy of chain homotopies between $\mathbf{H}'_k(\overline{E})$ and $\mathbf{H}^{(1)}_k(\overline{E})+\mathbf{H}_k(\overline{E})$. 

In fact, according to \cite[Theorem 3.14, Corollary 3.10]{KR1}, we have 
$$-\sum_{j=1}^{n-1}(-1)^j{\rm ch}_g(\chi_j)+T_g(\overline{F}.)-T_g\big(K(\overline{\mathcal{O}}_Y,\overline{N}_{X/Y})\big)=\sum_{j=0}^n(-1)^j{\rm ch}_g(\varepsilon_j)$$ 
up to ${\Im d_\mathcal{D}}$. We fix an element $\Delta$ such that 
$$d_\mathcal{D}\Delta=\sum_{j=0}^n(-1)^j{\rm ch}_g(\varepsilon_j)+\sum_{j=1}^{n-1}(-1)^j{\rm ch}_g(\chi_j)-T_g(\overline{F}.)+T_g\big(K(\overline{\mathcal{O}}_Y,\overline{N}_{X/Y})\big)$$ and set 
$$\widetilde{\mathbf{H}}_k(\overline{E}):=\Delta\bullet{\rm ch}_g(\pi^*\overline{E}).$$
Then 
$$d_\mathcal{D}\circ \widetilde{\mathbf{H}}_k(\overline{E})=\mathbf{H}^{(1)}_k(\overline{E})+\mathbf{H}_k(\overline{E})-\mathbf{H}'_k(\overline{E})+\widetilde{\mathbf{H}}_{k-1}\circ d(\overline{E}).$$
So we are done.
\end{proof}

Notice that the product $P\times (\mathbb{P}^1)^\cdot$ can be identified with the projective space bundle over $Y\times (\mathbb{P}^1)^\cdot$ with respect to the vector bundle $p_0^*N_{X/Y}$, and
$$0\to p_0^*\wedge^n\overline{Q}^\vee \to \cdots \to p_0^*\wedge\overline{Q}^\vee\to \overline{\mathcal{O}}_{P\times (\mathbb{P}^1)^\cdot}\to i_*\overline{\mathcal{O}}_{Y\times (\mathbb{P}^1)^\cdot}\to 0$$
is the Koszul resolution so that the corresponding Bott-Chern singular current is the pullback $p_0^*T_g\big(K(\overline{\mathcal{O}}_Y,\overline{N}_{X/Y})\big)$. We shall still write it as $T_g(K(\overline{\mathcal{O}}_Y,\overline{N}_{X/Y})$ for the sake of simplicity. Then, like before, by the projection formula and \cite[Remark 2.4, Lemma 2.5]{T3}, $\mathbf{H}_k(\overline{E})$ is homotopic to the following chain homotopy
$$\frac{(-1)^{k+1}}{(k+1)!(2\pi i)^k}\int_{(\mathbb{P}^1)^k}C_{k+1}\Big(T_g\big(K(\overline{\mathcal{O}}_Y,\overline{N}_{X/Y})\big)\bullet {\rm ch}_g^0\big({\rm tr}_k\circ \lambda(\overline{E})\big),\log\mid z_1\mid^2,\cdots,\log\mid z_{k}\mid^2\Big),$$
which will be still denoted by $\mathbf{H}_k(\overline{E})$.

Now, let us recall the Bismut-Ma's immersion formula which relates analytic torsion forms and the Bott-Chern singular current. Let $X$ be a smooth $\mu_n$-equivariant algebraic variety over $\C$ and let $i: Y\hookrightarrow X$ be an equivariant closed smooth subvariety. Let $S$ be a smooth algebraic variety with trivial $\mu_n$-action, and let $f: Y\rightarrow S$, $l: X\rightarrow S$ be two equivariant proper smooth morphisms such that $f=l\circ i$. Assume that $\overline{\eta}$ is an equivariant hermitian bundle on $Y$ and $\overline{\xi}.$ is a complex of equivariant hermitian bundles on $X$ which provides a resolution of $i_*\overline{\eta}$ such that the metrics on $\xi.$ satisfy the Bismut's assumption (A). Let $\omega^Y$, $\omega^X$ be two K\"{a}hler fibrations on $f$ and on $l$ respectively. We shall assume that $\omega^Y$ is the pull-back of
$\omega^X$ so that the K\"{a}hler metric on $Y$ is induced by the K\"{a}hler metric on $X$. Consider the following exact sequence
$$\overline{\mathcal{N}}:\quad 0\to \overline{Tf}\to
\overline{Tl}\mid_Y\to \overline{N}_{X/Y}\to 0$$
where $N_{X/Y}$ is endowed with the quotient metric. Denote by ${\rm Td}_g(\overline{\mathcal{N}})=\Phi^{-1}\big(\frac{\widetilde{{\rm Td}}_g(\overline{\mathcal{N}})}{2}\big)$ (see Section 5.3 in the Appendix) the equivariant secondary Todd form of $\overline{\mathcal{N}}$ which satisfies the identity
$$d_\mathcal{D}{\rm Td}_g(\overline{\mathcal{N}})={\rm Td}_g(Tl\mid_Y,h^{Tl})-{\rm Td}_g(Tf,h^{Tf}){\rm Td}_g(\overline{N}_{X/Y}).$$

We suppose that in the resolution $\xi.$,
$\xi_j$ are all $l-$acyclic and moreover $\eta$ is $f-$acyclic. Denote by $h^{H(\xi.)}$ the hermitian metric on $f_*\eta$ corresponding to the $L^2$-metric on the hypercohomology of $\xi.$ over the fibre of $l: X\rightarrow S$ (see Section 5.3 in the Appendix). By an easy argument of long exact sequence, we have the following
exact sequence of hermitian vector bundles on $S$
$$\overline{\Xi}:\quad 0\to l_*(\overline{\xi}_m)\to l_*(\overline{\xi}_{m-1})\to\ldots\to l_*(\overline{\xi}_0)\to \big(f_*\eta, h^{H(\xi.)}\big) \to 0.$$
We may split $\Xi.$ into a family of short exact sequence of hermitian bundles from $j=1$ to $m$
$$\xymatrix{\chi_j: \quad 0 \ar[r] & \Ker d_j \ar[r] & \overline{\Xi}_j \ar[r]^-{d_j} & \Ker d_{j-1} \ar[r] & 0}$$
such that the kernel of every map $d_{j-1}$ for $j=2,\ldots,m$ carries the metric induced by $\overline{\Xi}_j$ and $\Ker d_0=\overline{\Xi}_0=\big(f_*\eta, h^{H(\xi.)}\big), \Ker d_m=\overline{\Xi}_{m+1}=l_*(\overline{\xi}_m)$. We regard $\chi_j$ as a hermitian $1$-cube on $S$ and we set ${\rm ch}_g(\overline{\Xi}.)=\sum_{j=1}^m(-1)^j{\rm ch}_g(\chi_j)$. Then it satisfies the differential equation
$$d_\mathcal{D}{\rm ch}_g(\overline{\Xi}.)={\rm ch}_g\big(f_*\eta, h^{H(\xi.)}\big)-\sum_{j=0}^m{\rm ch}_g\big(l_*(\overline{\xi}_j)\big).$$
Set ${\rm ch}_g(\overline{\Xi}., f_*\overline{\eta}):={\rm ch}_g(\overline{\Xi}.)+{\rm ch}_g\big(f_*\eta, h^{H(\xi.)}, f_*h^\eta\big)$, it satisfies the differential equation
$$d_\mathcal{D}{\rm ch}_g(\overline{\Xi}., f_*\overline{\eta})={\rm ch}_g(f_*\overline{\eta})-\sum_{j=0}^m{\rm ch}_g\big(l_*(\overline{\xi}_j)\big).$$

With some abuse of notations, we still use $\overline{\Xi}$ to denote the long exact sequence
$$0\to l_*(\overline{\xi}_m)\to l_*(\overline{\xi}_{m-1})\to\ldots\to l_*(\overline{\xi}_0)\to f_*\overline{\eta} \to 0$$
and identify ${\rm ch}_g(\overline{\Xi}.)$ with ${\rm ch}_g(\overline{\Xi}., f_*\overline{\eta})$.

\begin{thm}\label{404}(Immersion formula)
Let notations and assumptions be as above. Then the following identity holds in $\bigoplus_{p\geq0}\big(D^{2p-1}(S,p)/{\Im d_\mathcal{D}}\big)$.
\begin{multline*}
\sum_{i=0}^m(-1)^iT_g(\omega^X,h^{\xi_i})-T_g(\omega^Y,h^\eta)+{\rm
ch}_g(\overline{\Xi}.)\\=-\frac{1}{(2\pi i)^{r_l}}\int_{X_{\mu_n}/S}{\rm
Td}_g(\overline{Tl})T_g(\overline{\xi}.)-\frac{1}{(2\pi i)^{r_f}}\int_{Y_{\mu_n}/S}{\rm
Td}_g(\overline{\mathcal{N}}){{\rm
Td}_g^{-1}(\overline{N}_{X/Y})}{\rm ch}_g(\overline{\eta})\hfill\\
+\frac{1}{(2\pi i)^{r_f}}\int_{Y_{\mu_n}/S}{\rm
Td}_g(\overline{Tf})R_g(\overline{N}_{X/Y}){\rm ch}_g(\overline{\eta})
\end{multline*}
where $r_f$ and $r_l$ are the relative dimensions of $Y_{\mu_n}/S$ and of $X_{\mu_n}/S$ respectively.
\end{thm}
\begin{proof}
This is a translation of \cite[Theorem 0.1 and 0.2]{BM} (see also Theorem~\ref{A30} in the Appendix). 
\end{proof}

With the same notations as in Remark~\ref{explain_tt} and Theorem~\ref{A3} in the Appendix, we set 
$$\Delta(f,l,i_*\overline{\eta},\overline{\xi}.):=-\Phi^{-1}\big(\frac{\Delta^0(f,l,i_*\overline{\eta},\overline{\xi}.)+\Delta_0(f,l,i_*\overline{\eta},\overline{\xi}.)}{2}\big)-\Delta(\overline{\Xi}.).$$ 
Then $d_\mathcal{D}\Delta(f,l,i_*\overline{\eta},\overline{\xi}.)$ measures the difference
\begin{multline*}
\sum_{i=0}^m(-1)^iT_g(\omega^X,h^{\xi_i})-T_g(\omega^Y,h^\eta)+{\rm
ch}_g(\overline{\Xi}.)+\frac{1}{(2\pi i)^{r_l}}\int_{X_{\mu_n}/S}{\rm
Td}_g(\overline{Tl})T_g(\overline{\xi}.)\\
+\frac{1}{(2\pi i)^{r_f}}\int_{Y_{\mu_n}/S}{\rm
Td}_g(\overline{\mathcal{N}}){{\rm
Td}_g^{-1}(\overline{N}_{X/Y})}{\rm ch}_g(\overline{\eta})-\frac{1}{(2\pi i)^{r_f}}\int_{Y_{\mu_n}/S}{\rm
Td}_g(\overline{Tf})R_g(\overline{N}_{X/Y}){\rm ch}_g(\overline{\eta})
\end{multline*}
in Theorem~\ref{404}. Let us go back to the same situation described before Lemma~\ref{delta} and assume that the following diagrams
$$\xymatrix{Y\times Z \ar[rd]_-{f_Z} \ar[rr]^-{i_Z} && X\times Z \ar[ld]^-{l_Z} \\ & S\times Z & }\quad \text{and}\quad
\xymatrix{Y\times Z_1 \ar[rd]_-{f_{Z_1}} \ar[rr]^-{i_{Z_1}} && X\times Z_1 \ar[ld]^-{l_{Z_1}} \\ & S\times Z_1 & }$$
are obtained by smooth base changes. Then $Y\times Z$ and $X\times Z_1$ intersect transversely along $Y\times Z_1$ and the singular currents can be pulled back.

\begin{lem}\label{deltai}
The restriction of $\Delta(f_Z,l_Z,{i_{Z}}_*\overline{\eta},\overline{\xi}.)$ over $S\times Z_1$ is equal to the differential form $\Delta(f_{Z_1},l_{Z_1},{i_{Z_1}}_*\overline{\eta}\mid_{Y\times Z_1},\overline{\xi}.\mid_{X\times Z_1})$.
\end{lem}
\begin{proof}
This is a consequence of Theorem~\ref{A3} in the Appendix.
\end{proof}

\begin{prop}\label{405}
Let $Y$ be a regular $\mu_n$-projective arithmetic scheme over $(D,\Sigma,F_\infty)$ and let $\overline{N}$ be a $\mu_n$-equivariant hermitian vector bundle on $Y$. Suppose that the $\mu_n$-action on $Y$ is trivial and consider the zero section embedding $$i: Y\hookrightarrow P:=\mathbb{P}(N\oplus \mathcal{O}_Y)$$ with hermitian normal bundle $\overline{N}$ and the natural projection $\pi: P\to Y$. Then for any element $x\in \widehat{K}_m(Y,\mu_n)_\Q$ with integer $m\geq1$, the following identity
$$x-R_g(N)\cdot x=\pi_*i_*(x)$$ holds in $\widehat{K}_m(Y,\mu_n)_\Q$.
\end{prop}
\begin{proof}
By the definition of the action of $R_g(N)$ on $\widehat{K}_m(Y,\mu_n)_\Q$, the map $x\mapsto x-R_g(N)\cdot x$ is defined via the chain homotopy
$$\Pi^{(0)}_k(\overline{E})=\frac{(-1)^{k+1}}{(k+1)!(2\pi i)^k}\int_{(\mathbb{P}^1)^k}C_{k+1}\Big(R_g(\overline{N})\bullet {\rm ch}_g^0\big({\rm tr}_k\circ \lambda(\overline{E})\big),\log\mid z_1\mid^2,\cdots,\log\mid z_{k}\mid^2\Big)$$ of the square
\begin{align*}
\xymatrix{ \widetilde{\Z}C_*(Y,\mu_n) \ar[r]^-{{\rm ch}_g} \ar[d]^-{{\rm Id}} & \bigoplus_{p\geq0}{D}^{2p-*}(Y_{\mu_n},p)_{R_n} \ar[d]^-{{\rm Id}} \\
\widetilde{\Z}C_*(Y,\mu_n) \ar[r]^-{{\rm ch}_g} & \bigoplus_{p\geq0}{D}^{2p-*}(Y_{\mu_n},p)_{R_n}.}
\end{align*}

According to Proposition~\ref{403}, to define the morphism $i_*: \widehat{K}_m(Y,\mu_n)_\Q\to \widehat{K}_m(P,\mu_n)_\Q$, we may choose a resolution $F.$ of $i_*\mathcal{O}_Y$ on $P$ such that every $F_j$ is $\pi$-acyclic. We shall endow $F.$ with the metrics satisfying the Bismut's assumption (A). Then we have an exact sequence of hermitian bundles on $Y$
$$\overline{\Xi}:\quad 0\to \pi_*(\overline{F}_m)\to \pi_*(\overline{F}_{m-1})\to\ldots\to \pi_*(\overline{F}_0)\to \overline{\mathcal{O}}_Y\to 0.$$
Like before, splitting $\overline{\Xi}$ into a family of short exact sequence of hermitian bundles from $j=1$ to $m$
$$\xymatrix{\chi_j: \quad 0 \ar[r] & \Ker d_j \ar[r] & \overline{\Xi}_j \ar[r]^-{d_j} & \Ker d_{j-1} \ar[r] & 0,}$$
we may construct a chain homotopy
$$H_{\pi\circ i}(\overline{E}):=\sum_{j=1}^{m}(-1)^j\chi_j\otimes\overline{E}\in \widetilde{\Z}C_{k+1}(Y,\mu_n)$$
between the maps ${\rm Id}$ and $\pi_*\circ i_*: \widetilde{\Z}C_{*}(Y,\mu_n)\to \widetilde{\Z}C_{*}(Y,\mu_n)$. Consequently, the formula
$$\mathbf{H}^{(1)}_k(\overline{E})=\frac{(-1)^{k+1}}{2(k+1)!(2\pi i)^{k+1}}\int_{(\mathbb{P}^1)^{k+1}}{\rm ch}_g^0\Big({\rm tr}_{k+1}\circ\lambda\big(H_{\pi\circ i}(\overline{E})\big)\Big)\wedge C_{k+1}(\log\mid z_1\mid^2,\cdots,\log\mid z_{k+1}\mid^2)$$ defines a chain homotopy between ${\rm ch}_g\circ {\rm Id}$ and ${\rm ch}_g\circ \pi_*\circ i_*$. Then $\mathbf{H}^{(1)}_k+\Pi_k^\pi\circ i_*+{\pi_{\mu_n}}_*\circ \big({\rm Td}_g(\overline{T\pi})\bullet \mathbf{H}_k\big)$ also defines a chain homotopy between ${\rm ch}_g\circ {\rm Id}$ and ${\rm Id}\circ {\rm ch}_g$.
We compare it with $\Pi^{(0)}_k$.

Firstly, denote by ${\rm Pr}_P$ (resp. ${\rm Pr}_Y$) the projection from $P\times (\mathbb{P}^1)^k$ (resp. $Y\times (\mathbb{P}^1)^k$) to $P$ (resp. $Y$). Then, according to the functoriality of projective space bundle construction we have used before, ${\rm Pr}_P^*\overline{F}.$ provides a resolution of $i_*\overline{\mathcal{O}}_{Y\times (\mathbb{P}^1)^k}$ on $P\times (\mathbb{P}^1)^k$. Hence we have an exact sequence
$$\overline{\Xi}':\quad 0\to \pi_*({\rm Pr}_P^*\overline{F}_m)\to \pi_*({\rm Pr}_P^*\overline{F}_{m-1})\to\ldots\to \pi_*({\rm Pr}_P^*\overline{F}_0)\to \overline{\mathcal{O}}_{Y\times (\mathbb{P}^1)^k}\to 0$$
which can be split into a family of short exact sequence of hermitian bundles from $j=1$ to $m$
$$\xymatrix{\chi'_j: \quad 0 \ar[r] & \Ker d_j \ar[r] & \overline{\Xi}'_j \ar[r]^-{d_j} & \Ker d_{j-1} \ar[r] & 0.}$$
Furthermore, the short exact sequence of hermitian $1$-cube
$$H^{(j)}(\overline{E}): \xymatrix{0 \ar[r] & \chi'_j\otimes {\rm tr}_k\circ \lambda(\overline{E}) \ar[r]^-{\rm Id} & {\rm Pr}_Y^*\chi_j\otimes {\rm tr}_k\circ \lambda(\overline{E}) \ar[r] & 0 \ar[r] & 0}$$
forms a hermitian $2$-cube on $Y\times (\mathbb{P}^1)^k$. We set
\begin{align*}
\widetilde{\mathbf{H}}_k(\overline{E}):=\frac{(-1)^{k+2}}{2(k+2)!(2\pi i)^{k+2}}&\int_{(\mathbb{P}^1)^{k+2}}{\rm ch}_g^0\Big(\sum_{j=1}^m(-1)^j{\rm tr}_2\circ \lambda\big(H^{(j)}(\overline{E})\big)\Big)\wedge C_{k+2}(\log\mid z_1\mid^2,\cdots,\log\mid z_{k+2}\mid^2),
\end{align*}
it satisfies the differential equation
\begin{multline*}
d_\mathcal{D}\circ \widetilde{\mathbf{H}}_k(\overline{E})=\widetilde{\mathbf{H}}_{k-1}\circ d\overline{E}+\sum_{j=0}^m(-1)^j\Pi'^\pi_k(\overline{F}_j\otimes\pi^*\overline{E})\\
+\frac{(-1)^{k+1}}{2(k+1)!(2\pi i)^{k+1}}\int_{(\mathbb{P}^1)^{k+1}}{\rm ch}_g^0\Big({\rm tr}_{k+1}\circ\lambda\big(H_{\pi\circ i}(\overline{E})\big)\Big)\wedge C_{k+1}(\log\mid z_1\mid^2,\cdots,\log\mid z_{k+1}\mid^2)\\
-\frac{(-1)^{k+1}}{2(k+1)!(2\pi i)^{k+1}}\int_{(\mathbb{P}^1)^{k+1}}{\rm ch}_g^0\big(\sum_{j=1}^m(-1)^j{\rm tr}_1\circ\lambda(\chi'_j)\boxtimes {\rm tr}_k\circ \lambda(\overline{E})\big)\wedge C_{k+1}(\log\mid z_1\mid^2,\cdots,\log\mid z_{k+1}\mid^2)\hfill\\
=\widetilde{\mathbf{H}}_{k-1}\circ d\overline{E}+\mathbf{H}^{(1)}_k(\overline{E})+\Pi'^\pi_k\circ i_*(\overline{E})\hfill\\
-\frac{(-1)^{k+1}}{(k+1)!(2\pi i)^{k}}\int_{(\mathbb{P}^1)^{k}}C_{k+1}\Big({\rm ch}_g\big(\overline{\Xi}'\otimes {\rm tr}_k\circ \lambda(\overline{E})\big),\log\mid z_1\mid^2,\cdots,\log\mid z_{k}\mid^2\Big).
\end{multline*}
On the other hand, we apply the immersion formula to the resolution ${\rm Pr}_P^*\overline{F}.\otimes {\rm tr}_k\circ \lambda(\overline{E})$. We then have
\begin{multline*}
\Pi''^\pi_k\circ i_*(\overline{E})=-\frac{(-1)^{k+1}}{(k+1)!(2\pi i)^{k}}\int_{(\mathbb{P}^1)^{k}}C_{k+1}\Big({\rm ch}_g\big(\overline{\Xi}'\otimes {\rm tr}_k\circ \lambda(\overline{E})\big),\log\mid z_1\mid^2,\cdots,\log\mid z_{k}\mid^2\Big)\\
-\frac{(-1)^{k+1}}{(k+1)!(2\pi i)^{k}}\int_{(\mathbb{P}^1)^{k}}C_{k+1}\Big(\frac{1}{(2\pi i)^{r_\pi}}\int_{P_{\mu_n}/Y}{\rm
Td}_g(\overline{T\pi})T_g\big({\rm Pr}_P^*\overline{F}.\otimes {\rm tr}_k\circ \lambda(\overline{E})\big), \log\mid z_1\mid^2,\cdots,\log\mid z_{k}\mid^2\Big)\hfill\\
+\frac{(-1)^{k+1}}{(k+1)!(2\pi i)^k}\int_{(\mathbb{P}^1)^k}C_{k+1}\Big(R_g(\overline{N})\bullet {\rm ch}_g^0\big({\rm tr}_k\circ \lambda(\overline{E})\big),\log\mid z_1\mid^2,\cdots,\log\mid z_{k}\mid^2\Big)\\
+\frac{(-1)^{k+1}}{(k+1)!(2\pi i)^k}\int_{(\mathbb{P}^1)^k}C_{k+1}\Big(d_\mathcal{D}\Delta\big({\rm tr}_k\circ \lambda(\overline{E})\big),\log\mid z_1\mid^2,\cdots,\log\mid z_{k}\mid^2\Big)\\
=-\frac{(-1)^{k+1}}{(k+1)!(2\pi i)^{k}}\int_{(\mathbb{P}^1)^{k}}C_{k+1}\Big({\rm ch}_g\big(\overline{\Xi}'\otimes {\rm tr}_k\circ \lambda(\overline{E})\big),\log\mid z_1\mid^2,\cdots,\log\mid z_{k}\mid^2\Big)-{\pi_{\mu_n}}_*\circ \big({\rm Td}_g(\overline{T\pi})\bullet \mathbf{H}_k(\overline{E})\big)\hfill\\
+\Pi^{(0)}_k(\overline{E})+\frac{(-1)^{k+1}}{(k+1)!(2\pi i)^k}\int_{(\mathbb{P}^1)^k}C_{k+1}\Big(d_\mathcal{D}\Delta\big({\rm tr}_k\circ \lambda(\overline{E})\big),\log\mid z_1\mid^2,\cdots,\log\mid z_{k}\mid^2\Big).
\end{multline*}

We then formally define a product $C_{k+1}\Big(\Delta\big({\rm tr}_k\circ\lambda(\overline{E})\big),\log\mid z_1\mid^2,\cdots,\log\mid z_{k}\mid^2\Big)$ in the same way as (\ref{NewC}), and we set
$$\Delta_k(\overline{E})=\frac{(-1)^{k+1}}{(k+1)!(2\pi i)^k}\int_{(\mathbb{P}^1)^k}C_{k+1}\Big(\Delta\big({\rm tr}_k\circ\lambda(\overline{E})\big),\log\mid z_1\mid^2,\cdots,\log\mid z_{k}\mid^2\Big).$$
Again, it is readily checked by Lemma~\ref{deltai} that
\begin{align*}
&\Delta_{k-1}(d\overline{E})-d_\mathcal{D}\Delta_k(\overline{E})\\
=&\frac{(-1)^{k}}{(k+1)!(2\pi i)^k}\int_{(\mathbb{P}^1)^k}C_{k+1}\Big(d_\mathcal{D}\Delta\big({\rm tr}_k\circ\lambda(\overline{E})\big),\log\mid z_1\mid^2,\cdots,\log\mid z_{k}\mid^2\Big).
\end{align*}

Getting together all the above discussions, we see that $\widetilde{\mathbf{H}}_k+\Delta_k$ provides a homotopy between $\Pi^{(0)}_k$ and $\mathbf{H}^{(1)}_k+\Pi_k^\pi\circ i_*+{\pi_{\mu_n}}_*\circ \big({\rm Td}_g(\overline{T\pi})\bullet \mathbf{H}_k\big)$ which implies that $x-R_g(N)\cdot x=\pi_*i_*(x)$ for any element $x\in \widehat{K}_m(Y,\mu_n)_\Q$ with integer $m\geq1$.
\end{proof}

\begin{cor}\label{406}
Let $S$ be another regular $\mu_n$-projective arithmetic scheme with the trivial $\mu_n$-action. Let $f: Y\to S$ and $l=f\circ \pi: P\to S$ be two equivariant morphisms which are smooth over the generic fibres. Then the identity
$$f_*(x)-f_*(R_g(N)\cdot x)=l_*\circ i_*(x)$$
holds in $\widehat{K}_m(S,\mu_n)_\Q$ for any element $x\in \widehat{K}_m(Y,\mu_n)$.
\end{cor}
\begin{proof}
This is an immediate consequence of Proposition~\ref{405} and Corollary~\ref{305}.
\end{proof}

Now, we consider general situation. Let $X,S$ be two regular $\mu_n$-projective arithmetic schemes over $(D,\Sigma,F_\infty)$, and let $Y$ be a regular $\mu_n$-equivariant arithmetic closed subscheme of $X$ with immersion $i: Y\to X$. Let $l: X\to S$ and $f=l\circ i: Y\to S$ be two equivariant morphisms which are smooth over the generic fibres. We shall suppose that the $\mu_n$-actions on $Y$ and on $S$ are trivial (e.g. $Y=X_{\mu_n},S={\rm Spec}D$). Then the main result in this subsection is the following.

\begin{thm}\label{407}
For any element $x\in \widehat{K}_m(Y,\mu_n)$ with integer $m\geq 1$, the identity
$$f_*(x)-f_*(R_g(N_{X/Y})\cdot x)=l_*\circ i_*(x)$$
holds in $\widehat{K}_m(S,\mu_n)_\Q$.
\end{thm}

To prove Theorem~\ref{407}, we use the deformation to the normal cone construction. Denote by $W$ the blowing up of $X\times \mathbb{P}^1$ along $Y\times \{0\}$, and denote by $q_W: W\to \mathbb{P}^1$ the composition of the blow-down map $W\to X\times \mathbb{P}^1$ with the projection $X\times \mathbb{P}^1\to \mathbb{P}^1$. For any point $t\in \mathbb{A}^1\subset \mathbb{P}^1$, $t$ is called a $\mathbb{Z}$-point if it corresponds to a prime ideal $(x-a)$ in $D[x]$ with $a\in \mathbb{Z}$. Then for any $\mathbb{Z}$-point $t\in \mathbb{P}^1$ we have
$$q_W^{-1}(t)\cong\left\{
\begin{array}{ll}
    X\times \{t\}, & \hbox{if $t\neq0$,} \\
    P\cup \widetilde{X}, & \hbox{if $t=0$,} \\
\end{array}%
\right.$$
where $\widetilde{X}$ is isomorphic to the blowing up of $X$ along $Y$ and $P$ is the projective space
bundle $\mathbb{P}(N_{X/Y}\oplus\mathcal{O}_Y)$. Let $j: Y\times \mathbb{P}^1\rightarrow W$ be the closed immersion
induced by $i\times {\rm Id}$, then the component $\widetilde{X}$ doesn't meet $j(Y\times \mathbb{P}^1)$ and the intersection of
$j(Y\times \mathbb{P}^1)$ with $P$ is exactly the image of $Y$ under the zero section embedding. Moreover, denote by $s_t$ the obvious section $Y\cong Y\times\{t\}\hookrightarrow Y\times \mathbb{P}^1$ for every $\mathbb{Z}$-point $t$ and denote by $u_t$ the natural inclusion $q_W^{-1}(t)\hookrightarrow W$. We have two ${\rm Tor}$-independent squares
$$\xymatrix{ Y\times \mathbb{P}^1 \ar[r]^-{j} & W \\ Y \ar[u]^-{s_t} \ar[r]^-i & X \ar[u]^-{u_t}}$$
with $t\neq0$ and
$$\xymatrix{ Y\times \mathbb{P}^1 \ar[r]^-{j} & W \\ Y \ar[u]^-{s_0} \ar[r]^-{i_0} & \mathbb{P}(N_{X/Y}\oplus\mathcal{O}_Y) \ar[u]^-{u_0}.}$$

Notice that the complement $X\setminus Y$ is contained in $W\setminus {Y\times \mathbb{P}^1}$, we have pull-back morphism $u_t^*: \widehat{K}_{Y\times \mathbb{P}^1,m}(W,\mu_n)\to \widehat{K}_{Y,m}(X,\mu_n)$.

\begin{lem}\label{408}
For any $\mathbb{Z}$-point $t\neq0$, the diagram
$$\xymatrix{ \widehat{K}_m(Y\times \mathbb{P}^1,\mu_n) \ar[r]^-{\cong}_-{j_*} \ar[d]^-{s_t^*} & \widehat{K}_{Y\times \mathbb{P}^1,m}(W,\mu_n) \ar[d]^-{u_t^*} \\ \widehat{K}_m(Y,\mu_n) \ar[r]^-{\cong}_-{i_*} & \widehat{K}_{Y,m}(X,\mu_n) }$$
is commutative.
\end{lem}
\begin{proof}
The commutativity of the algebraic prototype of this diagram follows from the Tor-independence of the deformation diagrams, but for arithmetic K-theory it is more complicated because the morphisms $j_*$ and $i_*$ are defined via another deformation to the normal cone construction according to the $\mathbb{A}^1$-homotopy invariance of the K-theory and the Deligne-Beilinson cohomology. 

Write $c_t^*: \widehat{K}_m(Y\times \mathbb{P}^1,\mu_n)\to \widehat{K}_m(Y,\mu_n)$ for the composition $i_*^{-1}\circ u_t^*\circ j_*$. We need to show that $c_t^*=s_t^*$. The morphism $s_t^*$ is induced by the commutativity between $s_t^*$ and $\widetilde{{\rm ch}}_g$, while the morphism $c_t^*$ is induced by the homotopy defining $j_*$ and the homotopy defining $i_*$. Again, using the $\mathbb{A}^1$-homotopy invariance of the K-theory and the Deligne-Beilinson cohomology, we may consider the pull-backs of $s_t^*$ and $c_t^*$ to $\widehat{K}_m(Y\times \mathbb{P}^1\times \mathbb{A}^1,\mu_n)\to \widehat{K}_m(Y\times \mathbb{A}^1,\mu_n)$ and restrict them to $\{0\}\hookrightarrow \mathbb{A}^1$, then the statement in this lemma will follows from the commutativity of the diagram
\begin{align}\label{dn}
\xymatrix{ \widehat{K}_m(Y\times \mathbb{P}^1,\mu_n) \ar[r]^-{\cong}_-{{j_0}_*} \ar[d]^-{s_t^*} & \widehat{K}_{Y\times \mathbb{P}^1,m}(P',\mu_n) \ar[d]^-{u_t^*} \\ \widehat{K}_m(Y,\mu_n) \ar[r]^-{\cong}_-{{i_0}_*} & \widehat{K}_{Y,m}(P,\mu_n) }
\end{align}
where $P'=\mathbb{P}\Big(\big(N_{X/Y}\boxtimes \mathcal{O}(-1)\big)\oplus\mathcal{O}_{Y\times \mathbb{P}^1}\Big)$ is the projective completion of $N_{W/{Y\times \mathbb{P}^1}}$ over $Y\times \mathbb{P}^1$. It is equivalent to show that the following diagram
\begin{align}\label{dn'}
\xymatrix{ \widehat{K}_m(Y\times \mathbb{P}^1,\mu_n) \ar[r]^-{{j_0}_*} \ar[d]^-{s_t^*} & \widehat{K}_{m}(P',\mu_n) \ar[d]^-{u_t^*} \\ \widehat{K}_m(Y,\mu_n) \ar[r]^-{{i_0}_*} & \widehat{K}_{m}(P,\mu_n) }
\end{align}
is commutative because the morphism ${i_0}_*: \widehat{K}_m(Y,\mu_n) \to \widehat{K}_{m}(P,\mu_n)$ is injective.
We endow $N_{X/Y}\boxtimes \mathcal{O}(-1)$ with the product metric coming from the metric on $N_{X/Y}$ and the Fubini-Study metric on $\mathcal{O}(-1)$, then the pull-back of $\overline{N}_{W/{Y\times \mathbb{P}^1}}$ along $s_t$ is isometric to $\overline{N}_{X/Y}$ so that the pull-back along $s_t$ of the Koszul resolution and of the corresponding Bott-Chern singular current with respect to $j_0$ is exactly the Koszul resolution and the corresponding Bott-Chern singular current with respect to $i_0$. According to the construction of the homotopies defining ${j_0}_*$ and ${i_0}_*$, we get the commutativity of the diagram (\ref{dn'}) and hence of (\ref{dn}). So we are done.
\end{proof}

\begin{cor}\label{409}
For any $\mathbb{Z}$-point $t\neq0$, the diagram
$$\xymatrix{ \widehat{K}_m(Y\times \mathbb{P}^1,\mu_n) \ar[r]^-{j_*} \ar[d]^-{s_t^*} & \widehat{K}_{m}(W,\mu_n) \ar[d]^-{u_t^*} \\ \widehat{K}_m(Y,\mu_n) \ar[r]^-{i_*} & \widehat{K}_{m}(X,\mu_n) }$$
is commutative.
\end{cor}

\begin{rem}\label{410}
Using the same argument as in Lemma~\ref{408}, we know that the diagram
$$\xymatrix{ \widehat{K}_m(Y\times \mathbb{P}^1,\mu_n) \ar[r]^-{j_*} \ar[d]^-{s_0^*} & \widehat{K}_{m}(W,\mu_n) \ar[d]^-{u_0^*} \\ \widehat{K}_m(Y,\mu_n) \ar[r]^-{{i_0}_*} & \widehat{K}_{m}(P,\mu_n) }$$
is also commutative.
\end{rem}

Next, we consider the commutative diagram
$$\xymatrix{ W \ar[rd]^-{l} & \\ X \ar[u]^-{u_t} \ar[r]^-{f}  & S} $$ with $\mathbb{Z}$-point $t\neq 0$ and we compare the map $f_*\circ u_t^*$ with the map $l_*$ from $\widehat{K}_m(W,\mu_n)_\Q$ to $\widehat{K}_m(S,\mu_n)_\Q$.

Firstly, for any $\mu_n$-invariant K\"{a}hler metric $\omega^X$ on $X$ which induces an invariant K\"{a}hler metric $\omega^Y$ on $Y$, there exists a $\mu_n$-invariant K\"{a}hler metric $\omega^W$ on $W$ such that the restrictions of $\omega^W$ over $X\cong X\times\{t\}$ with $t\neq 0$ and to $Y\cong Y\times\{0\}$ are exactly $\omega^X$ and $\omega^Y$. This fact follows from \cite[Lemma 3.5]{T2}. Actually, such a metric is constructed via the Grassmannian graph construction. In this construction, we have an embedding $W\to X\times
\mathbb{P}^r\times\mathbb{P}^1$ and the metric $\omega^W$ is the $\mu_n$-average of the restriction of a product metric on $X\times
\mathbb{P}^r\times\mathbb{P}^1$. We fix such an invariant K\"{a}hler metric $\omega^W$ on $W$ and endow all submanifolds of $W$ with the induced metrics. Moreover, all normal bundles appearing in the construction of the deformation to the normal cone will be endowed with the quotient metrics.

Secondly, to the three divisors $u_t(X)$, $u_0(P)$ and $u_0(\widetilde{X})$ in $W$, we have the following result.

\begin{lem}\label{deformation}
Over $W$, there are $\mu_n$-invariant hermitian metrics on $\mathcal{O}(X)$, $\mathcal{O}(P)$ and $\mathcal{O}(\widetilde{X})$ such that
the isometry $\overline{\mathcal{O}}(X)\cong\overline{\mathcal{O}}(P)\otimes\overline{\mathcal{O}}(\widetilde{X})$
holds and such that the restriction of $\overline{\mathcal{O}}(X)$ over $X$ yields the metric of
$N_{W/{X}}$, the restriction of $\overline{\mathcal{O}}(\widetilde{X})$ over $\widetilde{X}$
yields the metric of $N_{W/{\widetilde{X}}}$ and the restriction of $\overline{\mathcal{O}}(P)$ over $P$ yields
the metric of $N_{W/{P}}$.
\end{lem}
\begin{proof}
choose metric on $\mathcal{O}(P)$ in a small neighborhood of $P$ such that the restriction of $\overline{\mathcal{O}}(P)$ over $P$ yields the metric of the normal bundle. Do the same for $\mathcal{O}(\widetilde{X})$. Since $X$ is closed and disjoint from $\widetilde{X}$ and $P$, we can extend these metrics via a partition of unity to metrics defined on $W$ so that the restriction of the metric that $\mathcal{O}(X)$ inherits from the isomorphism $\mathcal{O}(X)\cong\mathcal{O}(P)\otimes\mathcal{O}(\widetilde{X})$ yields the metric of the normal bundle $N_{W/X}$. We then take the $\mu_n$-averages of these metrics to make them $\mu_n$-invariant. Since the metrics on $N_{W/{X}}$, $N_{W/{P}}$ and $N_{W/{\widetilde{X}}}$ are already $\mu_n$-invariant, the $\mu_n$-invariant metrics on $\mathcal{O}(X)$, $\mathcal{O}(P)$ and $\mathcal{O}(\widetilde{X})$ obtained as above have the properties that we require.
\end{proof}

Now, consider the Koszul resolution
$$0\to \overline{\mathcal{O}}(-X)\to \overline{\mathcal{O}}_W\to {u_t}_*\overline{\mathcal{O}}_X\to 0.$$
The associated equivariant singular Bott-Chern current $T_g(W/X)$ satisfies the identity
$$d_\mathcal{D}T_g(W/X)={\rm ch}_g^0(\overline{\mathcal{O}}_W)-{\rm ch}_g^0\big(\overline{\mathcal{O}}(-X)\big)-{u_t}_*[{\rm ch}_g^0(\overline{\mathcal{O}}_X){\rm Td}_g^{-1}(\overline{N}_{W/X})].$$
We claim the following result.

\begin{lem}\label{411}
For any element $x\in \widehat{K}_m(W,\mu_n)_\Q$ with integer $m\geq 1$, the identity
$$f_*\circ u_t^*(x)-f_*(R_g(N_{W/X})\cdot u_t^*x)=l_*(x)-l_*(\overline{O}(-X)\otimes x)$$ hold in
$\widehat{K}_m(S,\mu_n)_\Q$.
\end{lem}
\begin{proof}
Let $\overline{E}$ be a $l$-acyclic hermitian $k$-cube in $\widehat{\mathcal{P}}(W,\mu_n)$. Since $W$ admits a very ample invertible $\mu_n$-sheaf which is relative to the morphism $l: W\to S$ (cf. \cite[Lemma 3.9]{T2}), we may assume that $\overline{O}(-X)\otimes \overline{E}$ is also $l$-acyclic and $u_t^*\overline{E}$ is $f$-acyclic. Then we have a short exact sequence of hermitian $k$-cubes in $\widehat{\mathcal{P}}(S,\mu_n)$
$$\chi(\overline{E}): \quad 0\to l_*(\overline{\mathcal{O}}(-X)\otimes \overline{E})\to l_*(\overline{E})\to f_*({u_t}^*\overline{E})\to 0,$$ which will be regarded as a hermitian $(k+1)$-cube and as a chain homotopy between the maps $l_*-l_*(\overline{\mathcal{O}}(-X)\otimes)$ and $f_*\circ u_t^*$. Consequently, the formula
$$\mathbf{H}^{(1)}_k(\overline{E})=\frac{(-1)^{k+1}}{2(k+1)!(2\pi i)^{k+1}}\int_{(\mathbb{P}^1)^{k+1}}{\rm ch}_g^0\Big({\rm tr}_{k+1}\circ\lambda\big(\chi(\overline{E})\big)\Big)\wedge C_{k+1}(\log\mid z_1\mid^2,\cdots,\log\mid z_{k+1}\mid^2)$$ defines a chain homotopy between ${\rm ch}_g\circ l_*-{\rm ch}_g\circ l_*(\overline{\mathcal{O}}(-X)\otimes)$ and ${\rm ch}_g\circ f_*\circ u_t^*$.

On the other hand, for any element $\alpha\in \bigoplus_{p\geq 0}D^{2p-*}(W_{\mu_n},p)_{R_n}$, the formula $$\frac{1}{(2\pi i)^{r_l}}\int_{W_{\mu_n}/S}T_g(W/X)\bullet {\rm Td}_g(\overline{Tg})\bullet \alpha+\frac{1}{(2\pi i)^{r_f}}\int_{X_{\mu_n}/S}{\rm Td}_g(\overline{\mathcal{N}})\bullet{\rm Td}_g^{-1}(\overline{N}_{W/X})\bullet\alpha$$ gives a chain homotopy between the maps ${l_{\mu_n}}_!\circ ({\rm Td}_g(\overline{Tg})\bullet)-{l_{\mu_n}}_!\circ \Big({\rm Td}_g(\overline{Tg}){\rm ch}_g^0\big(\overline{\mathcal{O}}(-X)\big)\bullet\Big)$ and ${f_{\mu_n}}_!\circ ({\rm Td}_g(\overline{Tf})\bullet u_t^*)$. Hence, it defines a chain homotopy between ${l_{\mu_n}}_!\circ ({\rm Td}_g(\overline{Tg})\bullet{\rm ch}_g)-{l_{\mu_n}}_!\circ \Big({\rm Td}_g(\overline{Tg}){\rm ch}_g^0\big(\overline{\mathcal{O}}(-X)\big)\bullet{\rm ch}_g\Big)$ and ${f_{\mu_n}}_!\circ ({\rm Td}_g(\overline{Tf})\bullet u_t^*\circ {\rm ch}_g)$. Like before, using the projection formula and the fact that the deformation to the normal cone construction is base-change invariant along smooth morphisms, we write the induced homotopy as
\begin{multline*}
\mathbf{H}^{(2)}_k(\overline{E})=\frac{(-1)^{k}}{2k!(2\pi i)^{k}}\int_{(\mathbb{P}^1)^{k}}\bigg(\Big(\frac{1}{(2\pi i)^{r_l}}\int_{W_{\mu_n}\times (\mathbb{P}^1)^{k}/{S\times (\mathbb{P}^1)^{k}}}T_g(W/X)\bullet {\rm Td}_g(\overline{Tg}){\rm ch}_g^0\big({\rm tr}_{k}\circ\lambda(\overline{E})\big)\Big)\\
{\wedge C_{k}(\log\mid z_1\mid^2,\cdots,\log\mid z_{k}\mid^2)\bigg)}\\
+\frac{(-1)^{k}}{2k!(2\pi i)^{k}}\int_{(\mathbb{P}^1)^{k}}\bigg(\Big(\frac{1}{(2\pi i)^{r_f}}\int_{X_{\mu_n}\times (\mathbb{P}^1)^{k}/{S\times (\mathbb{P}^1)^{k}}}{\rm Td}_g(\overline{\mathcal{N}})\bullet{\rm Td}_g^{-1}(\overline{N}_{W/X}){\rm ch}_g^0\big({\rm tr}_{k}\circ\lambda(u_t^*\overline{E})\big)\Big)\hfill\\
{\wedge C_{k}(\log\mid z_1\mid^2,\cdots,\log\mid z_{k}\mid^2)\bigg).}
\end{multline*}

Now, we denote by $H_{\chi}(\overline{E})$ the following $2$-cube of hermitian bundles on $S\times (\mathbb{P}^1)^k$
$$\xymatrix{l_*\big({\rm tr}_k\circ\lambda(\overline{\mathcal{O}}(-X)\otimes\overline{E})\big) \ar[r]^-{\rm Id} \ar[d] & {\rm tr}_k\circ\lambda\big(l_*(\overline{\mathcal{O}}(-X)\otimes\overline{E})\big) \ar[r] \ar[d] & 0 \ar[d]\\
l_*\big({\rm tr}_k\circ\lambda(\overline{E})\big) \ar[r]^-{\rm Id} \ar[d] & {\rm tr}_k\circ\lambda\big(l_*(\overline{E})\big) \ar[r] \ar[d] & 0 \ar[d]\\
f_*\big({\rm tr}_k\circ\lambda(u_t^*\overline{E})\big) \ar[r]^-{\rm Id} & {\rm tr}_k\circ\lambda\big(f_*(u_t^*\overline{E})\big) \ar[r] & 0}$$ and we set
\begin{align*}
\widetilde{\mathbf{H}}_k(\overline{E}):=\frac{(-1)^{k+2}}{2(k+2)!(2\pi i)^{k+2}}&\int_{(\mathbb{P}^1)^{k+2}}{\rm ch}_g^0\Big({\rm tr}_2\circ \lambda\big(H_\chi(\overline{E})\big)\Big)\wedge C_{k+2}(\log\mid z_1\mid^2,\cdots,\log\mid z_{k+2}\mid^2),
\end{align*}
it satisfies the differential equation
\begin{multline*}
d_\mathcal{D}\circ \widetilde{\mathbf{H}}_k(\overline{E})\\
=\widetilde{\mathbf{H}}_{k-1}\circ d\overline{E}+\frac{(-1)^{k+1}}{2(k+1)!(2\pi i)^{k+1}}\int_{(\mathbb{P}^1)^{k+1}}{\rm ch}_g^0\Big({\rm tr}_{k+1}\circ\lambda\big(\chi(\overline{E})\big)\Big)\wedge C_{k+1}(\log\mid z_1\mid^2,\cdots,\log\mid z_{k+1}\mid^2)\hfill\\
-\frac{(-1)^{k+1}}{2(k+1)!(2\pi i)^{k+1}}\int_{(\mathbb{P}^1)^{k+1}}{\rm ch}_g^0\bigg({\rm tr}_1\circ\lambda\Big(\chi\big({\rm tr}_k\circ \lambda(\overline{E})\big)\Big)\bigg)\wedge C_{k+1}(\log\mid z_1\mid^2,\cdots,\log\mid z_{k+1}\mid^2)\hfill\\
-\Pi'^l_k(\overline{E})+\Pi'^l_k(\overline{\mathcal{O}}(-X)\otimes\overline{E})+\Pi'^f_k(u_t^*\overline{E})\\
=\widetilde{\mathbf{H}}_{k-1}\circ d\overline{E}+\mathbf{H}^{(1)}_k(\overline{E})-\Pi'^l_k(\overline{E})+\Pi'^l_k(\overline{\mathcal{O}}(-X)\otimes\overline{E})+\Pi'^f_k(u_t^*\overline{E})\hfill\\
-\frac{(-1)^{k+1}}{2(k+1)!(2\pi i)^{k+1}}\int_{(\mathbb{P}^1)^{k+1}}{\rm ch}_g^0\bigg({\rm tr}_1\circ\lambda\Big(\chi\big({\rm tr}_k\circ \lambda(\overline{E})\big)\Big)\bigg)\wedge C_{k+1}(\log\mid z_1\mid^2,\cdots,\log\mid z_{k+1}\mid^2).\hfill
\end{multline*}

Similar to the tricks that we used frequently before, we set
\begin{multline*}
\mathbf{H}^{(2')}_k(\overline{E})\\
=\frac{(-1)^{k+1}}{(k+1)!(2\pi i)^{k}}\int_{(\mathbb{P}^1)^{k}}C_{k+1}\Big(\frac{1}{(2\pi i)^{r_l}}\int_{W_{\mu_n}\times (\mathbb{P}^1)^{k}/{S\times (\mathbb{P}^1)^{k}}}T_g(W/X)\bullet {\rm Td}_g(\overline{Tg}){\rm ch}_g^0\big({\rm tr}_{k}\circ\lambda(\overline{E})\big)\hfill\\
{,\log\mid z_1\mid^2,\cdots,\log\mid z_{k}\mid^2\Big)}\\
+\frac{(-1)^{k+1}}{(k+1)!(2\pi i)^{k}}\int_{(\mathbb{P}^1)^{k}}C_{k+1}\Big(\frac{1}{(2\pi i)^{r_f}}\int_{X_{\mu_n}\times (\mathbb{P}^1)^{k}/{S\times (\mathbb{P}^1)^{k}}}{\rm Td}_g(\overline{\mathcal{N}})\bullet{\rm Td}_g^{-1}(\overline{N}_{W/X}){\rm ch}_g^0\big({\rm tr}_{k}\circ\lambda(u_t^*\overline{E})\big)\hfill\\
{,\log\mid z_1\mid^2,\cdots,\log\mid z_{k}\mid^2\Big).}
\end{multline*}
then our lemma follows from the Bimut-Ma's immersion formula and the fact that there exists a homotopy between $\mathbf{H}^{(2')}_k(\overline{E})$ and $\mathbf{H}^{(2)}_k(\overline{E})$. So we are done.
\end{proof}

\begin{rem}\label{412}
Similar to Lemma~\ref{411}, we consider other three divisors $\xymatrix{ W & P \ar[l]^-{u_0} \ar[r]^-{p} & S}$, $\xymatrix{ W & \widetilde{X} \ar[l]^-{u_0}  \ar[r]^-{h_1} & S}$ and $\xymatrix{ W & P\cap \widetilde{X} \ar[l]^-{u_0}  \ar[r]^-{h_2} & S}$ and corresponding Koszul resolutions
$$ 0\to \overline{\mathcal{O}}(-P)\to \overline{\mathcal{O}}_{W}\to {u_0}_*\overline{\mathcal{O}}_{P}\to 0,$$
$$ 0\to \overline{\mathcal{O}}(-\widetilde{X})\to \overline{\mathcal{O}}_{W}\to {u_0}_*\overline{\mathcal{O}}_{\widetilde{X}}\to 0,$$
and
$$ 0\to \overline{\mathcal{O}}(-\widetilde{X})\otimes\overline{\mathcal{O}}(-P)\to
\overline{\mathcal{O}}(-\widetilde{X})\oplus \overline{\mathcal{O}}(-P)\to
\overline{\mathcal{O}}_{W}\to {u_0}_*\overline{\mathcal{O}}_{\widetilde{X}\cap P}\to 0.$$
Then, for any element $x\in \widehat{K}_m(W,\mu_n)_\Q$, we have
$$p_*\circ u_0^*(x)-p_*(R_g(N_{W/P})\cdot u_0^*x)=l_*(x)-l_*\big(\overline{O}(-P)\otimes x\big),$$
$${h_1}_*\circ u_0^*(x)-{h_1}_*(R_g(N_{W/\widetilde{X}})\cdot u_0^*x)=l_*(x)-l_*\big(\overline{O}(-\widetilde{X})\otimes x\big),$$ and
\begin{multline*}
{h_2}_*\circ u_0^*(x)-{h_2}_*(R_g(N_{W/{P\cap \widetilde{X}}})\cdot u_0^*x)=l_*(x)-l_*(\overline{O}(-P)\otimes x)-l_*(\overline{O}(-\widetilde{X})\otimes x)+l_*(\overline{O}(-P)\otimes\overline{O}(-\widetilde{X})\otimes x)
\end{multline*}
which hold in
$\widehat{K}_m(S,\mu_n)_\Q$.
\end{rem}

Now, we are ready to give the proof of Theorem~\ref{407}.

\begin{proof}(of Theorem~\ref{407})
Let $x$ be an element in $\widehat{K}_m(Y,\mu_n)_\Q$, we consider the following two diagrams
$$\xymatrix{ Y\times \mathbb{P}^1 \ar[r]^-{j} & W \ar[rd]^-{h} & \\ Y \ar[u]^-{s_t} \ar[r]^-i & X \ar[u]^-{u_t} \ar[r]^-{l} & S}$$
with $\mathbb{Z}$-point $t\neq0$ and
$$\xymatrix{ Y\times \mathbb{P}^1 \ar[r]^-{j} & W \ar[rd]^-{h} & \\ Y \ar[u]^-{s_0} \ar[r]^-{i_0} & \mathbb{P}(N_{X/Y}\oplus\mathcal{O}_Y) \ar[u]^-{u_0} \ar[r]^-{p} & S.}$$
By Corollary~\ref{409} and the fact that $s_t$ is a section of the obvious projection $Pr$ from $Y\times \mathbb{P}^1$ to $Y$, we have that $i_*(x)=u_t^*\circ j_* \circ Pr^*(x)$ and hence $l_*\circ i_*(x)=l_*\circ u_t^*\circ j_* \circ Pr^*(x)$. According to Lemma~\ref{411}, $$l_*\circ u_t^*(j_*Pr^*x)=h_*(j_*Pr^*x)-h_*(\overline{O}(-X)\otimes j_*Pr^*x)+l_*(R_g(N_{W/X})\cdot i_*x).$$
Similarly, we have
$$l_*\circ u_0^*(j_*Pr^*x)=h_*(j_*Pr^*x)-h_*(\overline{O}(-P)\otimes j_*Pr^*x)+p_*(R_g(N_{W/P})\cdot i_*x).$$
Notice that the image $j(Y\times \mathbb{P}^1)$ doesn't meet $\widetilde{X}$, the localization sequence of the higher equivariant arithmetic K-groups implies that $u_0^*(j_*Pr^*x)$ vanishes in $\widehat{K}_m(\widetilde{X},\mu_n)_\Q$ and in $\widehat{K}_m(P\cap \widetilde{X},\mu_n)_\Q$ so that $$h_*(\overline{O}(-X)\otimes j_*Pr^*x)=h_*(\overline{O}(-P)\otimes j_*Pr^*x).$$ This can be seen from the several identities mentioned in Remark~\ref{412}. On the other hand,
\begin{align*}
R_g(N_{W/X})\cdot i_*x=&R_g(N_{W/X}){\rm ch}_g(i_*x)=R_g(N_{W/X})i_*\big({\rm Td}_g^{-1}(N_{X/Y}){\rm ch}_g(x)\big)\\
=&i_*\big(i^*R_g(N_{W/X}){\rm Td}_g^{-1}(N_{X/Y}){\rm ch}_g(x)\big)=0.
\end{align*}
The same reasoning gives that $R_g(N_{W/P})\cdot i_*x=0$ also. So $l_*\circ i_*(x)$ is actually equal to $p_*\circ {i_0}_*(x)$. Therefore, the statement in Theorem~\ref{407} follows from Corollary~\ref{406}.
\end{proof}

\subsection{Proof of the statement}
In this subsection, we give a complete proof of Theorem~\ref{402}. Denote by $i$ the closed immersion $X_{\mu_n}\to X$, then the arithmetic concentration theorem (cf. \cite[Theorem 5.2]{T3}) tells us that
$$i_*: \widehat{K}_m(X_{\mu_n},\mu_n)_\rho\cong \widehat{K}_m(X,\mu_n)_\rho$$ with inverse map $\otimes \lambda_{-1}^{-1}(\overline{N}_{X/{X_{\mu_n}}})\circ \tau$.

Then let $x$ be any element in $\widehat{K}_m(X,\mu_n)$, we apply Theorem~\ref{407} to the morphisms $i$, $f$ and $f_{\mu_n}=f\circ i$ and we compute
\begin{align*}
f_*(x)=&f_*\big(i_*\circ \otimes \lambda_{-1}^{-1}(\overline{N}_{X/{X_{\mu_n}}})\circ \tau(x)\big)\\
=&f_*\circ i_*\big(\otimes \lambda_{-1}^{-1}(\overline{N}_{X/{X_{\mu_n}}})\circ \tau(x)\big)\\
=&{f_{\mu_n}}_*\big(\otimes \lambda_{-1}^{-1}(\overline{N}_{X/{X_{\mu_n}}})\circ \tau(x)\big)-{f_{\mu_n}}_*\big(\otimes R_g(N_{X/{X_{\mu_n}}})\circ\otimes \lambda_{-1}^{-1}(\overline{N}_{X/{X_{\mu_n}}})\circ \tau(x)\big)\\
=&{f_{\mu_n}}_*\big(\Lambda_R\circ \tau(x)\big)
\end{align*}
which holds in $\widehat{K}_m(Y,\mu_n)_\rho\otimes \Q$. This completes the proof of Theorem~\ref{402}.

\section{Appendix: Remarks on the equivariant analytic torsion forms and the immersion formula}

\subsection{Anomaly formula for the equivariant analytic torsion forms}
\label{Ma2.s1}
Let $W, V$ be two $\mu_n$-projective complex manifolds, and
 let $f: W\to V$ be an equivariant, holomorphic submersion with fiber $X$. 
 Fix a $\mu_n$-invariant K\"{a}hler metric on $W$ and choose 
 corresponding K\"{a}hler form $\omega$ as a K\"{a}hler fibration 
 structure on $f$. We fix a primitive $n$-th root of unity $g$ as 
 a generator of $\mu_n(\C)$. In the following, ${\rm ch}_g$
  and ${\rm Td}_g$ should stand for the usual Chern-Weil forms 
  with the factor $2\pi i$ in their definitions. Notice that they are 
  denoted by ${\rm ch}'_g$ and ${\rm Td}'_g$ in the text.

Let $(E,h^E)$ be a $\mu_n$-equivariant hermitian vector bundle 
on $W$ such that $E$ is $f$-acyclic. 
Let $T_g(f,\omega,h^E)\in\bigoplus_{p\geq0}A^{p,p}(V_{\mu_n})$ 
be the equivariant analytic torsion form
 \cite[(2.27)]{Ma1} which satisfies 
the differential equation
$$\frac{\bar\partial \partial}{2\pi i}T_g(f,\omega,h^E)
={\rm ch}_g(f_*E,f_*h^E)
-\int_{W_{\mu_n}/{V_{\mu_n}}}
{\rm Td}_g(Tf,h^{Tf}){\rm ch}_g(E,h^E)$$
where $h^{Tf}$ is the hermitian metric induced by $\omega$
 on the holomorphic tangent bundle $Tf$. We shall write 
 $T_g(\omega,h^E)$ for $T_g(f,\omega,h^E)$, 
 if there is no ambiguity about the underlying map.
  The following result is \cite[Theorem 2.13]{Ma1}
  which extends  \cite[Theorem 1.23]{BGS3}, 
  \cite[Theorem 3.10]{BK}, \cite[Theorem 2.5]{Bi95}.

\begin{thm}\label{A10}(Anomaly formula)
Let $\omega'$ be the form associated to another K\"{a}hler
 fibration structure on $f: W\to V$. Let $h'^{Tf}$ be the metric on
$Tf$ induced by  $\omega'$.  
Then the following identity holds in 
$\bigoplus_{p\geq 0}A^{p,p}(V_{\mu_n})/
{(\Im{\partial}+\Im{\bar\partial})}$:
\begin{align*}
T_g(\omega,h^E)-T_g(\omega',h^E)=& - 
\widetilde{{\rm ch}}_g(f_*E,h^{f_*E},h'^{f_*E})\\
&+ \int_{W_{\mu_n}/{V_{\mu_n}}}
\widetilde{{\rm Td}}_g(Tf,h^{Tf},h'^{Tf}){\rm ch}_g(E,h^E)
\end{align*}
where $(f_*E,h^{f_*E},h'^{f_*E})$ and $(Tf,h^{Tf},h'^{Tf})$ 
stand for the exact sequences of hermitian vector bundles
$$\xymatrix{ 0\ar[r] & (f_*E,h^{f_*E}) \ar[r]^-{\rm Id} &
 (f_*E,h'^{f_*E}) \ar[r] & 0 \ar[r] & 0}$$
and
$$\xymatrix{ 0\ar[r] & (Tf,h^{Tf}) \ar[r]^-{\rm Id} &
 (Tf,h'^{Tf}) \ar[r] & 0 \ar[r] & 0}.$$
\end{thm}

We shall see that there is a natural way to write down explicitly 
some differential forms $\Delta^0(f,\overline{E},\omega,\omega')$, 
$\Delta_0(f,\overline{E},\omega,\omega')$ such that they 
are functorial in certain sense and they measure 
the difference of the anomaly formula.
\begin{align*}
\Delta&=\partial\Delta^0(f,\overline{E},\omega,\omega')
+\bar\partial\Delta_0(f,\overline{E},\omega,\omega')\\
&=T_g(\omega,h^E)-T_g(\omega',h^E)
+ \widetilde{{\rm ch}}_g(f_*E,h^{f_*E},h'^{f_*E})
- \int_{W_{\mu_n}/{V_{\mu_n}}}
\widetilde{{\rm Td}}_g(Tf,h^{Tf},h'^{Tf}){\rm ch}_g(E,h^E).
\end{align*}
To do so, we need to fix the construction of
$\widetilde{{\rm ch}}_g(f_*E,h^{f_*E},h'^{f_*E})$, 
$\widetilde{{\rm Td}}_g(Tf,h^{Tf},h'^{Tf})$ at the differential form level,
i.e., without modulo $\Im{\partial}+\Im{\bar\partial}$.
Let's fix the  definition of $\widetilde{{\rm ch}}_g(f_*E,h^{f_*E},h'^{f_*E})$
as the left side of  \cite[(2.42)]{Bi95},  and the definition 
$\widetilde{{\rm Td}}_g(Tf,h^{Tf},h'^{Tf})$ as the integral for $0$
to $1$ for the parameter $c$ of the differential form as the part
$\frac{\partial}{\partial b}\cdots$ via the last term of \cite[(3.67)]{BK}, 
note that we can also 
fix the path of the metric as the segment direct connecting 
two metrics. Thus we can write them under the notation in
\cite[(2.34), (2.56)] {Ma1} 
(cf. also the convention before Theorem 2.5 of this paper),
\begin{align}\label{Ma2.3.2}\begin{split}
&\widetilde{{\rm ch}}_g(f_*E,h^{f_*E},h'^{f_*E})
=  \int_0^1 \Phi \mbox{Tr}_s[g Q^{H(X, E|_{X})} _c 
\exp(-(\nabla^{H(X, E|_{X})}_c)^2)] dc,\\
&\widetilde{{\rm Td}}_g(Tf,h^{Tf},h'^{Tf})
=   \int_0^1 \frac{\partial}{\partial b}
\Big [ \mbox{Td}\Big (\frac{-R^{TX^g}_c}{2i\pi}
-b(h^{TX^g})^{-1}\frac{\partial h^{TX^g}}{\partial c}\Big ) \\
&\hspace{3cm} \times \prod_{j=1}^q 
\frac{\mbox{Td}}{e}\Big (\frac{-R^{N_{X^g/X}^{\theta_j}}_c}{2i\pi}
-b(h^{N^{\theta_j}_{X^g/X}})^{-1}
\frac{\partial h^{N^{\theta_j}_{X^g/X}}}{\partial c}+i \theta_j\Big )
\Big ]_{b=0} d c.
\end{split}\end{align}

Let $V_1$ be an equivariant closed submanifold of $V$, and let 
$W_1=f^{-1}(V_1)\subset W$ be the closed submanifold of $W$ 
with restricted K\"{a}hler metric. Then $f_1: W_1\to V_1$ is also 
an equivariant holomorphic submersion with compact fibre. 
Denote by $j$ (resp. $i$) the natural embedding $W_1\to W$ 
(resp. $V_1\to V$) and by $\omega_1, \omega'_1$
 the induced K\"{a}hler forms $j^*\omega, j^*\omega'$.  
 Let $\overline{E}$ be an $f$-acyclic hermitian bundle on $W$.

\begin{thm}\label{A1}
There is a natural way to write down explicitly differential forms 
$\Delta^0(f,\overline{E},\omega,\omega')$,
 $\Delta_0(f,\overline{E},\omega,\omega')$ such that 
$\Delta=\partial\Delta^0(f,\overline{E},\omega,\omega')
+\bar\partial\Delta_0(f,\overline{E},\omega,\omega')$ and 
they are functorial in the following sense.
\begin{align}\label{Ma2.3.5}
	i_{\mu_n}^*\Delta^0(f,\overline{E},\omega,\omega')
=\Delta^0(f_1,j^*\overline{E},\omega_1,\omega'_1)
\end{align}
and
\begin{align}\label{Ma2.3.6}
	i_{\mu_n}^*\Delta_0(f,\overline{E},\omega,\omega')
=\Delta_0(f_1,j^*\overline{E},\omega_1,\omega'_1).
\end{align}
\end{thm}
\begin{proof}
By the equivariant extension of 
\cite[Definition 3.14, Theorems 3.16, 3.17]{BK}] (cf. \cite[(2.34)]{Ma1}), 
there exist differential forms $\theta^1, \theta^2$ and $\theta^3$ 
such that 
\begin{align}\label{Ma2.3.7}
	\Delta+ d \varrho=\bar{\partial}\theta^1+\partial \theta^2
+\bar{\partial}\partial\theta^3
\end{align}
and $d\varrho$ is from the last term of  
\cite[(3.38)]{BK}, in particular, $\varrho$ is a local term from the 
small time heat kernel asymptotics of Bismut Superconnection, 
$\theta^k$ ($k=1, 2, 3$) have universal expression in terms
 of $g, \omega, \omega'$ and $h^E$ via the Bismut superconnection. 
 Thus, from \cite[Definition 3.14, Theorem 3.16]{BK}] 
 and  \cite[(2.34)]{Ma1}, 
we know that if $i: V_1\to V$ is a complex submanifold of $V$, 
when we consider the corresponding objects for the submersion $ f_1$,
each above term is the restriction of the corresponding term
for the global submersion $f$.
Thus let  $\Delta_1, \theta_1^k$ be the corresponding terms associated 
 to the fibration $f_1: W_1\to V_1$, then we have 
 $\Delta_1=i_{\mu_n}^*\Delta$ and 
 $\theta_1^k=i_{\mu_n}^*\theta^k$ ($k=1, 2, 3$),
 $\varrho_{1}=i_{\mu_n}^*\varrho$.
 So write $\Delta^0(f,\overline{E},\omega,\omega')=\theta^2-\varrho$ 
 and $\Delta_0(f,\overline{E},\omega,\omega')
 =\theta^1+\partial\theta^3- \varrho$, 
 we are done.
\end{proof}

\subsection{Functoriality of the equivariant analytic torsion forms} \label{Ma2.s2}
Let $W, V$ and $S$ be three $\mu_n$-equivariant projective
 complex manifolds with $S=S_{\mu_n}$. Suppose that $f: W\to V$ 
 and $h: V\to S$ are two holomorphic submersions with compact fibres
 $X,Y$.  Then $h\circ f$ is also a holomorphic submersion 
 with compact fibre $Z$.

Let $\omega^W$ and $\omega^V$ be two $\mu_n$-invariant K\"{a}hler 
forms on $W$ and on $V$. As usual, $\omega^W$ and $\omega^V$ 
decide K\"{a}hler fibration structures on the morphisms $f, h$ and $g$ 
and they induce $\mu_n$-invariant hermitian metrics
associated with the K\"ahler forms
 $\omega^X, \omega^Y$ and $\omega^Z$ 
on relative tangent bundles $Tf, Th$ and $T(h\circ f)$. 
Consider the following short exact sequence of hermitian vector bundles
$$\overline{T}(f,h,h\circ f):\quad 0\to \overline{Tf}\to
 \overline{T(h\circ f)}\to f^*\overline{Th}\to 0.$$
Denote by ${\rm Td}_g(\overline{T}(f,h,h\circ f))$
 the equivariant secondary Todd form, it satisfies the differential equation
$$\frac{\bar \partial \partial}{2\pi i}{\rm Td}_g(\overline{T}(f,h,h\circ f))
={\rm Td}_g(\overline{T(h\circ f)})-f_{\mu_n}^*{\rm Td}_g
(\overline{Th}){\rm Td}_g(\overline{Tf}).$$

Now let $\overline{E}$ be a hermitian vector bundle on $W$, 
we shall assume that $E$ is $f$-acyclic and $h\circ f$-acyclic. 
Then the Leray spectral sequence
$E_2^{i,j}=R^ih_*(R^jf_*E)$ degenerates at $E_2$ 
so that $f_*E=R^0f_*(E)$ is $h$-acyclic and 
 $(h\circ f)_*E\cong h_*f_*E$. Clearly, $(h\circ f)_*E$ and $h_*f_*E$ 
 carry in general different $L^2$ metrics
 (Note that  for $\sigma\in ((h\circ f)_*E)_b$, $b\in S$,
 \begin{align} \label{Ma2.4.2} \begin{split}
&	 \|\sigma\|_{(h\circ f)_*E}^2=(2\pi)^{-\dim Z}\int_{Z_b} |\sigma|^2 
   \frac{(\omega^Z)^{\dim Z}}{(\dim Z)!} , \\
 &  \|\sigma\|_{h_*f_*E}^2=(2\pi)^{-\dim Z}\int_{Y_b} 
  \Big( \int_{X} |\sigma|^2 
 \frac{(\omega^X)^{\dim X}}{(\dim X)!}\Big) 
 \frac{(\omega^Y)^{\dim Y}}{(\dim Y)!} ,
 \end{split}\end{align}
 thus they are different in general). 
   Consider the following short exact sequence of hermitian vector bundles
$$\overline{E}(f,h,h\circ f):\quad 0\to h_*f_*\overline{E} 
\to (h\circ f)_*\overline{E} \to 0\to 0.$$
The equivariant secondary Bott-Chern form 
$\widetilde{{\rm ch}}_g(\overline{E}(f,h,h\circ f))$ satisfies the differential equation
$$\frac{\bar \partial \partial}{2\pi i}\widetilde{{\rm ch}}_g
(\overline{E}(f,h,h\circ f))
={\rm ch}_g((h\circ f)_*\overline{E})-{\rm ch}_g(h_*f_*\overline{E}).$$

\begin{thm}\label{A20}
Let notations and assumptions be as above. 
Then the following identity holds in 
$\bigoplus_{p\geq 0}A^{p,p}(S)/{(\Im{\partial}+\Im\bar\partial)}$:
\begin{align} \label{Ma2.4.3} \begin{split}
T_g(h\circ f,\omega^W,h^E)-T_g(h,\omega^V,h^{f_*E})&
-\int_{V_{\mu_n}/S}{\rm Td}_g(\overline{Th})T_g(f,\omega^W,h^E)\\
=&\widetilde{{\rm ch}}_g(\overline{E}(f,h,h\circ f))
-\int_{W_{\mu_n}/S}\widetilde{{\rm Td}}_g(\overline{T}(f,h,h\circ f))
{\rm ch}_g(\overline{E}).
\end{split}\end{align}
\end{thm}
\begin{proof}(a sketch)
This is a natural extension of \cite[Th\'{e}or\`{e}me 3.5]{Ma2} 
to the equivariant case, or the family extension of 
\cite[Theorem 3.1]{Ma1} which is an equivariant 
extension of \cite[Theorem 3.1]{BerB}.
To prove this extension, one may
follow the same approach as \cite[Sections 4-9]{Ma2}.
In fact, as a purely functional analysis argument, 
the \cite[Theorems 4.5, 4.6 and 4.7]{Ma2} can be extended 
formally to the equivariant case by introduing in the right place 
the operator $g$. The reason one can do this formal extension 
has been given in \cite[Section 5]{Ma1}. For the equivariant 
extensions of \cite[Theorems 4.8, 4.9, 4.10 and 4.11]{Ma2}, 
one can show that their proofs are local on $f^{-1}(V_{\mu_n})$ 
and certain rescaling on Clifford variables which doesn't effect
the action of $g$ can be made (cf. \cite[Section 7 b)]{Ma2}). 
Replacing the equivariant local index technique in 
\cite[Sections 7, 8, and 9]{Ma1} by its equivariant relative local index,
one gets the desired identity.

To help the readers, we will use directly the notation in 
 \cite[Section 4]{Ma2}. By the anomaly formula Theorem \ref{207},
 we only need to establish Theorem \ref{A20} for a special coupe
 of K\"ahler forms, thus we will assume that
 $\omega^{W}=\widetilde{\omega}^{W}+ f^{*} \omega^{V}$
with $\widetilde{\omega}^{W}$ a K\"ahler form on $W$. 
 
 Let $\Delta$ be the rectangular domain 
 in $\R^2$ with coordinates $(u,T)$,
  defined by the four vertices $(1, \varepsilon), (T_0, \varepsilon)$, 
 $(T_0, A), (1,A)$, following \cite[(4.7)]{Ma2}, set 
\begin{align} \label{Ma2.4.7}
\begin{split}
\theta^0_1 &= (2\pi i)^{-1/2} \int_\Delta \frac{2}{u} \frac{\partial}{
\partial b} \left\{\varphi \mbox{Tr}_s \left[ g \left[ B'_{3,u^2,T},
N_{3,u^2,T}\right] \exp(-B^2_{3,u^2,T} - b
M_{3,u^2,T})\right]\right\}_{b=0} dudT, \\
\theta^0_2& = (2\pi i)^{-1/2} \int_\Delta  \frac{2}{u}  \frac{\partial}{
\partial b} \left\{\varphi \mbox{Tr}_s \left[ g \left[ B''_{3,u^2,T},
N_{3,u^2,T}\right] \exp(-B^2_{3,u^2,T} - b
M_{3,u^2,T})\right]\right\}_{b=0} dudT,   \\
\theta^0_3 &= (2\pi i)^{-1} \int_\Delta  \frac{2}{u}  \frac{\partial}{
\partial b} \left\{\varphi \mbox{Tr}_s \left[ g 
N_{3,u^2,T} \exp(-B^2_{3,u^2,T} - b
M_{3,u^2,T})\right]\right\}_{b=0} dudT .
\end{split}
\end{align}
The only difference  comparing with  \cite[(4.7)]{Ma2}
is that in \eqref{Ma2.4.7}, we add the operator $g$ as the first term
in ${\rm Tr}_s[\cdots]$ in \cite[(4.7)]{Ma2},
i.e., replace ${\rm Tr}_s[\cdots]$ by ${\rm Tr}_s[g \cdots]$.
Note that $B_{3,u^2,T}$ is the Bismut superconnection 
assocaited with the submersion $h\circ f$ and the form
$\omega^W_T= \frac{1}{T^{2}}
\widetilde{\omega}^{W}+ f^{*} \omega^{V}$, and 
$B'_{3,u^2,T}, B''_{3,u^2,T}$ are holomorphic and 
anti-holomorphic part of $B_{3,u^2,T}$.
Moreover $N_{3,u^2,T}$ is a generalized number operator 
associated with $\omega^W_T$. 

 The boundary of $\Delta$ composes as four oriented segments 
 $\Gamma_1,\cdots,\Gamma_4$.
 Let $I^0_k$ be the integral of the one form on $\R^2$ with values
 in $\Lambda^\bullet(T^*_\R S)$ defined by 
 replacing ${\rm Tr}_s[\cdots]$ by ${\rm Tr}_s[g \cdots]$ in
 \cite[Definition 4.2]{Ma2}, then we have the $g$-analogue of
  \cite[(4.8)]{Ma2}: 
 \begin{eqnarray}
\sum^4_{k=1} I^0_k = \overline \partial \theta^0_1 - \partial  \theta^0_2 -
\overline \partial \partial  \theta^0_3 .
\end{eqnarray}
We study the terms $I^0_k$ and $\theta^0_j$ in succession
 as $A\to +\infty$, $T_0\to +\infty$, $\varepsilon\to 0$ : 
 roughly, we get 
 
 $\bullet$ the term $-T_g(h,\omega^V,h^{f_*E})$ from $I^0_1$,
 
 $\bullet$ a differential form version of
 $- \widetilde{{\rm ch}}_g(\overline{E}(f,h,h\circ f))$
 (via \cite[(1.58)]{BGS} or \cite[(4.17)]{Ma1} by replacing  
 ${\rm Tr}_s[\cdots]$ by ${\rm Tr}_s[g \cdots]$)
 from $I^0_2$, 
 
 $\bullet$ $T_g(h\circ f,\omega^W,h^E)$ from  $I^0_3$, 
 
 $\bullet$ 
$-\int_{V_{\mu_n}/S}{\rm Td}_g(\overline{Th})T_g(f,\omega^W,h^E)
 + \int_{W_{\mu_n}/S} \widetilde{{\rm Td}}_g(\overline{T}(f,h,h\circ f))
{\rm ch}_g(\overline{E})$
(here we should use the differential form version of
$\widetilde{{\rm Td}}_g(\overline{T}(f,h,h\circ f))$ from 
the term $\int_{1}^{\infty}\cdots$  in  \cite[(4.72)]{BerB} by 
replacing ${\rm Td}$ therein by ${\rm Td}_g$)
from $I^0_4$.

Let $\theta_j^3 (j=1,2,3)$ be the differential forms on $S$ obtained 
from $\theta_j^0$ by the above procedure,  then 
the difference of two sides in \eqref{Ma2.4.3}
(by using the differential form versions of
$\widetilde{{\rm ch}}_g(\overline{E}(f,h,h\circ f))$ and
$\widetilde{{\rm Td}}_g(\overline{T}(f,h,h\circ f))$ as above)  is 
\begin{align}\label{Ma2.4.10}
	\Delta+ d^{S} \Theta=\bar{\partial}\theta_1^3-\partial\theta_2^3
-\bar{\partial}\partial\theta_3^3
\end{align}
and $\theta_k^3$ ($k=1, 2, 3$) have universal expressions 
via the Bismut superconnection $B_{3,u^2,T}$, 
$\Theta$ is a combination of local terms from the small time heat kernel 
asymptotics of the Bismut superconnection for the fibration $h$
and $h\circ f$, cf. 
 \cite[(2.24), (2.27)]{Ma1} and \cite[(4.27), (4.29)]{Ma2}.
\end{proof}

Let $\theta_k^3$ ($k=1, 2, 3$) be the form
in \eqref{Ma2.4.10} associated with the couple 
$\omega^{M}=\widetilde{\omega}^{W}+ f^{*}\omega^{V}$, $\omega^{V}$.
Set
\begin{align}\label{Ma2.4.11}
\Delta^0(f, h, \omega^{W}, \omega^{V}, \overline{E})
=-\theta_{2}^3 -\Theta
\quad \text{ and }
\Delta_0(f, h, \omega^W, \omega^V, \overline{E})
=\theta_{1}^3-\partial\theta_{3}^3 -\Theta.
\end{align}
Then when we fix the differential form versions of
$\widetilde{{\rm ch}}_g(\overline{E}(f,h,h\circ f))$ and
$\widetilde{{\rm Td}}_g(\overline{T}(f,h,h\circ f))$ as above, 
\eqref{Ma2.4.11}
measure the difference of the formula \eqref{Ma2.4.3}
at the differential form level from \eqref{Ma2.4.10}: 
\begin{multline}\label{Ma2.4.12}
\Delta=\partial\Delta^0(f, h, \omega^W, \omega^V, \overline{E})
+\bar{\partial}\Delta_0(f, h, \omega^W, \omega^V, \overline{E})\\
=T_g(h\circ f,\omega^W,h^E)-T_g(h,\omega^V,h^{f_*E})
-\int_{V_{\mu_n}/S}{\rm Td}_g(\overline{Th})T_g(f,\omega^W,h^E)\\
-\widetilde{{\rm ch}}_g(\overline{E}(f,h,h\circ f))
+\int_{W_{\mu_n}/S}\widetilde{{\rm Td}}_g(\overline{T}(f,h,h\circ f))
{\rm ch}_g(\overline{E}).
\end{multline}

Let $S_1$ be a closed submanifold of $S$, and let 
$V_1=h^{-1}(S_1)\subset V$ (resp. $W_1=(h\circ f)^{-1}(S_1)\subset W$) 
be the closed submanifold of $V$ (resp. $W$) with restricted 
K\"{a}hler metric. Then $f_1: W_1\to V_1$, 
$h_1: V_1\to S_1$ and $h_1\circ f_1: W_1\to S_1$
also form a triple of equivariant holomorphic submersions 
with compact fibres. 
Denote by $j$ (resp. $i$) the natural embedding $W_1\to W$ 
(resp. $V_1\to V$) and by $\omega^{W_1}, \omega^{V_1}$ 
the induced K\"{a}hler forms $j^*\omega^{W}, i^*\omega^{V}$.  
Denote by $l$ the embedding $S_1\to S$. 
Let $\overline{E}$ be an $f$-acyclic and $h\circ f$-acyclic hermitian 
bundle on $W$.

\begin{thm}\label{A2} 
The forms $\Delta^0(f, h, \omega^{W}, \omega^{V}, \overline{E})$ and 
$\Delta_0(f, h, \omega^W, \omega^V, \overline{E})$
are functorial in the following sense that
$$l^*\Delta^0(f, h, \omega^{W}, \omega^{V}, \overline{E})
=\Delta^0(f_1, h_1, \omega^{W_1},  \omega^{V_1}, j^*\overline{E})$$
and 
$$l^*\Delta_0(f, h, \omega^{W}, \omega^{V}, \overline{E})
=\Delta_0(f_1, h_1, \omega^{W_1},  \omega^{V_1}, j^*\overline{E}).$$
\end{thm}
\begin{proof}
Note that the square of the 
Bismut superconnection is a second order fiberwise elliptic operator 
with differential form coefficients
\cite[Theorem 3.6]{Bi86} (cf. also \cite[Theorem 10.17]{BGV}),
in particular, its heat kernel along the fibers is well-defined, 
and  in \eqref{Ma2.4.7},  the terms  $\left[B'_{3,u^2,T},
N_{3,u^2,T}\right]$,   $\left[B''_{3,u^2,T},
N_{3,u^2,T}\right]$ are first oder differential operators along the fiber,
the terms $N_{3,u^2,T}$, $M_{3,u^2,T}$  are tensors,
thus we see clearly that when we consider
the corresponding objects for the submersion $h_1\circ f_1$,
each above term is the restriction of the corresponding term
for the global submersion $h\circ f$.
 
We obtain that if 
$l: S_1\hookrightarrow S$ is a complex submanifold of $S$, 
and $\theta_{k,1}^0$,  $\Theta_{1}$ are the corresponding terms
associated to the relevant fibrations,  then we have 
\begin{align} \label{Ma2.4.9}
\Theta_{1}=l^*\Theta,\quad
\theta_{k,1}^0=l^*\theta_{k}^0\qquad (k=1, 2, 3). 
\end{align}
Now we make the  procedure as $A\to +\infty$, $T_0\to +\infty$, 
$\varepsilon\to 0$, to get $\theta_{k,1}^3$, 
then from \eqref{Ma2.4.9}, we get 
$\theta_{k,1}^3
=l^*\theta_{k}^3$ ($k=1, 2, 3$). 
Combining it with \eqref{Ma2.4.11}, we get Theorem \ref{A2}.
\end{proof}

Now for a general $\omega^{W}$, as we can use the anomaly 
formula for the trip
$(h\circ f, \omega^{W}, \omega^{W}+ f^{*}\omega^{V})$,
in particular, its differential form version as in 
Section \ref{Ma2.s1}, we can still define
$\Delta^0(f, h, \omega^W, \omega^V, \overline{E})$
and $\Delta_0(f, h, \omega^W, \omega^V, \overline{E})$
such that  Theorem \ref{A2} and \eqref{Ma2.4.12} still hold,
again we need to fix a differential form version of 
$\widetilde{{\rm Td}}_g(\overline{T}(f,h,h\circ f))$.

\subsection{Immersion formula}
Let $V, W$ be two $\mu_n$-equivariant projective complex manifolds 
and let $i: W\hookrightarrow V$ be an equivariant closed immersion. 
Let $S$ be a compact complex manifold with trivial $\mu_n$-action, 
and let $f: W\rightarrow S$, $l: V\rightarrow S$ be two equivariant 
holomorphic submersions with fibers $Y,X$ such that $f=l\circ i$. 
Assume that $\overline{\eta}$ is an equivariant hermitian bundle 
on $W$ and $(\overline{\xi}., v)$ is a complex of equivariant hermitian 
bundles on $V$ which provides a resolution of $i_*\overline{\eta}$
such that the metrics on $\xi.$ satisfy the Bismut's assumption (A).
Let $\omega^W$, $\omega^V$ be two K\"{a}hler fibrations on $f$
and on $l$ respectively. We shall assume that $\omega^W$ is the 
pull-back of
$\omega^V$ so that the K\"{a}hler metric on $W$ is induced 
by the K\"{a}hler metric on $V$. Consider the following exact sequence
$$\overline{\mathcal{N}}:\quad 0\to \overline{Tf}\to
\overline{Tl}\mid_W\to \overline{N}_{X/Y}\to 0$$
where $N_{X/Y}$ is endowed with the quotient metric.
Then the equivariant secondary Todd form of $\overline{\mathcal{N}}$ 
satisfies the identity
$$\frac{\bar \partial\partial}{2\pi i}
\widetilde{{\rm Td}}_g(\overline{\mathcal{N}})
={\rm Td}_g(Tl\mid_W,h^{Tl})-{\rm Td}_g(Tf,h^{Tf})
{\rm Td}_g(\overline{N}_{X/Y}).$$

We suppose that in the resolution $\xi.$,
$\xi_j$ are all $l-$acyclic and moreover $\eta$ is $f-$acyclic. 

Let $T_g(\omega^V, h^{\xi})$ be
the equivariant analytic torsion forms associated with 
the family of  relative Dolbeault double complexes 
$(\Omega(X,\xi |_{X}), \overline{\partial}^X +v)$.
Let $h^{H(X,\xi |_{X})}$ be the
corresponding $L_2$ metric on  the hypercohomology $H(X,\xi|_X)$
of $\xi|_X$. 

Note that under our assumption,
$H(X,\xi|_X)\simeq f_{*}\eta$. And we have the following
exact sequence of hermitian vector bundles on $S$
$$\overline{\Xi}:\quad 0\to l_*(\overline{\xi}_m)\to 
l_*(\overline{\xi}_{m-1})\to\ldots\to l_*(\overline{\xi}_0)
\to \overline{H(X,\xi|_X)}\to 0.$$
We can split $\overline{\Xi}.$ into a family of short exact sequence 
of hermitian bundles from $j=1$ to $m$
$$\xymatrix{\chi_j: \quad 0 \ar[r] & \Ker d_j \ar[r] &
\overline{\Xi}_j \ar[r]^-{d_j} & \Ker d_{j-1} \ar[r] & 0}$$
such that the kernel of every map $d_{j-1}$ for $j=2,\ldots,m$
carries the metric induced by $\overline{\Xi}_j$ and 
$\Ker d_0=\overline{\Xi}_0=\overline{H(X,\xi|_X)}, 
\Ker d_m=\overline{\Xi}_{m+1}=l_*(\overline{\xi}_m)$. 
We set $\widetilde{{\rm ch}}_g(\overline{\Xi}.)
=\sum_{j=0}^{m+1}(-1)^j\widetilde{{\rm ch}}_g(\chi_j)$. 
Then it satisfies the differential equation
$$\frac{\bar\partial\partial}{2\pi i}
\widetilde{{\rm ch}}_g(\overline{\Xi}.)
={\rm ch}_g(\overline{H(X,\xi|_X)})    
-\sum_{j=0}^m (-1)^{j} {\rm ch}_g(l_*(\overline{\xi}_j)).$$

The following result is the combination of \cite[Theorems 0.1 and 0.2]{BM}
which is an equivariant  extension of \cite[Theorems 0.1 and 0.2]{Bi},
and  a family extension of \cite[Theorem 0.1]{Bi95}, 
\cite[Theorem 0.1]{BiL}, 

Let $R_g$ be the equivariant R-genus of Bismut \cite{Bi94}.

\begin{thm}\label{A30}(Immersion formula)
The following identity holds in 
$\bigoplus_{p\geq0}A^{p,p}(S)/{(\Im\partial+\Im\bar\partial)}$.
\begin{align}
&T_g(\omega^V, h^{\xi}) -T_g(\omega^W,h^\eta)
+\widetilde{{\rm ch}}_g(f_*\eta, h^{H(X,\xi |_{X})}, 
h^{f_{*}\eta}) 
=-\int_{V_{\mu_n}/S}{\rm
Td}_g(\overline{Tl})T_g(\overline{\xi}.)\nonumber\\
&\label{BM.2.3} \hspace{0.5cm}  -\int_{W_{\mu_n}/S}
 \widetilde{{\rm Td}}_g(\overline{\mathcal{N}}){{\rm
Td}_g^{-1}(\overline{N}_{X/Y})}{\rm ch}_g(\overline{\eta})
+\int_{W_{\mu_n}/S}{\rm
Td}_g(\overline{Tf})R_g(\overline{N}_{X/Y}){\rm 
ch}_g(\overline{\eta}),\\
&\label{BM.2.4} T_g(\omega^V, h^{\xi})- \sum_{i=0}^m(-1)^iT_g(\omega^V,h^{\xi_i})
- \widetilde{{\rm ch}}_g(\overline{\Xi}.)=0.
\end{align}
\end{thm}

Again to understand \eqref{BM.2.3} at the differential form level,
 i.e., without modulo ${\Im\partial+\Im\bar\partial}$,
 then we need to fix first 
 $\widetilde{{\rm ch}}_g(f_*\eta, h^{H(X,\xi |_{X})}, 
h^{f_{*}\eta}) $
 and $\widetilde{{\rm Td}}_g(\overline{\mathcal{N}})$ as differential forms, 
and $T_g(\overline{\xi}.)$ as a current. 
The natural and nice way is that we use 
\cite[(7.33)]{Bi95} to replace 
$- \widetilde{{\rm Td}}_g(\overline{\mathcal{N}}){{\rm
Td}_g^{-1}(\overline{N}_{X/Y})} + {\rm
Td}_g(\overline{Tf})R_g(\overline{N}_{X/Y})$ by the differential 
form ${\bf B}_g(\overline{\mathcal{N}})$ in \cite[(7.24)]{Bi95}. 
Then we use the current $T_g(\overline{\xi}.)$ defined  in 
 \cite[(6.30)]{Bi95} and 
 $\widetilde{{\rm ch}}_g(f_*\eta, h^{H(X,\xi |_{X})}, 
h^{f_{*}\eta})$ as the integral  $\int_1^{+\infty}$ in  \cite[(3.24)]{BM}.

Let $\Delta^0(f,l,i_*\overline{\eta},\overline{\xi}.)$ and 
$\Delta_0(f,l,i_*\overline{\eta},\overline{\xi}.)$ 
be the differential forms such that
$$\Delta:=\partial\Delta^0(f,l,i_*\overline{\eta},\overline{\xi}.)
+\bar\partial\Delta_0(f,l,i_*\overline{\eta},\overline{\xi}.)$$ 
measures the difference
\begin{align*}
&T_g(\omega^V, h^{\xi}) -T_g(\omega^W,h^\eta)
+\widetilde{{\rm ch}}_g(f_*\eta, h^{H(X,\xi |_{X})}, 
h^{f_{*}\eta})\\
&+ \int_{V_{\mu_n}/S}{\rm
Td}_g(\overline{Tl})T_g(\overline{\xi}.)
+\int_{W_{\mu_n}/S} {\bf B}_g(\overline{\mathcal{N}})
{\rm ch}_g(\overline{\eta}).
\end{align*}
We claim that $\Delta^0(f,l,i_*\overline{\eta},\overline{\xi}.)$ 
and $\Delta_0(f,l,i_*\overline{\eta},\overline{\xi}.)$ 
can be written down explicitly and they admit certain functoriality.

Let $S_1$ be a closed submanifold of $S$, and let 
$W_1=f^{-1}(S_1)\subset W$ (resp. $V_1=l^{-1}(S_1)\subset V$) 
be the closed submanifold of $W$ (resp. $V$) with restricted 
K\"{a}hler metric. Then $i_1: W_1\to V_1$, $l_1: V_1\to S_1$
and $f_1:W_1\to S_1$ also form a triple of equivariant morphisms 
such that $f_1=l_1\circ i_1$. 
Denote by $j$ the embedding $S_1\to S$. 

\begin{thm}\label{A3}
There is a natural way to write down explicitly differential forms 
$\Delta^0(f, l, {i}_*\overline{\eta}, \overline{\xi}.)$ and 
$\Delta_0(f, l, {i}_*\overline{\eta}, \overline{\xi}.)$ such that  
$\Delta:=\partial\Delta^0(f,l,i_*\overline{\eta},\overline{\xi}.)
+\bar\partial\Delta_0(f,l,i_*\overline{\eta},\overline{\xi}.)$ 
and they are functorial in the following sense.
$$j^*\Delta^0(f, l, {i}_*\overline{\eta}, \overline{\xi}.)
=\Delta^0(f_{1},l_{1},{i_{1}}_*\overline{\eta}\mid_{W_1},
\overline{\xi}.\mid_{V_1})$$ and 
$$j^*\Delta_0(f, l, {i}_*\overline{\eta}, \overline{\xi}.)
=\Delta_0(f_{1},l_{1},{i_{1}}_*\overline{\eta}\mid_{W_1},
\overline{\xi}.\mid_{V_1}).$$
\end{thm}
\begin{proof}
By the equivariant extension of 
\cite[(6.109), (6.110), (6.158), (6.170)]{Bi}
in \cite[Definition 3.4]{BM}, 
there exist universal smooth forms $\gamma^3, \delta^3$
on $S$ such that 
$$\Delta+ d^{S}\beta=\bar{\partial}\gamma^3+\partial\delta^3.$$
Again $\beta$ is a combination of local terms from the small time 
 heat kernel 
asymptotics of the Bismut superconnection for the fibration $h$
and $h\circ f$, cf. \cite[Theorem 6.4, (6.36), (6.55)]{Bi} and
 \cite[(2.24), (2.27)]{Ma1}.
More precisely, before we make  the  procedure 
as $A\to +\infty$, $T_0\to +\infty$, 
$\varepsilon\to 0$,  the forms $\gamma, \delta$ 
defined in \cite[(3.13)]{BM}  are
double integrals of certain supertrace
of the heat kernel of the square of Bismut superconnection
as in  \eqref{Ma2.4.7}. 
Note that  the square of the Bismut superconnection is 
a second order fiberwise elliptic operator with differential form 
coefficients and when we consider
the corresponding objects for the submersion $l_1$,
each above term is the restriction of the corresponding term
for the global submersion $l$, thus 
if $\Delta_1, \gamma_1^3, \delta_1^3$, $\beta_{1}$
are corresponding terms associated to the relevant fibrations 
$i_1, l_1$ and $f_1$, we have
$$\Delta_1=j^*\Delta, \gamma_1^3=j^*\gamma^3, 
\delta_1^3=j^*\delta^3, \beta_1=j^*\beta,$$
So write $\Delta^0(f, l, {i}_*\overline{\eta}, 
\overline{\xi}.)=\gamma^3-\beta$ 
and $\Delta_0(f, l, {i}_*\overline{\eta}, \overline{\xi}.)
= \delta^3-\beta$, we are done.
\end{proof}

We can do the same analysis for \eqref{BM.2.4}.

Note that we can relax our condition on $f: V\to S$ as follows:
$S$ is a (possible noncompact) complex manifold
and $f: V\to S$ is a  K\"ahler fibration in the sense of 
Bismut-Gillet-Soul\'e \cite[Definition 1.4]{BGS2}.

\hspace{5cm} \hrulefill\hspace{5.5cm}

Shun Tang

Beijing Advanced Innovation Center for Imaging Theory and Technology

Academy for Multidisciplinary Studies, Capital Normal University

School of Mathematical Sciences, Capital Normal University

West 3rd Ring North Road 105, 100048 Beijing, P. R. China

E-mail: shun.tang@outlook.com
\\

Xiaonan Ma

Universit\'{e} Paris Diderot - Paris 7

UFR de Math\'{e}matiques, Case 7012

75205 Paris Cedex 13, France

E-mail: xiaonan.ma@imj-prg.fr

\end{document}